\documentclass[a4paper,11pt]{article}

\usepackage[T1]{fontenc}
\usepackage{lmodern}

\usepackage{dsfont}

\usepackage[latin1]{inputenc}
\usepackage{amsmath}
\usepackage{amsthm}
\usepackage{amssymb}
\usepackage{mathrsfs}
\usepackage{graphicx}
\usepackage[all]{xy}
\usepackage{hyperref}

\usepackage{stmaryrd}
\usepackage{caption}

\usepackage{abstract} 

\newtheorem{thm}{Theorem}[section]
\newtheorem{cor}[thm]{Corollary}
\newtheorem{claim}[thm]{Claim}
\newtheorem{fact}[thm]{Fact}

\newtheorem{lemma}[thm]{Lemma}
\newtheorem{prop}[thm]{Proposition}

\theoremstyle{definition}
\newtheorem{definition}[thm]{Definition}
\newtheorem{ex}[thm]{Example}
\newtheorem{remark}[thm]{Remark}
\newtheorem{question}[thm]{Question}

\newcommand{\supp}{\operatorname{supp}}

\title{Embeddings into Thompson's groups from quasi-median geometry}
\date{\today}
\author{Anthony Genevois}

\begin{document}

\maketitle

\begin{abstract}
The main result of this article is that any braided (resp. annular, planar) diagram group $D$ splits as a short exact sequence $1 \to R \to D \to S \to 1$ where $R$ is a subgroup of some right-angled Artin group and $S$ a subgroup of Thompson's group $V$ (resp. $T$, $F$). As an application, we show that several braided diagram groups embeds into Thompson's group $V$, including Higman's groups $V_{n,r}$, groups of quasi-automorphisms $QV_{n,r,p}$, and generalised Houghton's groups $H_{n,p}$. 
\end{abstract}

\tableofcontents

\section{Introduction}

Historically, Thompson's groups $F$, $T$ and $V$ are important: $F$ was the first example of a torsion-free group of type $F_{\infty}$ which does not admit a finite classifying space, and $T$ and $V$ were the first examples of infinite finitely presented simple groups. Since then hundreds of articles have been written on these groups, however many problems remain open, and constructing Thompson-like groups is still an active area of research. The goal of this article is to exploit one of these families of Thompson-like groups, namely Guba and Sapir's braided diagram groups, in order to produce new constructions of embeddings into Thompson's groups. The motivation is to better understand which groups arise as subgroups of Thompson's groups, especially $V$. For background on this question, see for instance \cite{HigmanBook, RoverThesis, FreeProductsInV, CorwinThesis, ObstructionSubV, BMN}. 

A consequence of our work is that an investigation of braided diagram groups leads to new examples of groups embedding into Thompson's groups. More precisely, our main result is:

\begin{thm}\label{thm:mainsequence}
Any braided (resp. planar, annular) diagram group $D$ decomposes as a short exact sequence
$$1 \to R \to D \to S \to 1$$
for some subgroup $R$ of a right-angled Artin group and some subgroup $S$ of Thompson's group $V$ (resp. $F$, $T$). 
\end{thm}

It is worth noticing that the statement has been proved in \cite{MR2193191} for planar diagram groups using completely different methods. However, since our applications deal with braided diagram groups and Thompson's group $V$, they do not follow from Guba and Sapir's work. In several situations, the group $R$ in the previous statement turns out to be trivial, so that the diagram group $D$ embeds into the corresponding Thompson's group. For instance:

\begin{prop}
Every simple subgroup of a braided (resp. planar, annular) diagram group is isomorphic to a subgroup of Thompson's group $V$ (resp. $F$, $T$).
\end{prop}

As an application of our construction, we find several examples of subgroups of Thompson's groups. We refer to Section \ref{section:ex} for precise definitions, and to Remark \ref{remark:distortion} for information about the distortions of the embeddings.

\begin{cor}
For every integers $n \geq 2$, $r \geq 1$ and $p \geq 0$,
\begin{itemize}
	\item the generalised Thompson's group $F_{n,r}$ (resp. $T_{n,r}$) embeds into $F$ (resp. $T$);
	\item Higman's group $V_{n,r}$ embeds into $V$;
	\item the group of quasi-automorphisms $QV_{n,r,p}$ embeds into $V$;
	\item the generalised Houghton's group $H_{r,p}$ embeds into $V$. 
\end{itemize}
\end{cor}

In the opposite direction, Theorem \ref{thm:mainsequence} can be used to prove that some Thompson-like groups are quite different from diagram groups. For instance:

\begin{cor}
For every $n \geq 2$, the Brin-Thompson group $nV$ does not embed into any braided diagram group.
\end{cor}

Our strategy to prove Theorem \ref{thm:mainsequence} is the following. First, we generalise Guba and Sapir's diagram product \cite{MR1725439} by defining braided (resp. planar, annular) picture products, and we notice that there exists a family of braided (resp. planar, annular) picture products $\{ \mathscr{V}_{\alpha} \}$ (resp. $\{ \mathscr{F}_{\alpha} \}$, $ \{ \mathscr{T}_{\alpha} \}$) such that any braided (resp. planar, annular) diagram group embeds into some of these products. Next, following \cite{Qm}, we make picture products act on quasi-median graphs, and we use this action to decompose $\mathscr{V}_{\alpha}$ (resp. $\mathscr{F}_{\alpha}$, $\mathscr{T}_{\alpha}$) as a semidirect product between a right-angled Artin group and Thompson's group $V$ (resp. $F$, $T$). Finally, Theorem \ref{thm:mainsequence} follows easily.

The article is organised as follows. Section \ref{section:braideddiaggroups} is dedicated to the formalism of braided diagram groups, introduced in \cite{MR1725439} and further studied in \cite{FarleyPicture}. Next, we define picture products in Section \ref{section:pictureproducts} and we study their quasi-median geometry. In Section \ref{section:universalpictureproducts}, we show that diagram groups embed into specific picture groups. Finally, Section \ref{section:applications} is dedicated to applications of our constructions.


\section{Braided diagram groups}\label{section:braideddiaggroups}

\subsection[Thompson's groups F, T, V]{Thompson's groups $F$, $T$, $V$}\label{section:Thompson}

\noindent
In this section, our goal is to motivate the formalism of \emph{braided semigroup diagrams} studied in the next sections by introducing Thompson's groups $F$, $T$ and $V$. More information on these groups can be found in \cite{ThompsonNotes}. We begin with Thompson's group $V$. 

\begin{definition}
Consider $\mathbb{S}^1$ as the interval $[0,1]$ with the endpoints identified. Thompson's group $V$ is the group of right-continuous bijections of $\mathbb{S}^1$ which map dyadic rational numbers to dyadic rational numbers, which are differentiable exactly at finitely many dyadic rational numbers, and such that, on each maximal interval on which the function is differentiable, the function is linear with derivative a power of two.  
\end{definition}

\noindent
An element of $V$ can be naturally defined by a \emph{pair of trees} whose leaves are numbered: the two trees describe dyadic subdivisions of the domain and of the image (respectively called the domain tree and the image tree), and the numbers indicate to which subinterval of the image is sent a subinterval of the domain. See Figure \ref{figure3}. It is worth noticing that an element of $V$ is never represented by a unique pair of trees: you can always add a caret below two leaves labelled by the same number and next associate the left (resp. right) leaf of the new caret in the domain tree to the left (resp. right) leaf of the new caret in the image tree. Nevertheless, a pair of trees is said \emph{reduced} if this phenomenon does not occur, and any element of $V$ turns out to be represented by a unique reduced pair of trees. 
\begin{figure}
\begin{center}
\includegraphics[trim={0 19cm 27cm 0},clip,scale=0.45]{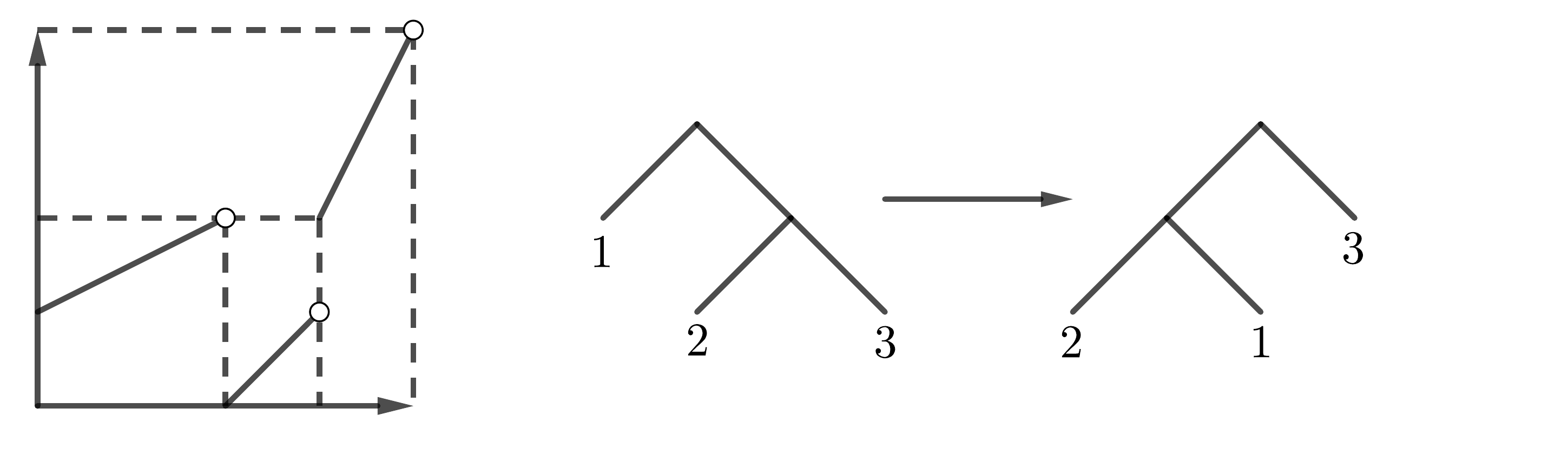}
\caption{An element of $V$ and its associated pair of trees.}
\label{figure3}
\end{center}
\end{figure}
\begin{figure}
\begin{center}
\includegraphics[trim={0 11cm 40cm 0},clip,scale=0.3]{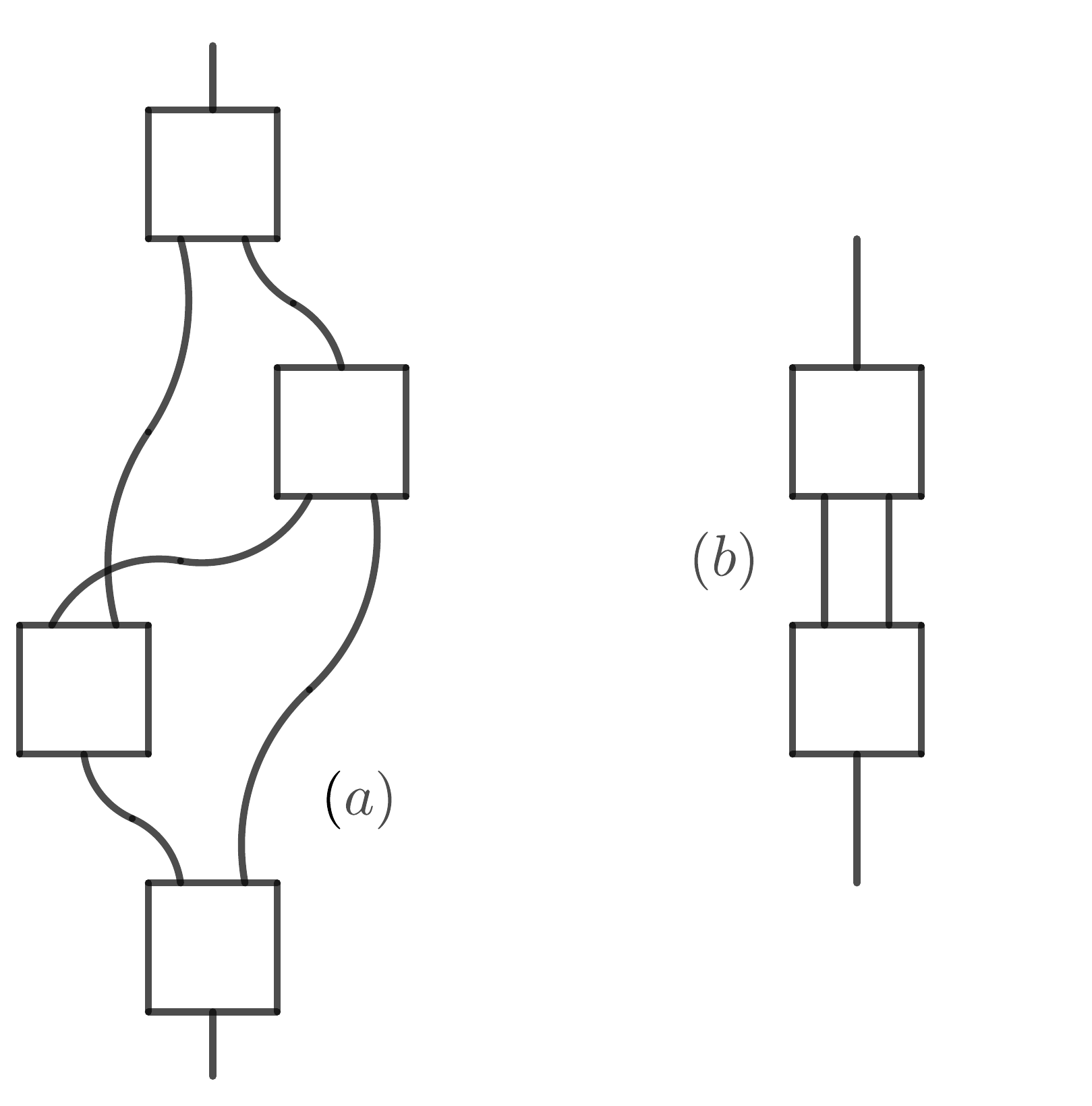}
\caption{(a) Braided diagram; (b) Dipole.}
\label{figure4}
\end{center}
\end{figure}

\medskip \noindent
Alternatively, the elements of $V$ may be represented by \emph{braided diagrams}, constructed in the following way. Consider a pair of trees representing a given element of $V$. Now, turn the image tree upside down and place it below the domain tree. Next, link any two leaves which have the same number. See Figure \ref{figure4} (a). There, the vertices of the trees are drawn as \emph{transistors} to be distinguished from the intersections of the \emph{wires} connecting the leaves of the trees. Notice that, if our initial pair of trees is not reduced, then the corresponding braided diagram contains a \emph{dipole}, ie., a subdiagram as depicted by Figure \ref{figure4} (b).

\medskip \noindent
Next, we turn to Thompson's group $T$. 

\begin{definition}
Thompson's group $T$ is the group of piecewise linear homeomorphisms of $\mathbb{S}^1$ which map dyadic rational numbers to dyadic rational numbers, which are differentiable except at finitely many dyadic rational numbers, and whose derivates on intervals of differentiability are powers of two.
\end{definition}

\noindent
Clearly, $T$ is a subgroup of $V$, so that the elements of $T$ may similarly be represented by pairs of trees. In fact, $T$ coincides with the elements of $V$ represented by pairs of trees satisfying the following condition: if the leaves of the domain tree are numbered by following the left-to-right ordering, then the leaves of the image tree are numbered by following a cyclic shift of the left-to-right ordering. As a consequence, it is sufficient to indicate where the leftmost leaf of the domain tree is sent in the image tree in order to determine the whole pair of trees. See Figure \ref{figure5}.
\begin{figure}
\begin{center}
\includegraphics[trim={0 19cm 27cm 0},clip,scale=0.45]{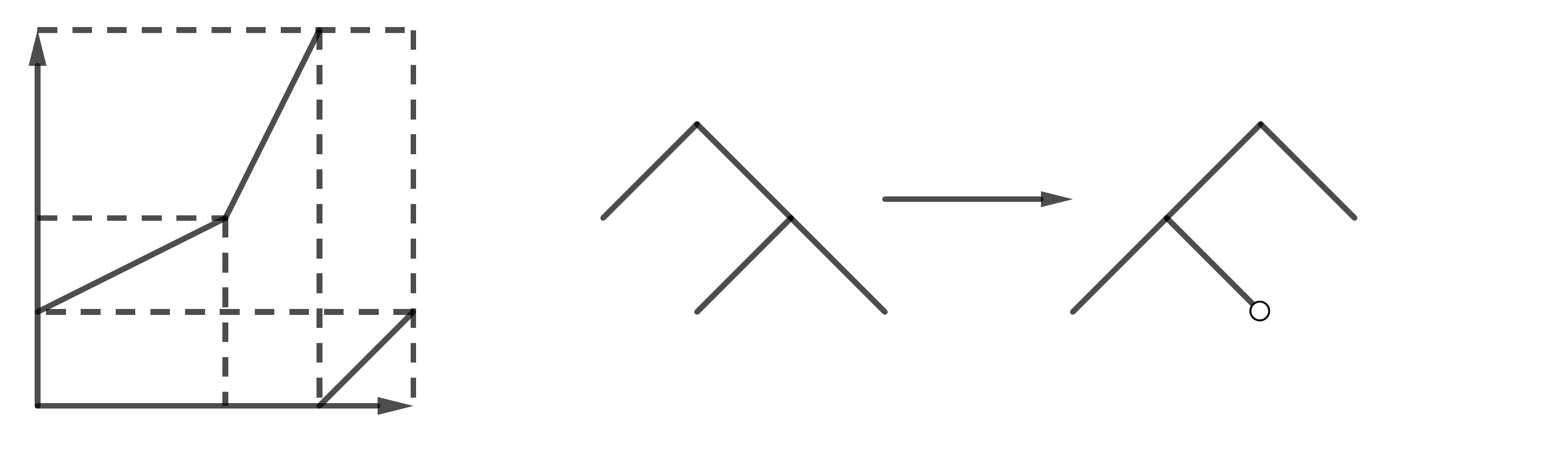}
\caption{An element of $T$ and its associated pair of trees.}
\label{figure5}
\end{center}
\end{figure}
\begin{figure}
\begin{center}
\includegraphics[trim={0 0cm 29cm 0},clip,scale=0.25]{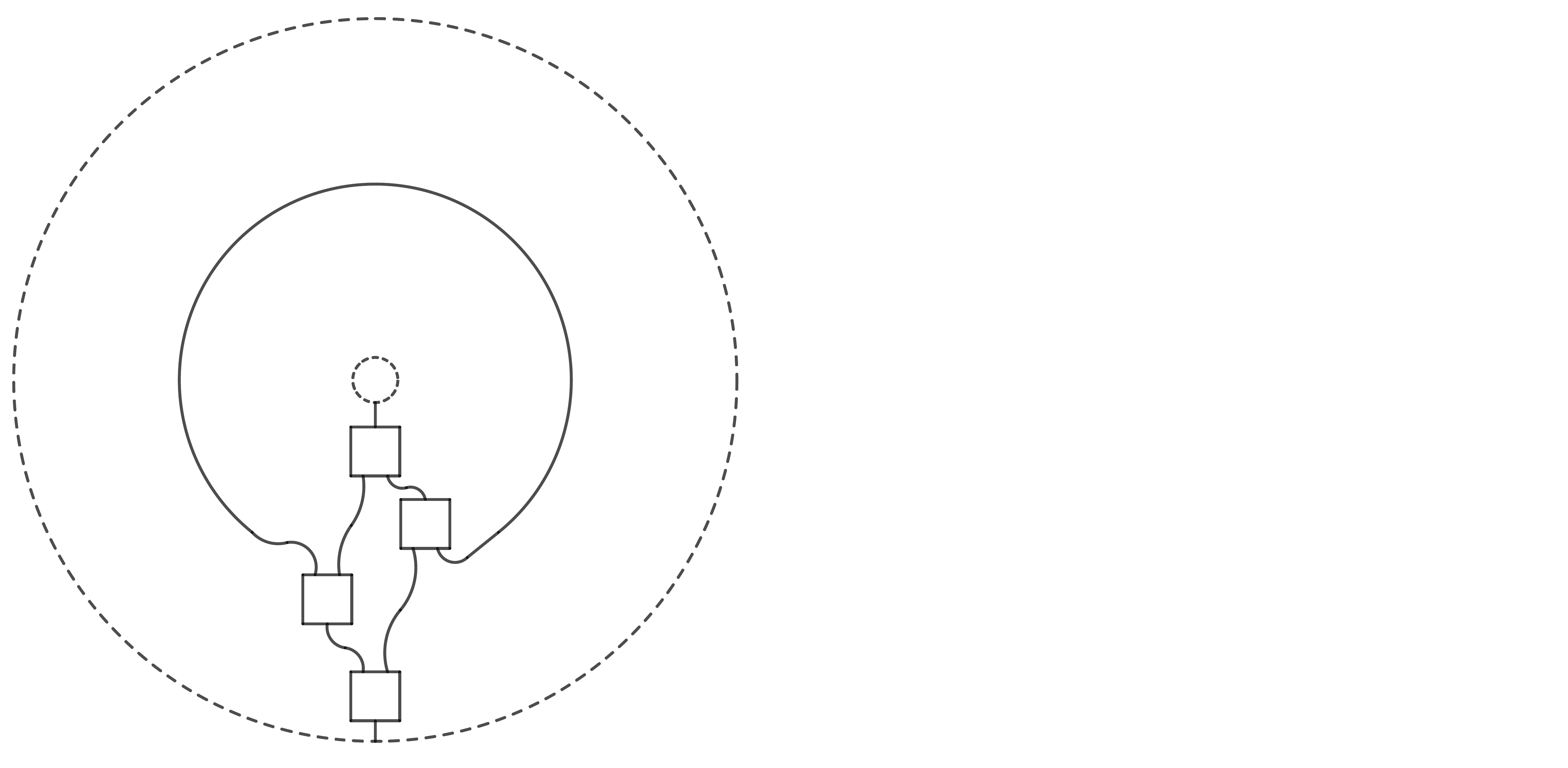}
\caption{Annular diagram.}
\label{figure6}
\end{center}
\end{figure}

\medskip \noindent
The braided diagram associated to a pair of trees representing an element of $T$ satisfies the following property: it embeds into an annulus, ie., it can be drawn on an annulus so that its wires are pairwise disjoint. For instance, Figure \ref{figure6}. Such a diagram is called \emph{annular}. 

\medskip \noindent
Finally, let us introduce Thompson's group $F$. 

\begin{definition}
Thompson's group $F$ is the group of piecewise linear homeomorphisms of $[0,1]$ which are differentiable except at finitely many dyadic rational numbers, and such that, on intervals of differentiability, the derivatives are powers of two. 
\end{definition}

\noindent
By looking at the restrictions of the elements of $F$ to $[0,1)$, Thompson's group $F$ naturally embeds into $V$, so that the elements of $F$ may also be represented by pairs of trees. In fact, $F$ coincides with the elements of $V$ represented by pairs of trees whose sets of leaves are both numbered from left to right. Therefore, it is not necessary to number the leaves of the pairs of trees. See Figure \ref{figure7}. A consequence of this description is that $F$ is naturally a subgroup of $T$, so that $F \subset T \subset V$. 
\begin{figure}
\begin{center}
\includegraphics[trim={0 19cm 27cm 0},clip,scale=0.45]{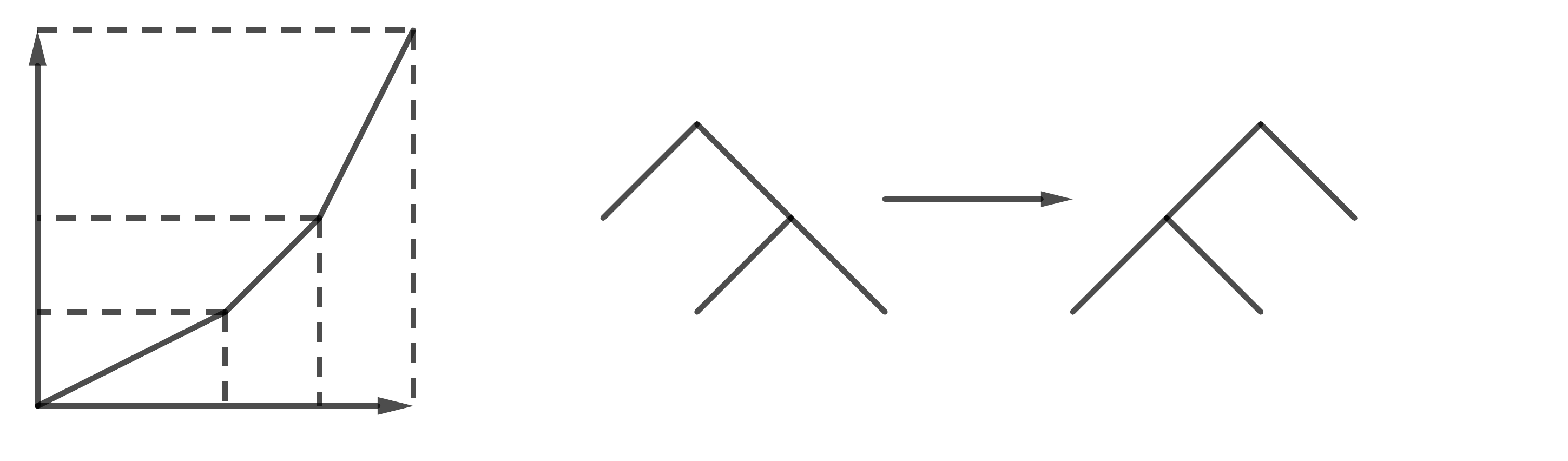}
\caption{An element of $F$ and its associated pair of trees.}
\label{figure7}
\end{center}
\end{figure}

\medskip \noindent
The braided diagram associated to a pair of trees representing an element of $F$ satisfies the following property: it embeds into the plane, ie., it can be drawn on the plane so that its wires are pairwise disjoint. For instance, Figure \ref{figure8}. Such a diagram is called \emph{planar}. 
\begin{figure}
\begin{center}
\includegraphics[trim={0 11cm 50cm 0},clip,scale=0.3]{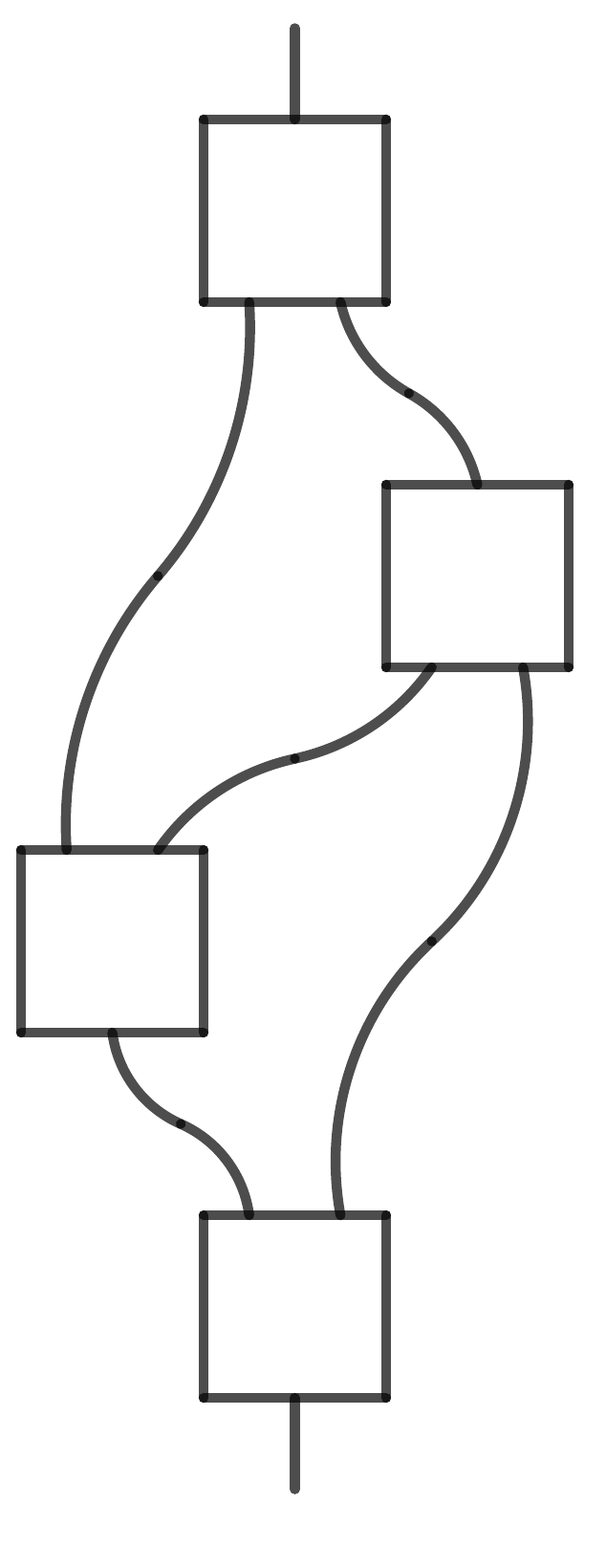}
\caption{Planar diagram.}
\label{figure8}
\end{center}
\end{figure}

\subsection{Braided semigroup diagrams}\label{section:diagrams}

\noindent
In this section, we introduce the main definitions related to \emph{braided semigroup diagrams}, firstly introduced in \cite{MR1396957} and further studied in \cite{FarleyPicture}. We refer to Example \ref{ex:diagram} below to get a good picture to keep in mind when reading the definition (notice, however, that on our pictures the frames are not drawn). 

\medskip \noindent
We begin by defining the fundamental pieces which we will glue together to construct pictures.
\begin{itemize}
	\item A \emph{wire} is a homeomorphic copy of $[0,1]$. The point $0$ is the \emph{bottom} of the wire, and the point $1$ its \emph{top}.
	\item A \emph{transistor} is a homeomorphic copy of $[0,1]^2$. One says that $[0,1] \times \{1 \}$ (resp. $[0,1] \times \{ 0 \}$, $\{0\} \times [0,1]$, $\{1 \} \times [0,1]$) is the \emph{top} (resp. \emph{bottom}, \emph{left}, \emph{right}) \emph{side} of the transistor. Its top and bottom sides are naturally endowed with left-to-right orderings.
	\item A \emph{frame} is a homeomorphic copy of $\partial [0,1]^2$. It has \emph{top}, \emph{bottom}, \emph{left} and \emph{right sides}, just as a transistor does, and its top and bottom sides are naturally endowed with left-to-right orderings.
\end{itemize}
Fix a semigroup presentation $\mathcal{P}= \langle \Sigma \mid \mathcal{R} \rangle$. In all our paper, we follow the convention that, if $u=v$ is a relation of $\mathcal{R}$, then $v=u$ does not belong to $\mathcal{R}$; as a consequence, $\mathcal{R}$ does not contain relations of the form $u=u$. A \emph{braided diagram over $\mathcal{P}$} is a labelled oriented quotient space obtained from
\begin{itemize}
	\item a finite non-empty collection $W(\Delta)$ of wires;
	\item a labelling function $\ell : W(\Delta) \to \Sigma$;
	\item a finite (possibly empty) collection $T(\Delta)$ of transistors;
	\item and a frame,
\end{itemize}
which satisfies the following conditions:
\begin{itemize}
	\item each endpoint of each wire is attached either to a transistor or to the frame;
	\item the bottom of a wire is attached either to the top of a transistor or to the bottom of the frame;
	\item the top of a wire is attached either to the bottom of a transistor or the top of the frame;
	\item the images of two wires in the quotient must be disjoint;
	\item if $\mathrm{top}(T)$ (resp. $\mathrm{bot}(T)$) denotes the word obtained by reading from left to right the labels of the wires which are connected to the top side (resp. the bottom side) of a given transistor $T$, then either $\mathrm{top}(T)= \mathrm{bot}(T)$ or $\mathrm{bot}(T)= \mathrm{top}(T)$ belongs to $\mathcal{R}$;
	\item if, given two transistors $T_1$ and $T_2$, we write $T_1 \prec T_2$ when there exists a wire whose bottom contact is a point on the top side of $T_1$ and whose top contact is a point on the bottom side of $T_2$, and if we denote by $<$ the transitive closure of $\prec$, then $<$ is a strict partial order on the set of the transistors. 
\end{itemize}
Two braided diagrams are \emph{equivalent} if there exists a homeomorphism between them which preserves the labellings of wires and all the orientations (left-right and top-bottom) on all the transistors and on the frame. From now on, every braided diagram will be considered up to equivalence. This allows us to represent a diagram in the plane so that the frame and the transistors are straight rectangles such that their left-right and top-bottom orientations coincide with a fixed orientation of the plane, and so that wires are vertically monotone. 

\medskip \noindent
Given a braided diagram $\Delta$, one defines its \emph{top label} (resp. its \emph{bottom label}), denoted by $\mathrm{top}(\Delta)$ (resp. $\mathrm{bot}(\Delta)$), as the word obtained by reading from left to right the labels of the wires connected to the top (resp. the bottom) of the frame. A \emph{braided $(u,v)$-diagram} is a braided diagram whose top label is $u$ and whose bottom label is $v$; a \emph{braided $(u, \ast)$-diagram} is a braided diagram whose top label is $u$. 

\medskip \noindent
Fixing two braided diagrams $\Delta_1$ and $\Delta_2$ satisfying $\mathrm{top}(\Delta_2)= \mathrm{bot}(\Delta_1)$, one can define their \emph{concatenation} $\Delta_1 \circ \Delta_2$ by gluing bottom endpoints of the wires of $\Delta_1$ which are connected to the bottom side of the frame to the top endpoints of the wires of $\Delta_2$ which are connected to the top side of the frame, following the left-to-right ordering, and by identifying, and next removing, the bottom side of the frame of $\Delta_1$ with the top side of the frame of $\Delta_2$. Loosely speaking, we ``glue'' $\Delta_2$ below $\Delta_1$.

\medskip \noindent
Given a braided diagram $\Delta$, a \emph{dipole} in $\Delta$ is the data of two transistors $T_1,T_2$ satisfying $T_1 \prec T_2$ such that, if $w_1, \ldots, w_n$ denotes the wires connected to the top side of $T_1$, listed from left to right:
\begin{itemize}
	\item the top endpoints of the $w_i$'s are connected to the bottom of $T_2$ with the same left-to-right order, and no other wires are attached to the bottom of $T_2$;
	\item the top label of $T_2$ is the same as the bottom label of $T_1$. 
\end{itemize}
Given such a dipole, one may \emph{reduce} it by removing the transistors $T_1, T_2$ and the wires $w_1, \ldots, w_n$, and connecting the top endpoints of the wires which are connected to the top side of $T_2$ with the bottom endpoints of the wires which are connected to the bottom side of $T_1$ (preserving the left-to-right orderings). 

\medskip \noindent
A braided diagram without dipoles is \emph{reduced}. Clearly, any braided diagram can be transformed into a reduced one by reducing its dipoles, and according to \cite[Lemma 2.2]{FarleyPicture}, the reduced diagram we get does not depend on the order we choose to reduce its dipoles. We refer to this reduced diagram as the \emph{reduction} of the initial diagram. Two diagrams are the same \emph{modulo dipoles} if they have the same reduction. 

\begin{definition}
For every $w \in \Sigma^+$, the \emph{braided diagram group} $D_b(\mathcal{P},w)$ is the set of all the braided $(w,w)$-braided semigroup diagrams over $\mathcal{P}$, modulo dipoles, endowed with the concatenation. 
\end{definition}

\noindent
This is indeed a group according to \cite{FarleyPicture}.

\begin{ex}\label{ex:diagram}
Consider the semigroup presentation 
$$\mathcal{P}= \langle a,b,c \mid ab=ba, ac=ca, bc=cb \rangle.$$
Figure \ref{figure9} shows the concatenation of two braided diagrams over $\mathcal{P}$, and the reduction of the resulting braided diagram. 
\begin{figure}
\begin{center}
\includegraphics[trim={0 11cm 20cm 0},clip,scale=0.35]{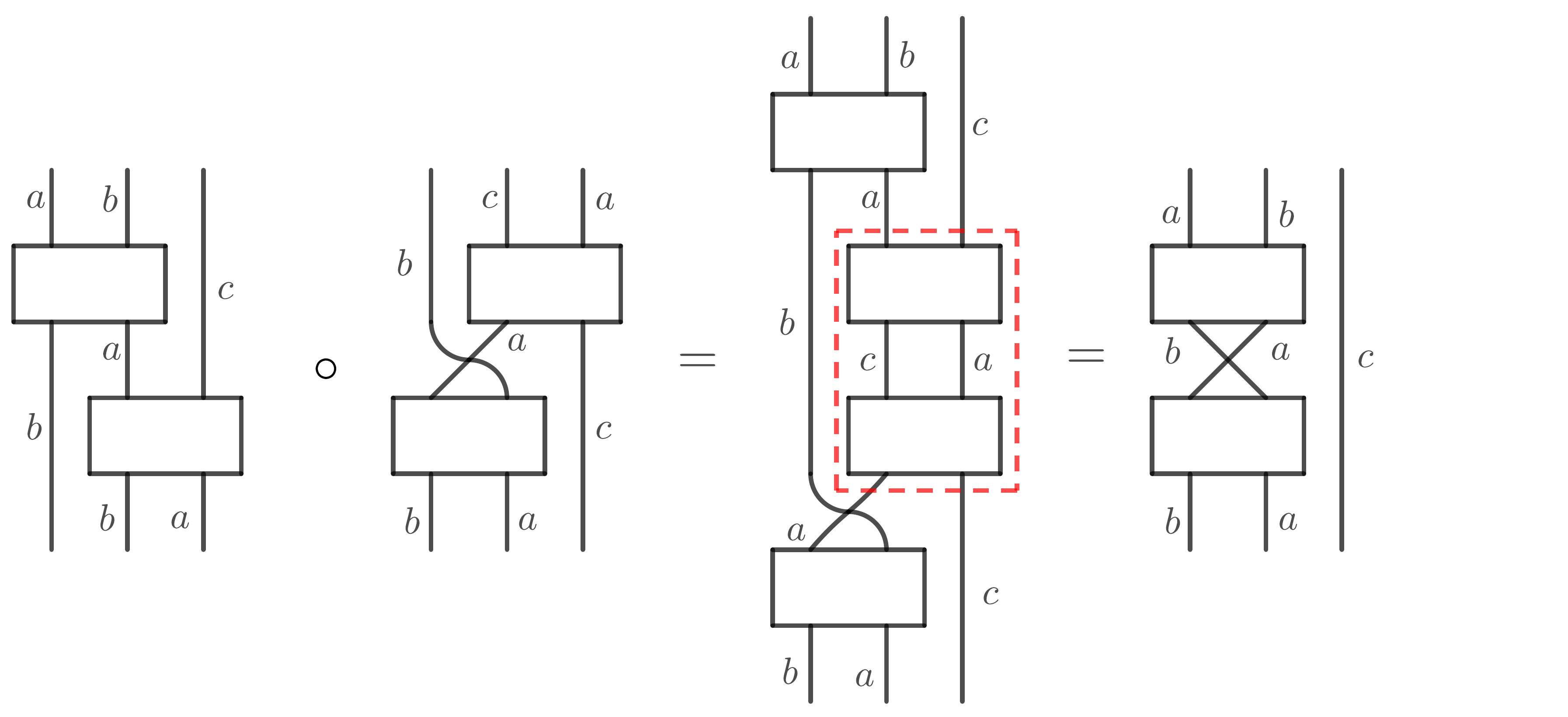}
\caption{}
\label{figure9}
\end{center}
\end{figure}
\end{ex}

\begin{ex}\label{ex:V}
As suggested by the discussion of the previous section, the braided diagram group $D_b(\mathcal{P},x)$ associated to the semigroup presentation $\mathcal{P} = \langle x \mid x=x^2 \rangle$ is isomorphic to Thompson's group $V$. See \cite[Example 16.6]{MR1396957} for a proof. 
\end{ex}

\subsection{Planar and annular diagram groups}

\noindent
In this section, we fix a semigroup presentation $\mathcal{P}= \langle \Sigma \mid \mathcal{R} \rangle$ and a baseword $w \in \Sigma^+$. Below, we define two variations of braided semigroup diagrams. Our first variation is the main topic of \cite{MR1396957}.

\begin{definition}
A braided diagram over $\mathcal{P}$ is \emph{planar} if there exists an embedding $\Delta \to \mathbb{R}^2$ which preserves the left-to-right orderings and the top-bottom orientations on the transistors and the frame. The \emph{planar diagram group} $D(\mathcal{P},w)$ is the subgroup of $D_b(\mathcal{P},w)$ consisting of all planar diagrams.
\end{definition}

\begin{ex}\label{ex:F}
As suggested by the discussion in Section \ref{section:Thompson}, the planar diagram group $D(\mathcal{P},x)$ associated to the semigroup presentation $\mathcal{P} = \langle x \mid x=x^2 \rangle$ is isomorphic to Thompson's group $F$. See \cite[Example 6.4]{MR1396957} for a proof. 
\end{ex}

\noindent
Our second variation was also introduced in \cite{MR1396957}, and further studied in \cite{FarleyPicture}. Keep in mind Figure \ref{figure6} when reading the definition.

\begin{definition}
A braided diagram $\Delta$ is \emph{annular} if it can be embedded into an annulus by preserving the left-to-right orderings and the top-bottom orientations on the transistors and the frame. More precisely, suppose that we replace the frame of $\Delta$ with a pair of disjoint circles, both endowed with the counterclockwise orientation in place of the previous left-to-right orderings of the top and bottom sides of the frame; we also fix a basepoint on each circles (which will be disjoint from the wires). Transistors and wires are defined as before, and their attaching maps are subject to the same conditions as before, where the inner (resp. outer) circle of the frame plays the role of the top (resp. bottom) side of the frame. The resulting diagram is \emph{annular} if it embeds into the plane by preserving the left-to-right orderings on the transistors. The \emph{annular diagram group} $D_a(\mathcal{P},w)$ is the set of annular $(w,w)$-diagrams over $\mathcal{P}$, modulo dipoles, endowed with the concatenation.
\end{definition}

\noindent
All the definitions introduced in the previous section naturally generalise to planar and annular semigroup diagrams. Nevertheless, we mention that, when concatenating two annular diagrams, ie., when identifying the top circle of the second diagram with the bottom circle of the second one, we have to match the basepoints on the different circles. 

\begin{ex}\label{ex:T}
As suggested by the discussion in Section \ref{section:Thompson}, the annular diagram group $D_a(\mathcal{P},x)$ associated to the semigroup presentation $\mathcal{P} = \langle x \mid x=x^2 \rangle$ is isomorphic to Thompson's group $T$. See \cite[Example 16.5]{MR1396957} for a proof. 
\end{ex}

\noindent
It is worth noticing that a planar diagram is an annular diagram as well, so that, in the same way that $F \subset T \subset V$, one has
$$D(\mathcal{P},w) \subset D_a(\mathcal{P},w) \subset D_b(\mathcal{P},w)$$
for every semigroup presentation $\mathcal{P}= \langle \Sigma \mid \mathcal{R} \rangle$ and every baseword $w \in \Sigma^+$.

\subsection{Examples of braided diagram groups}\label{section:ex}

\noindent
In this section, we mention the examples of braided diagram groups which we will consider in Section \ref{section:applications}. In fact, only few explicit braided diagram groups are known, and it would be an interesting problem to investigate this class of groups. 

\medskip \noindent
Our first family of examples comes from a natural generalisation of Thompson's groups $F \subset T \subset V$. 

\begin{ex}\label{ex:ThompsonsVariations}
Let $n \geq 2$ and $r \geq 1$ be two integers. As noticed in \cite[Section 16]{MR1396957}, the braided diagram group $D_b(\mathcal{P}_n,x^r)$, where $\mathcal{P}_n$ denotes the semigroup presentation $\langle x \mid x=x^n \rangle$, is isomorphic to Higman's group $V_{n,r}$ introduced in \cite{HigmanBook}. It is also clear that the planar diagram group $D(\mathcal{P}_n,x^r)$ and the annular diagram group $D_a(\mathcal{P}_n,x^r)$ are respectively isomorphic to the groups $F_{n,r}$  and $T_{n,r}$ introduced in \cite[Section 4]{BrownFiniteness}.
\end{ex}

\noindent
Our second family of examples is a variation of Thompson's groups, obtained by mimicking their actions on rooted trees with \emph{quasi-automorphisms}.

\begin{ex}
Let $n \geq 2$, $r \geq1$, $p \geq 0$ be three integers. We denote by $QV_{n,r,p}$ the \emph{quasi-automorphism group} of the union $\mathcal{T}=\mathcal{T}_{n,r,p}$ of $r$ infinite rooted $n$-regular trees with $p$ isolated vertices (that is, $p$ rooted trees without children), ie., the group of bijections $\mathcal{T}^{(0)} \to \mathcal{T}^{(0)}$ between the vertices of $\mathcal{T}$ which preserve adjacency and the left-to-right ordering of the children of a given vertex for all but finitely many vertices. In particular, $QV_{2,0,0}$ and $QV_{2,0,1}$ are respectively the groups $QV$ and $\bar{Q}V$ studied in \cite{QuasiAuto}. As noticed in \cite[Section 5]{FarleyQuasiAuto}, $QV_{n,r,p}$ is isomorphic to the braided diagram group $D_b(\mathcal{P}_n,x^ra^p)$ where $\mathcal{P}_n$ denotes the semigroup presentation $\langle x,a \mid x=x^na \rangle$.

\medskip \noindent
It is possible to introduce $QF_{n,r,p}$ and $QT_{n,r,p}$ by analogy with $F$ and $T$, but these groups turn out to be different from $D(\mathcal{P}_n,x^ra^p)$ and $D_a(\mathcal{P}_n,x^ra^p)$. Indeed, planar diagram groups are torsion-free and annular diagram groups do not contain non-cyclic finite subgroups whereas any finite group embeds into $QF_{n,r,p}$ and $QT_{n,r,p}$. However, these groups can naturally be thought of as kinds of braided diagrams; see \cite{FarleyQuasiAuto} for more information. 
\end{ex}

\noindent
Our third and last family of examples are Houghton's groups, introduced in \cite{Houghton}.

\begin{ex}
Let $n \geq 1$ be an integer. Let $R_n$ denote the graph which is the disjoint union of $n$ rays $\mathbb{N}= \{0,1,2, \ldots \}$. The \emph{$n$th Houghton's group} $H_n$ is the group of bijections $R_n^{(0)} \to R_n^{(0)}$ preserving adjacency for all but finitely many vertices. According to \cite[Example 4.3]{FarleyHughes}, $H_n$ is isomorphic to the braided diagram group $D_b(\mathcal{P}_n,r)$ where $\mathcal{P}_n$ denotes the semigroup presentation 
$$\langle a,r,x_1, \ldots, x_n \mid r=x_1 \cdots x_n, x_1=ax_1, x_2=ax_2, \ldots, x_n=ax_n \rangle.$$
Similarly, for every $p \geq 0$, the braided diagram group $D_b(\mathcal{P}_n,ra^p)$ coincides with the group $H_{n,p}$ of bijections $R_{n,p}^{(0)} \to R_{n,p}^{(0)}$ preserving adjacency for all but finitely many vertices, where $R_{n,p}$ denotes the union of $R_n$ with $p$ isolated vertices. 
\end{ex}

\section{Picture products}\label{section:pictureproducts}

\subsection{Generalized braided diagrams}\label{section:diagramsPG}

\noindent
In this section, we introduce a variation of braided semigroup diagrams in the spirit of \cite[Section 10.1]{Qm}, the goal being to define a kind of product of groups similar to \emph{diagram products} \cite{MR1725439}, we call \emph{picture products}. The basic idea is to label the wires of our diagrams by both letters and group elements (coming from the factors of our product), and next to adapt the definitions of dipoles and concatenation to this context. Example \ref{ex:diagramsPG} and Figure \ref{figure10} below will illustrate the corresponding definitions which we give now. 

\medskip \noindent
Let $\mathcal{P}= \langle \Sigma \mid \mathcal{R} \rangle$ be a semigroup presentation and $\mathcal{G}= \{ G_s \mid s \in \Sigma \}$ a collection of groups indexed by the alphabet $\Sigma$. We set a new alphabet
$$\Sigma(\mathcal{G})= \{ (s,g) \mid s \in \Sigma, \ g \in G_s \}$$
and a new set of relations $\mathcal{R}(\mathcal{G})$ containing
$$ (u_1,g_1) \cdots (u_n,g_n) = (v_1,h_1) \cdots (v_m,h_m) $$
for every $u_1 \cdots u_n=v_1 \cdots v_m \in \mathcal{R}$ and $g_1 \in G_{u_1}, \ldots, g_n \in G_{u_n}$, $h_1 \in G_{v_1}, \ldots, h_n \in G_{u_m}$. We get a new semigroup presentation $\mathcal{P}(\mathcal{G})= \langle \Sigma(\mathcal{G}) \mid \mathcal{R} (\mathcal{G}) \rangle$. A \emph{braided diagram over $(\mathcal{P}, \mathcal{G})$} is a braided semigroup presentation over $\mathcal{P}(\mathcal{G})$. If $\Delta$ is such a diagram, we denote by $\mathrm{top}^-(\Delta)$ (resp. $\mathrm{bot}^-(\Delta)$) the image of $\mathrm{top}(\Delta)$ (resp. $\mathrm{top}(\Delta)$) under the natural projection $\Sigma(\mathcal{G})^+ \to \Sigma^+$ (which ``forgets'' the second coordinate). If $u,v \in \Sigma^+$ are words, a \emph{$(u,v)$-diagram over $(\mathcal{P}, \mathcal{G})$} is a diagram $\Delta$ satisfying $\mathrm{top}^-(\Delta)=u$ and $\mathrm{bot}^-(\Delta)=v$; we say that $\Delta$ is a \emph{$(u,\ast)$-diagram over $(\mathcal{P}, \mathcal{G})$} if we do not want to mention $v$. 

\medskip \noindent
All the vocabulary introduced in Section \ref{section:diagrams} applies to braided diagrams over $(\mathcal{P}, \mathcal{G})$ thought of as braided semigroup diagrams over $\mathcal{P}(\mathcal{G})$, except the concatenation and the dipoles which we define now. 

\medskip \noindent
If $\Delta_1$ and $\Delta_2$ are two braided diagrams over $(\mathcal{P}, \mathcal{G})$ satisfying $\mathrm{top}^-(\Delta_2)= \mathrm{bot}^-(\Delta_1)$, we define the \emph{concatenation} of $\Delta_1$ and $\Delta_2$ as the braided diagram over $(\mathcal{P}, \mathcal{G})$ obtained in the following way. Write $\mathrm{bot}(\Delta_1) = (w_1,g_1) \cdots  (w_n,g_n)$ and $\mathrm{top}(\Delta_2)= (w_1,h_1) \cdots (w_n,h_n)$ for some $w_1, \ldots, w_n \in \Sigma$ and $g_1, \ldots, g_n,h_1, \ldots, h_n \in \bigsqcup\limits_{G \in \mathcal{G}} G$. Now, 
\begin{itemize}
	\item we glue the top endpoints of the wires of $\Delta_2$ connected to the top side of the frame to the bottom endpoints of the wires of $\Delta_1$ connected to the bottom side of the frame (respecting the left-to-right ordering);
	\item we label these wires from left to right by $(w_1,g_1h_1), \ldots, (w_n,g_nh_n)$;
	\item we identify, and then remove, the bottom side of the frame of $\Delta_1$ and the top side of the frame of $\Delta_2$.
\end{itemize}
The braided diagram we get is the \emph{concatenation} $\Delta_1 \circ \Delta_2$.

\medskip \noindent
Given a braided diagram $\Delta$ over $(\mathcal{P}, \mathcal{G})$, a \emph{dipole} in $\Delta$ is the data of two transistors $T_1,T_2$ satisfying $T_1 \prec T_2$ such that, if $w_1, \ldots, w_n$ denote the wires connected to the top side of $T_1$, listed from left to right:
\begin{itemize}
	\item the top endpoints of the $w_i$'s are connected to the bottom of $T_2$ with the same left-to-right order, and no other wires are attached to the bottom of $T_2$;
	\item the labels $\mathrm{top}^-(T_2)$ and $\mathrm{bot}^-(T_1)$ are the same;
	\item the wires $w_1, \ldots, w_n$ are labeled by letters of $\Sigma(\mathcal{G})$ with trivial second coordinates.
\end{itemize}
Given such a dipole, one may \emph{reduce} it by 
\begin{itemize}
	\item removing the transistors $T_1,T_2$ and the wires $w_1, \ldots, w_n$;
	\item connecting the top endpoints of the wires $a_1, \ldots, a_m$ (from left to right) which are connected to the top side of $T_2$ with the bottom endpoints of the wires $b_1, \ldots, b_m$ (from left to right) which are connected to the bottom side of $T_1$ (preserving the left-to-right orderings);
	\item and labelling the new wires by $( \ell_1, g_1h_1), \ldots, (\ell_m, g_mh_m)$ from left to right, if $a_i$ is labelled by $(\ell_i,h_i)$ and $b_i$ by $(\ell_i,g_i)$ for every $1 \leq i \leq m$.
\end{itemize}
A braided diagram over $(\mathcal{P}, \mathcal{G})$ which does not contain any dipole is \emph{reduced}. Of course, any braided diagram can be transformed into a reduced one by reducing its dipoles, and the same argument as \cite[Lemma 2.2]{FarleyPicture} shows that the reduced diagram we get does not depend on the order we choose to reduce its dipoles. We refer to this reduced diagram as the \emph{reduction} of the initial diagram. Two diagrams are the same \emph{modulo dipoles} if they have the same reduction. 

\begin{definition}
For every $w \in \Sigma^+$, the \emph{braided picture product} $D_b(\mathcal{P}, \mathcal{G},w)$ is the set of all the braided $(w,w)$-diagrams $\Delta$ over $(\mathcal{P},\mathcal{G})$, modulo dipoles, endowed with the concatenation. 
\end{definition}

\noindent
This is indeed a group, for the same reason that a braided diagram group turns out to be a group.

\begin{ex}\label{ex:diagramsPG}
Consider the semigroup presentation $\mathcal{P}= \langle a,b,p \mid a=ap, b=pb \rangle$ and the collection of groups $\mathcal{G}= \{ G_a=G_b=G_p= \mathbb{Z} \}$. Figure \ref{figure10} shows the concatenation of two braided diagrams over $(\mathcal{P}, \mathcal{G})$, and the reduction of the resulting diagram. 
\begin{figure}
\begin{center}
\includegraphics[trim={0 4cm 15cm 0},clip,scale=0.3]{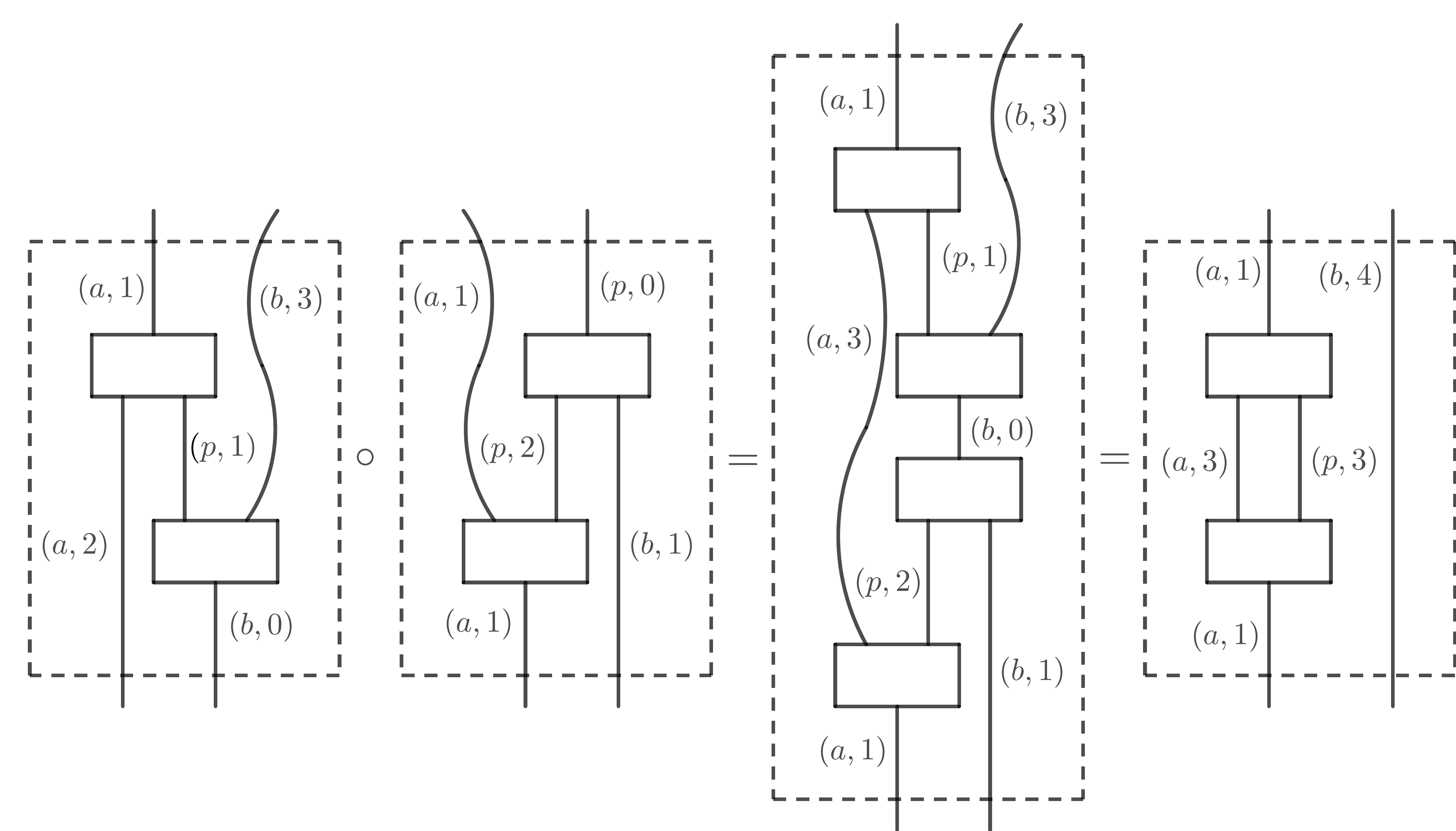}
\caption{Concatenation and reduction of two braided diagrams.}
\label{figure10}
\end{center}
\end{figure}
\end{ex}

\noindent
It is worth noticing that any braided diagram over $(\mathcal{P}, \mathcal{G})$ decomposes as a concatenation of ``elementary diagrams'', which we now define.

\begin{definition}
A \emph{permutation diagram} is a diagram which does not contain any transistor and whose wires are labelled by letters of $\Sigma(\mathcal{G})$ with trivial second coordinates. A \emph{transistor diagram} is a planar diagram which contains a single transistor and whose wires are labelled by letters of $\Sigma(\mathcal{G})$ with trivial second coordinates. A \emph{linear diagram} is a planar diagram which does not contain any transistor and which contains a single wire labelled by a letter of $\Sigma(\mathcal{G})$ with a non-trivial second coordinate. Transistor and linear diagrams are \emph{unitary diagrams}. 
\end{definition}

\noindent
We defined the concatenation $\circ$ on the set of braided diagrams. For convenience, we introduce the following notation:

\begin{definition}
Let $\Delta_1,\Delta_2$ be two braided diagrams over $(\mathcal{P}, \mathcal{G})$. The reduction of the concatenation $\Delta_1 \circ \Delta_2$ is denoted by $\Delta_1 \cdot \Delta_2$. 
\end{definition}

\noindent
Also, we will need another operation between diagrams, a \emph{sum}.

\begin{definition}
Let $\Delta_1,\Delta_2$ be two braided diagrams over $(\mathcal{P}, \mathcal{G})$. The sum $\Delta_1+ \Delta_2$ is the braided diagrams over $(\mathcal{P}, \mathcal{G})$ obtained by gluing, and next removing, the right side of the frame of $\Delta_1$ and the left side of the frame of $\Delta_2$. See Figure \ref{figure11}.
\end{definition}
\begin{figure}
\begin{center}
\includegraphics[trim={0 13cm 15cm 0},clip,scale=0.35]{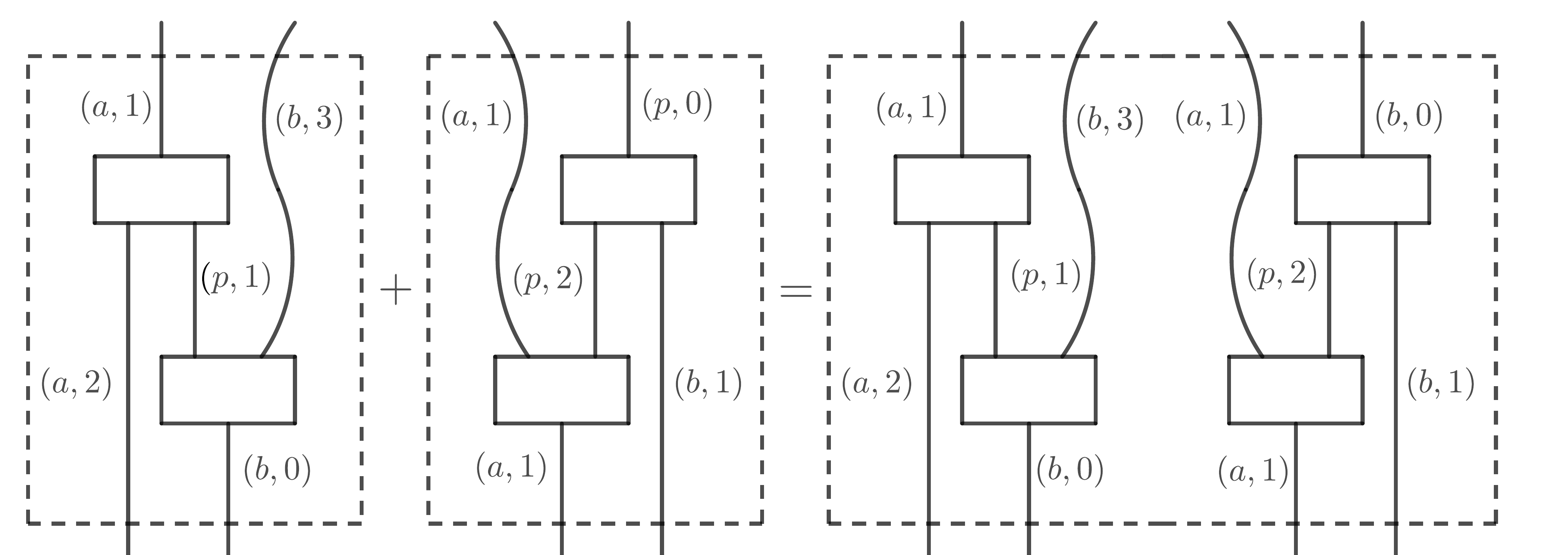}
\caption{Sum of two braided diagrams.}
\label{figure11}
\end{center}
\end{figure}

\noindent
For convenience, we introduce the following notation. If $u_1, \ldots, u_n \in \Sigma(\mathcal{G})$ are letters, we denote by $\epsilon(u_1 \cdots u_n)$ the planar diagram which does not contain any transistor and whose wires are labelled (from left to right) by $u_1, \ldots, u_n$. If $s_1, \ldots, s_m \in \Sigma$, we set $\epsilon(s_1 \cdots s_m)= \epsilon( (s_1,1) \cdots (s_m,1))$. For instance, notice that a linear diagram decomposes as a sum $\epsilon(a)+ \epsilon(u)+ \epsilon(b)$ for some words $a,b \in \Sigma^+$ and some letter $u \in \Sigma(\mathcal{G})$.

\subsection{Quasi-median geometry}\label{section:QMgeom}

\noindent
In order to study the algebraic properties of picture products of groups, we want to make such products act on \emph{quasi-median graphs}. We begin by defining these graphs. (For us, a graph does not contain multiple edges nor loops.)

\begin{definition}
A graph is \emph{weakly modular} if it satisfies the following two conditions:
\begin{description}
	\item[(triangle condition)] for any vertex $u$ and any two adjacent vertices $v,w$ at distance $k$ from $u$, there exists a common neighbor $x$ of $v,w$ at distance $k-1$ from $u$;
	\item[(quadrangle condition)] for any vertices $u,z$ at distance $k$ apart and any two neighbors $v,w$ of $z$ at distance $k-1$ from $u$, there exists a common neighbor $x$ of $v,w$ at distance $k-2$ from $u$.
\end{description}
A graph is \emph{quasi-median}\index{Quasi-median graphs} if it weakly modular and does not contain $K_4^-$ and $K_{3,2}$ as induced subgraphs (see Figure \ref{figure12}).
\end{definition}
\begin{figure}
\begin{center}
\includegraphics[trim={0 22cm 22cm 0},clip,scale=0.5]{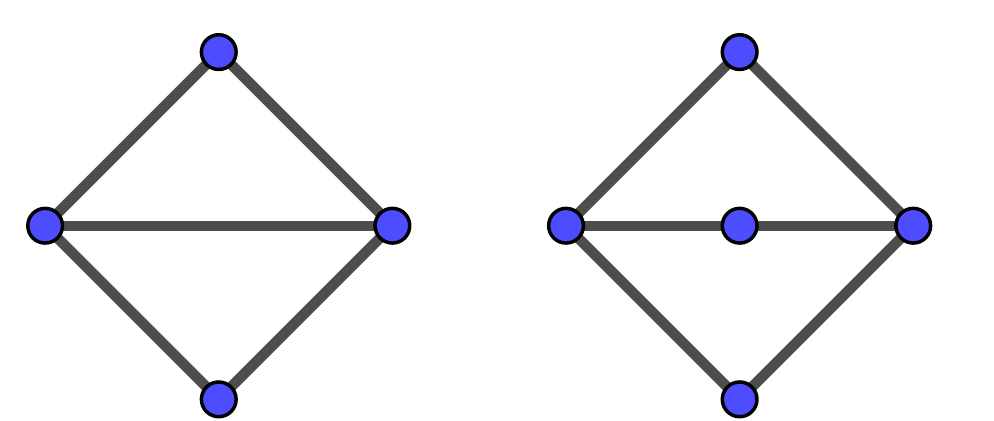}
\end{center}
\caption{From left to right, the graphs $K_4^-$ and $K_{3,2}$.}
\label{figure12}
\end{figure}
\noindent
It is worth noticing that the class of quasi-median graphs includes median graphs. A graph $X$ is \emph{median} if any triple of vertices $x,y,z \in X$ admits a unique vertex $m \in X$ which belongs to the intersection of three geodesics respectively between $x$ and $y$, $y$ and $z$, and $x$ and $z$. Actually, as a consequence of \cite[Proposition 3]{BC}, a graph is median if and only if it is quasi-median and triangle-free. For more information about quasi-median graphs, we refer to \cite{Qm} and references therein. 

\medskip \noindent
The main construction of our section is the following.

\medskip \noindent
Let $\mathcal{P}= \langle \Sigma \mid \mathcal{R} \rangle$ be a semigroup presentation, $\mathcal{G}$ a collection of groups indexed by $\Sigma$, and $w \in \Sigma^+$ a baseword. We define $X_b(\mathcal{P}, \mathcal{G})$ as the graph whose vertices are the classes
$$[\Delta]= \{ \Delta \cdot P \mid P \ \text{permutation diagram} \}$$
where $\Delta$ is a braided diagram over $(\mathcal{P}, \mathcal{G})$; and whose edges link two classes $[\Delta_1]$ and $[\Delta_2]$ if there exists a unitary diagram $U$ satisfying $\Delta_2'= \Delta_1' \cdot U$ for some $\Delta_1' \in [\Delta_1]$ and $\Delta_2' \in [\Delta_2]$. We denote by $X=X_b(\mathcal{P}, \mathcal{G},w)$ the connected component of $X_b(\mathcal{P}, \mathcal{G})$ containing the trivial diagram $\epsilon(w)$. Notice that the vertices of $X$ are precisely the braided $(w,\ast)$-diagrams over $(\mathcal{P}, \mathcal{G})$. As a consequence, the braided picture product $D_b(\mathcal{P}, \mathcal{G},w)$ acts on $X$ by left-multiplication.

\medskip \noindent
Formally, the set of all braided diagrams over $(\mathcal{P}, \mathcal{G})$, modulo dipoles, endowed with the concatenation is a groupoid, and $X_b(\mathcal{P}, \mathcal{G})$ is the coset graph of this groupoid associated to the subgroupoids containing only permutation diagrams and to the generating set defined as the set of unitary diagrams. An easy consequence of this point of view is that our groupoid acts vertex-transitively on $X_b(\mathcal{P}, \mathcal{G})$, so that it is always possible to choose our basepoint. More precisely, if $\Delta \in X_b(\mathcal{P},\mathcal{G},w)$ is a basepoint, then conjugating by $\Delta$ sends $\Delta$ to a trivial diagram, but which does not necessarily belong to $X_b(\mathcal{P},\mathcal{G},w)$: it will belong to $X_b(\mathcal{P},\mathcal{G}, \mathrm{bot}^-(\Delta))$. At the same time, the braided picture product $D_b(\mathcal{P},\mathcal{G},w)$ becomes $D_b(\mathcal{P}, \mathcal{G},\mathrm{bot}^-(\Delta))$. Therefore, we get a commutative diagram
\begin{displaymath}
\xymatrix{ D_b(\mathcal{P},\mathcal{G},w) \ar[rr]^{\text{isomorphism}} \ar[d] & & D_b(\mathcal{P}, \mathcal{G}, \mathrm{bot}(\Delta)) \ar[d] \\ X_b(\mathcal{P},\mathcal{G}, w) \ar[rr]^{\text{isometry}} & & X_b(\mathcal{P}, \mathcal{G},\mathrm{bot}(\Delta)) }
\end{displaymath}
Thus, we may always suppose that a fixed basepoint is a trivial diagram up to changing the base word, which does not disturb the group nor the associated graph.

\medskip \noindent
The main result of this section is that our graph turns out be quasi-median.

\begin{thm}\label{thm:XQM}
$X_b(\mathcal{P}, \mathcal{G},w)$ is a quasi-median graph.
\end{thm}

\noindent
We follow closely the proof of \cite[Proposition 10.12]{Qm}. We begin by studying the geodesics of $X$. 

\begin{definition}
The \emph{length} of a braided diagram $\Delta$ over $(\mathcal{P}, \mathcal{G})$ is the sum of the number of transistors of the reduction of $\Delta$ with its number of wires labelled by letters of $\Sigma(\mathcal{G})$ with non-trivial second coordinates. 
\end{definition}

\noindent
For example, if $\Delta$ denotes the leftmost braided diagram of Figure \ref{figure10}, then $\# \Delta= 6$. Notice that two vertices $[\Delta_1], [\Delta_2] \in X$ are adjacent if and only if $\# \Delta_1^{-1} \Delta_2 = 1$.

\begin{definition}
Let $A_1, \ldots, A_n$ be a sequence of braided diagrams such that the concatenation $A_1 \circ \cdots \circ A_k$ is well-defined for every $1 \leq k \leq n$. If
$$\# \left( A_1 \circ \cdots \circ A_n \right)= \# A_1+ \cdots + \# A_n,$$
we say that the concatenation $A_1 \circ \cdots \circ A_n$ is \emph{absolutely reduced}. (Equivalently, such a concatenation is absolutely reduced if it is reduced and if two wires, respectively connected to the bottom of the frame of some $A_i$ and to the top of the frame of some $A_j$, which are glued together in $A_1 \circ \cdots \circ A_n$ are never both labelled by letters of $\Sigma(\mathcal{G}$ with non-trivial second coordinates.) If a braided diagram $\Delta$ decomposes as an absolutely reduced concatenation $A \circ B$, we say that $A$ is a \emph{prefix} of $\Delta$, written $A \leq \Delta$, and that $B$ is a \emph{suffix} of $\Delta$. 
\end{definition}

\begin{lemma}\label{lem:Xgeod}
Let $[\Delta_1],[\Delta_2] \in X$ be two vertices. Write a braided diagram $\Delta_1^{-1} \Delta_2 \in X$ as an absolutely reduced concatenation $U_1 \circ P_1 \circ \cdots \circ  U_n \circ P_n$ where the $U_i$'s are unitary diagrams and where the $P_i$'s are permutation diagrams. Then
$$[\Delta_1], \ [\Delta_1 \cdot U_1], \ [ \Delta_1 \cdot U_1 \circ P_1 \circ U_2], \ldots, \ [ \Delta_1 \cdot U_1 \circ P_1 \circ \cdots \circ U_n]$$
defines a geodesic from $[\Delta_1]$ to $[\Delta_2]$ in $X$. Conversely, if
$$[\Delta_1], \ [\Delta_1 \cdot A_1], \ [ \Delta_1 \cdot A_1 \cdot A_2], \ldots, \ [\Delta_1 \cdot A_1 \cdots A_n]$$
is a geodesic in $X$ from $[\Delta_1]$ to $[\Delta_2]$ for some diagrams $A_i = U_i \circ P_i$ with $U_i$ unitary and $P_i$ permutational, then there exists a permutation diagram $P$ such that $\Delta_1^{-1} \Delta_2$ decomposes as an absolutely reduced concatenation $A_1 \circ \cdots \circ A_n \circ P$. 
\end{lemma}

\begin{proof}
The path $[\Delta_1], \ [\Delta_1 \cdot U_1], \ [ \Delta_1 \cdot U_1 \circ P_1 \circ U_2], \ldots, \ [ \Delta_1 \cdot U_1 \circ P_1 \circ \cdots \circ U_n]$ has length 
$$n = \sum\limits_{i=1}^n ( \#U_i + \# P_i ) = \# \Delta_1^{-1} \Delta_2$$
as $\# U_i=1$ and $\#P_i=0$ for every $1 \leq i \leq n$.
On the other hand, it is clear that the length of any path in $X$ from $[\Delta_1]$ to $[\Delta_2]$ must be at least $\# \Delta_1^{-1} \Delta_2$. This implies that the previous path turns out to be a geodesic, and that the distance from $[\Delta_1]$ to $[\Delta_2]$ is equal to $\# \Delta$. Now, fix a geodesic 
$$[\Delta_1], \ [\Delta_1 \cdot A_1], \ [ \Delta_1 \cdot A_1 \cdot A_2], \ldots, \ [\Delta_1 \cdot A_1 \cdots A_n]$$
from $[\Delta_1]$ to $[\Delta_2]$ for some diagrams $A_i = U_i \circ P_i$ with $U_i$ unitary and $P_i$ permutational. Because $\Delta_2 \in [\Delta_1 \cdot A_1 \cdots A_n]$, there must exist some permutation diagram $P$ such that $\Delta_2 = \Delta_1 \cdot A_1 \cdots A_n \cdot P$. Notice that $\#P=0$ and $\#A_i=1$ for every $1 \leq i \leq n$. From the inequalities
$$d([\Delta_1],[\Delta_2])= \# \Delta_1^{-1} \Delta_2 \leq \sum\limits_{i=1}^n \# A_i + \#P =n=d([\Delta_1],[\Delta_2]),$$
we deduce that $\# \Delta_1^{-1} \Delta_2 = \sum\limits_{i=1}^n \# A_i + \# P$. Therefore, the concatenation $\Delta_1^{-1} \Delta_2 = A_1 \circ \cdots \circ A_n \circ P$ is absolutely reduced. 
\end{proof}

\noindent
A direct consequence of Lemma \ref{lem:Xgeod} (which was proved during the previous proof in fact) is the following expression of the distance in our graph $X$:

\begin{cor}
The distance between two vertices $[\Delta_1], [\Delta_2] \in X$ is equal to $\# \left( \Delta_1^{-1} \Delta_2 \right)$. 
\end{cor}

\noindent
Our next lemma, which explains how the length of a braided diagram behaves after right-multiplication by a unitary diagram, will be fundamental in the sequel.

\begin{lemma}\label{lem:lengthrightmult}
Let $\Delta$ be a braided diagram and $U$ a unitary diagram such that the concatenation $\Delta \circ U$ is well-defined. If $U$ is a transistor diagram, then
$$\#( \Delta \cdot U)= \left\{ \begin{array}{cl} \# \Delta +1 & \text{if $\Delta \circ U$ is reduced} \\ \# \Delta -1 & \text{otherwise} \end{array} \right.$$
Otherwise, if 
$$U= \epsilon( (w_1,1) \cdots (w_{i-1},1) (w_i,g) (w_{i+1},1) \cdots (w_n,1))$$ 
where $\mathrm{bot}(\Delta)= (w_1,h_1) \cdots (w_n,h_n)$ for some $h_1 \in G_{w_1}, \ldots, h_n \in G_{w_n}$, then
$$\# ( \Delta \cdot U ) = \left\{ \begin{array}{cl} \# \Delta -1 & \text{if $g=h_i^{-1}$ and $h_i \neq 1$} \\ \# \Delta & \text{if $g \neq h_i^{-1}$ and $h_i \neq 1$} \\ \# \Delta+1 & \text{if $h_i=1$} \end{array} \right.$$
\end{lemma}

\begin{proof}
The second assertion of our lemma is clear. So suppose that $U$ is a transistor diagram. If $\Delta \circ U$ is reduced, then this braided diagram is obtained from $\Delta$ by adding a transistor $T$. Notice that the wires of $\Delta \circ U$ which do not belong to $\Delta$ are precisely those which are connected to the bottom side of $T$; these wires are labelled by letters of $\Sigma(\mathcal{G})$ with trivial second coordinates. Therefore, $\# ( \Delta \cdot U)= \# \Delta +1$. Otherwise, if $\Delta \circ U$ is not reduced, then $\Delta \cdot U$ is obtained from $\Delta \circ U$ by reducing a dipole, and so from $\Delta$ by removing a transistor $T$. Notice that, during this process, the wires which are removed are those which are connected to the bottom side of $T$, and they are labelled by letters of $\Sigma (\mathcal{G})$ with trivial second coordinates since $T$ must define a dipole with the transistor of $U$ in $\Delta \circ U$. Therefore, $\# ( \Delta \cdot U)= \# \Delta - 1$. 
\end{proof}

\noindent
Finally, before proving Theorem \ref{thm:XQM}, we need to understand the triangles of $X$. For this purpose, we consider particular complete subgraphs of $X$, which we call \emph{pins}, and which will turn out to be cliques of $X$. (Recall that a \emph{clique} of a graph is a maximal complete subgraph.)

\begin{definition}\label{def:coset}
A \emph{pin} of $X_b(\mathcal{P}, \mathcal{G},w)$ is a complete subgraph generated by the vertices
$$[ \Delta \cdot \epsilon( (w_1,1) \cdots (w_{i-1},1) (w_i,g) (w_{i+1},1) \cdots (w_n,1) )], \ g \in G_{w_i},$$
where $\Delta$ is a fixed braided diagram which satisfies $\mathrm{bot}^-(\Delta)= w_1 \cdots w_n$. 
\end{definition}

\begin{lemma}\label{lem:cosetinter}
The intersection between two distinct pins is either empty or a single vertex.
\end{lemma}

\begin{proof}
Let $C_1,C_2$ be two intersecting pins. Fix some $[\Delta] \in C_1 \cap C_2$ and write $\mathrm{bot}(\Delta)=(w_1,g_1) \cdots (w_n,g_n)$. So
$$C_1= \{ \Delta \cdot \epsilon( (w_1,1) \cdots (w_{i-1},1) (w_i,g) (w_{i+1},1) \cdots (w_n,1) ), \ g \in G_{w_i} \}$$
and 
$$C_2 = \{ \Delta \cdot \epsilon( (w_1,1) \cdots (w_{j-1},1) (w_j,g) (w_{j+1},1) \cdots (w_n,1) ), \ g \in G_{w_j} \}$$
for some $1 \leq i,j \leq n$. Clearly, if $i=j$ then $C_1=C_2$, and if $i \neq j$ then $[\Delta]$ is the single vertex in the intersection $C_1 \cap C_2$. 
\end{proof}

\begin{lemma}\label{lem:triangleincoset}
Any triangle of $X$ is contained into some pin. 
\end{lemma}

\begin{proof}
Given a triangle, up to translating, we may suppose without loss of generality that $[\epsilon(w)]$ is one of its vertices; let $[E],[F]$ denote its other two vertices, where $E$ and $F$ are unitary. 

\medskip \noindent
Suppose by contradiction that $E$ and $F$ are both transistor diagrams. Two cases may happen: either $E^{-1}F$ contains two transistors, which is impossible since otherwise we would have $d([E],[F])\# E^{-1}F \geq 2$; or $E^{-1}F$ contains no transistors, so that, since its wires are labelled by letters of $\Sigma(\mathcal{G})$ with trivial second coordinates, we have $d([E],[F])= \# E^{-1}F =0$, ie., $[E]=[F]$. 

\medskip \noindent
Next, suppose by contradiction that $E$ is a transistor diagram and $F$ a linear diagram. Then $d([E],[F]) = \# E^{-1}F =2$, so that $[E]$ and $[F]$ cannot be adjacent. By symmetry, the same conclusion holds if $E$ is a linear diagram and $F$ a transistor diagram.

\medskip \noindent
Consequently, $E$ and $F$ must be both linear diagrams. So, if $w=w_1 \cdots w_n$ where $w_1, \ldots, w_n \in \Sigma$, then there exist some $1 \leq i,j \leq n$ and $g \in G_{w_i}$, $h \in G_{w_j}$ such that
$$E= \epsilon( (w_1,1) \cdots (w_{i-1},1) (w_i,g) (w_{i+1},1) \cdots (w_n,1) )$$
and
$$F = \epsilon( (w_1,1) \cdots (w_{j-1},1) (w_j,h) (w_{j+1},1) \cdots (w_n,1) ).$$
Since $1= d([E],[F])= \# E^{-1}F$, necessarily $i=j$, so that our triangle must be included into the pin
$$\{ \epsilon( (w_1,1) \cdots (w_{i-1},1) (w_i,g) (w_{i+1},1) \cdots (w_n,1) ) \mid g \in G_{w_i} \},$$
which concludes the proof of our lemma. 
\end{proof}

\begin{proof}[Proof of Theorem \ref{thm:XQM}.]
We begin by proving the triangle condition. So let $[\Delta], [\Gamma],[\Xi] \in X$ be three vertices such that $[\Delta]$ and $[\Gamma]$ are adjacent, and such that the distance from $[\Xi]$ to $[\Delta]$ is equal to the distance from $[\Xi]$ to $[\Gamma]$. Up to conjugating by $\Xi^{-1}$, we may suppose without loss of generality that $\Xi= \epsilon(w)$. So our previous equality between distances can be written equivalently as $\# \Delta= \# \Gamma$. Because $[\Delta]$ and $[\Gamma]$ are adjacent, there must exist some $\Delta' \in [\Delta]$ and some unitary diagram $U$ such that $\Gamma= \Delta' \cdot U$. By noticing that
$$\# \Delta' \cdot U = \# \Gamma = \# \Delta = \# \Delta',$$
we deduce from Lemma \ref{lem:lengthrightmult} that we can write $\mathrm{bot}(\Delta')=(w_1,g_1) \cdots (w_n,g_n)$ and
$$U= \epsilon( (w_1,1) \cdots (w_{i-1},1) (w_i,h) (w_{i+1},1) \cdots (w_n,1)),$$
where $g_i \neq 1$ and $h \neq g_i^{-1}$. Therefore, setting
$$\Phi = \Delta' \cdot \epsilon( (w_1,1) \cdots (w_{i-1},1) (w_i, g_i^{-1}) (w_{i+1},1) \cdots (w_n,1),$$
we have $\# \Phi = \# \Delta'-1 = \# \Delta-1$ and $[\Phi]$ is adjacent to $[\Delta]$ and $[\Gamma]$. Thus, $[\Phi]$ is the vertex we are looking for.

\medskip \noindent
Now, we want to prove the quadrangle condition. So let $[A],[B],[C],[D] \in X$ be four pairwise distinct vertices such that $[B]$ is adjacent to both $[A]$ and $[C]$, such that distance from $[D]$ to $[A]$ is equal to the distance from $[D]$ to $[C]$, say $\ell$, and such that distance from $[D]$ to $[B]$ is equal to $\ell+1$. Up to conjugating by $D^{-1}$, we may suppose without loss of generality that $D=\epsilon(w)$. As a consequence, the equalities between the previous distances can be written as $\# A = \# C= \ell$ and $\# B=\ell-1$. Because $[A]$ and $[B]$ are adjacent, there exist $B' \in [B]$ and some unitary diagram $U_1$ such that $A= B' \cdot U_1$. Similarly, because $[B]$ and $[C]$ are adjacent, there exist $B'' \in [B]$ and some unitary diagram $U_2$ such that $C= B'' \cdot U_2$. We distinguish four cases, depending on whether $U_1,U_2$ are transistor or linear diagrams.

\medskip \noindent
Suppose that $U_1$ and $U_2$ are both transistor diagrams. From the equalities
$$\# B' \cdot U_1 = \# A = \ell = \# B -1 = \# B' - 1$$
and 
$$\# B'' \cdot U_2 = \# C = \ell = \# B - 1 = \# B'' - 1,$$
we deduce thanks to Lemma \ref{lem:lengthrightmult} that the concatenations $B' \circ U_1$ and $B'' \circ U_2$ are not reduced. Therefore, the transistor of $U_1$ in $B' \circ U_1$ must define a dipole together with some transistor $T_1$ of $B$; similarly, the transistor of $U_2$ in $B'' \circ U_2$ must define a dipole together with some transistor $T_2$ of $B$. Notice that $T_1 \neq T_2$, since otherwise we would have $[A]=[C]$. As a consequence, if we fix some $B''' \in [B]$ so that the wires of $B'''$ connected to the bottom side of $T_1$ (resp. $T_2$) are consecutive in left-to-right ordering of the wires connected to the bottom side of the frame, then there exist two words $x,y \in \Sigma^+$ such that $\mathrm{bot}^-(B''')=xy$ and two transistor diagrams $X,Y$ such that $[B''' \cdot (X+\epsilon(y))] = [A]$ and $[B''' \cdot (\epsilon(x)+ Y)]= [C]$. (So, in the concatenation $B''' \circ (X+ \epsilon(y))$, the transistor of $X$ must define a dipole with the transistor $T_1$ of $B$, so that the reduction of this dipole provide an element of $[A]$; and similarly for $T_2$.) Now, set $\Phi = B''' \cdot (X+Y)$. By construction, $[\Phi]$ is adjacent to both $[A]$ and $[C]$, and $\# \Phi = \#B'''-2= \#B -2 = \ell-1$. Therefore, $[\Phi]$ is the vertex we are looking for.

\medskip \noindent
Next, suppose that $U_1$ is a linear diagram and $U_2$ a transistor diagram. From the equality
$$\# B' \cdot U_1 = \# A = \ell - 1 = \#B -1 = \# B' - 1,$$
we deduce thanks to Lemma \ref{lem:lengthrightmult} that $B'$ has a wire $w$ connected to the bottom side of the frame, say labelled by $(m,g) \in \Sigma(\mathcal{G})$, such that $g \neq 1$ and such that $A$ is obtained from $B'$ by replacing the label of $w$ with $(m,1)$. Similarly, we deduce from the equality
$$\# B'' \cdot U_2 = \# C = \ell = \# B-1 = \# B''-1$$
that the transistor of $U_2$ in the concatenation $B'' \circ U_2$ must define a dipole with some transistor $T$ of $B$. A fortiori, the wires which are connected to the bottom side of $T$ must be labelled by letters of $\Sigma(\mathcal{G})$ with trivial second coordinates, so that $w$ cannot be one of these wires. As a consequence, $w$ is also a wire of $C$ (since the wires of $C$ are those of $B''$ (or equivalently, of $B$ or of $B'$) which are not connected to the bottom side of $T$, together with the new wires which are connected to the bottom side of the transistor of $U_2$).  Let $\Phi$ denote the braided diagram obtained from $C$ by replacing the label of $w$ with $(m,1)$. By construction, $[\Phi]$ is adjacent to both $[A]$ and $[C]$, and $\# \Phi = \# C -1 = \ell -1$. Therefore, $[\Phi]$ is the vertex we are looking for.

\medskip \noindent
The case where $U_1$ is a transistor diagram and $U_2$ a linear diagram is symmetric to the previous one.

\medskip \noindent
Finally, suppose that $U_1$ and $U_2$ are both linear diagrams. In this case, since a linear diagram ``commutes'' with any permutation, we may suppose without loss of generality that $B'=B''=B$. From the equalities
$$\# B \cdot U_1 = \# A = \ell = \# B -1$$
and 
$$\# B'' \cdot U_2 = \# C = \ell = \# B - 1,$$
we deduce from Lemma \ref{lem:lengthrightmult} that $B$ has two wires $w_1,w_2$ connected to the bottom side of the frame, say labelled by $(m_1,g_1), (m_2,g_2)$ respectively, such that $g_1 \neq 1, g_2 \neq 1$, such that $A$ is obtained from $B$ by replacing the label of $w_1$ with $(m_1,1)$, and such that $C$ is obtained from $B$ by replacing the label of $w_2$ with $(m_2,1)$. Notice that, since $[A] \neq [C]$, necessarily $w_1 \neq w_2$. Let $\Phi$ denote the braided diagram obtained from $B$ by replacing the labels of $w_1$ and $w_2$ with $(m_1,1)$ and $(m_2,1)$ respectively. Then $[\Phi]$ is adjacent to both $[A]$ and $[C]$, and $\# \Phi = \# B -2 = \ell -1$. Thus, $[\Phi]$ is the vertex we are looking for.

\medskip \noindent
This concludes the proof of the quadrangle condition. To conclude the proof of our theorem, it remains to show that $X$ does not contain induced subgraphs isomorphic to either $K_{2,3}$ or $K_4^-$. 

\medskip \noindent
Notice from Lemmas \ref{lem:triangleincoset} and \ref{lem:cosetinter} that two triangles intersecting along an edge must be included into a complete subgraph. Therefore, $X$ does not contain induced subgraphs isomorphic to $K_4^-$. Next, in order to prove that $X$ does not contain induced subgraphs isomorphic to $K_{3,2}$, it is sufficient to show that there exist at most two geodesics between two vertices at distance two apart. This is a consequence of the following claim:

\begin{claim}\label{claim:atmosttwogeod}
Let $[\Delta] \in X$ be a vertex where $\# \Delta = 2$. There exist at most two geodesics in $X$ from $[\epsilon(w)]$ to $[\Delta]$. 
\end{claim}

\noindent
Fix a geodesic
$$[\epsilon(w)], \ [A_1], \ [A_1 \cdot A_2]$$
from $[\epsilon(w)]$ to $[\Delta]$ in $X$. According to Lemma \ref{lem:Xgeod}, there exists a permutation diagram $P$ such that $\Delta$ decomposes as the absolutely reduced concatenation $A_1 \circ A_2 \circ P$. Notice that $\# A_1 =1$, so we can write $A_1= P_1 \circ A_1' \circ P_2$ for some unitary diagram $A_1'$. Notice that the concatenation $\Delta = P_1 \circ A_1' \circ P_2 \circ A_2 \circ P$ is also absolutely reduced. We distinguish two cases.

\medskip \noindent
Suppose first that $A_1'$ is a transistor diagram. Observe that the class $[P_1 \circ A_1']$ is uniquely determined by $\mathrm{top}^-(P_1)$, the top and bottom labels of the transistor of $A_1'$, and the places along the top side of the frame where the wires which are connected to the top side of the transistor of $A_1'$ are connected. On the other hand, all these data are uniquely determined by the choice of a transistor of $\Delta$. Because $\#\Delta=2$, we know that there exist at most two such transistors, so we have at most two choices for $[P_1 \circ A_1']=[A_1]$. But this class is also the unique interior vertex of our geodesic, so that this vertex uniquely determines the geodesic itself.

\medskip \noindent
Next, suppose that $A_1'$ is a linear diagram. Observe that the class $[P_1 \circ A_1']$ is uniquely determined by $\mathrm{top}^-(P_1)$, the place along the top side of the frame where the unique wire labelled by a letter of $\Sigma(\mathcal{G})$ with a non-trivial second coordinate is connected, and the label of this wire. On the other hand, all these data are uniquely determined by the choice of a wire of $\Delta$ which is labelled by a letter of $\Sigma(\mathcal{G})$ with non-trivial second coordinate. But, since $\# \Delta=2$, we know that there exist at most two such wires, so we have at most two choices for $[P_1 \circ A_1]=[A_1]$. Because this class is also the unique interior vertex of our geodesic, this vertex uniquely determines the geodesic itself. 

\medskip \noindent
This concludes the proof of our claim, and finally of our theorem.
\end{proof}

\noindent
It is worth noticing that, when all the groups of $\mathcal{G}$ are trivial, then the braided picture product $D_b(\mathcal{P}, \mathcal{G},w)$ coincides with the braided diagram group $D_b(\mathcal{P}, w)$. Therefore, it follows that any braided diagram group acts on a quasi-median graph. In fact, in this context, the graph $X_b(\mathcal{P},\mathcal{G},w)$ turns out to be precisely the one-skeleton of the cube complex constructed in \cite{FarleyPicture}. We recover from Theorem \ref{thm:XQM} that this cube complex is CAT(0).

\begin{cor}
If all the groups of $\mathcal{G}$ are trivial, $X(\mathcal{P}, \mathcal{G},w)$ is a median graph. 
\end{cor}

\begin{proof}
If all the groups of $\mathcal{G}$ are trivial, then every pin of $X$ is reduced to a single vertex. It follows from Lemma \ref{lem:triangleincoset} that $X$ must be triangle-free. The conclusion follows since a triangle-free quasi-median graph turns out to be median (as a consequence of \cite[Proposition 3]{BC}; see also \cite[Corollary 2.92]{Qm}). 
\end{proof}

\subsection{A decomposition theorem}\label{section:DecompositionTheorem}

\noindent
In this section, our goal is to state and prove a general result allowing us to decompose a given group which acts on a quasi-median graph in a specific way. We begin by giving a few general definitions related to quasi-median graphs.

\medskip \noindent
Given a graph $X$, a vertex $x \in X$ and subgraph $Y \subset X$, a vertex $y \in Y$ is a \emph{gate} in $Y$ for $x$ if, for every $z \in Y$, there exists a geodesic from $x$ to $z$ passing through $y$. If any vertex of $X$ admits a gate in $Y$, then $Y$ is a \emph{gated} subgraph. It is worth noticing that, when it exists, the gate of $x$ in $Y$ coincides with the unique vertex of $Y$ which minimises the distance to $x$. As a consequence, the gate of $x$ in $Y$ will be also refer to as the \emph{projection} of $x$ onto $Y$. 

\medskip \noindent
Gatedness can be thought of as a strong convexity property. As shown in \cite[Theorem 1]{quasimedian}, cliques of quasi-median graphs turn out to be gated. We will often use this fact in the sequel without mentioning it.

\medskip \noindent
Another fundamental tool used to study the geometry of quasi-median graph is the notion of \emph{hyperplane}. 

\begin{definition}
Let $X$ be a quasi-median graph. A \emph{hyperplane} is an equivalence class of edges with respect to the transitive closure of the relation identifying two edges whenever they belong to a common clique or whenever they are opposite edges in some square of $X$. The \emph{neighborhood} of a hyperplane $J$, denoted by $N(J)$, is the subgraph of $X$ generated by the edges of $J$. A \emph{fiber} (resp. a \emph{sector}) of $J$ is a connected component of the graph $\partial J$ (resp. $X \backslash \backslash J$) obtained from $N(J)$ (resp. from $X$) by removing the interiors of the edges of $J$. Finally, two hyperplanes $J_1$ and $J_2$ are \emph{transverse} if there exist a prism $C_1 \times C_2$ in $X$ such that $C_1 \subset J_1$ and $C_2 \subset J_2$.
\end{definition}

\noindent
Recall that a \emph{clique} is a maximal complete subgraph, and that a \emph{prism} is a subgraph which decomposes as the Cartesian product of two cliques. 

\medskip \noindent
In the sequel, the following result about hyperplanes in quasi-median graphs will be needed. We refer to the appendix for a proof. 

\begin{thm}\label{thm:MainQM}
Let $X$ be a quasi-median graph and $J$ a hyperplane. The neighborhood and the fibers of $J$ are gated, every geodesic of $X$ crosses $J$ at most once, and $X \backslash \backslash J$ is disconnected. More precisely, if $C$ is a clique contained into $J$ and if $p : X \to C$ denotes the projection onto $C$, then $\{ p^{-1}(x) \mid x \in C\}$ is the collection of the connected components of $X \backslash \backslash J$.
\end{thm}

\noindent
An easy consequence of this statement is:

\begin{cor}\label{cor:projseparate}
Let $X$ be a quasi-median graph, $Y \subset X$ a gated subgraph and $x \in X$ a vertex. Any hyperplane separating $x$ from its projection onto $Y$ separates $x$ from $Y$. 
\end{cor}

\begin{proof}
Suppose by contradiction that there exists a hyperplane $J$ separating $x$ and its projection $p$ onto $Y$ which intersects $Y$. In particular, $J$ must separate $p$ from some vertex $y \in Y$. As a consequence, if we fix a geodesic $\gamma$ between $x$ and $y$ passing through $p$, necessarily $J$ has to intersect $J$ at least twice. This contradicts Theorem \ref{thm:MainQM}.
\end{proof}

\noindent
The following general lemma will be also useful in the next section:

\begin{lemma}
Let $X$ be a quasi-median graph, $C$ a clique and $J$ the corresponding hyperplane. If an edge $e$ is dual to $J$, there exists a sequence of edges $e_1, \ldots, e_n$ such that $e_1=e$, $e_n \subset C$, and $e_i$ and $e_{i+1}$ are opposite sides of a square for every $1 \leq i \leq n-1$. 
\end{lemma}

\begin{proof}
Let $a,b \in X$ denote the endpoints of $e$. Because the clique of $X$ containing $e$ is contained into $J$, it follows from Theorem \ref{thm:MainQM} that $a$ and $b$ belong to distinct sectors delimited by $J$. Again according to Theorem \ref{thm:MainQM}, we know that $a$ and $b$ have distinct projections onto $C$, say $a'$ and $b'$ respectively. Because $a'$ is the unique vertex of $C$ minimising the distance to $a$, it follows that $d(a,a') \leq d(b,b')$. Similarly, we know that $d(b,b') \leq d(a,a')$, hence $d(a,a')=d(b,b')$. The consequence is that the four vertices $a,a',b,b'$ all belong to the union $I(a,b')$ of all the geodesic between $a$ and $b'$. According to \cite[Theorem 1]{quasimedian}, such a subgraph must be median (or equivalently \cite{mediangraphs} the one-skeleton of a CAT(0) cube complex). Because median graphs are triangle-free quasi-median graphs, it is sufficient to show that the two edges $[a,b]$ and $[a',b']$ are dual to the same hyperplane in $I(a,b')$ in order to deduce our lemma. The desired observation follows from Theorem \ref{thm:MainQM} as, in a median (quasi-)graph, a geodesic crosses each hyperplane at most once. 
\end{proof}

\noindent
Before stating the main result of this section, we need two more definitions:

\begin{definition}
Let $G$ be a group acting on a quasi-median graph $X$. The \emph{rotative stabiliser} of a hyperplane $J$ is
$$\mathrm{stab}_{\circlearrowleft}(J)= \bigcap \{ \mathrm{stab}(C) \mid \text{$C$ clique in $J$} \}.$$
Given a collection hyperplanes $\mathcal{J}$, the action $G \curvearrowright X$ is \emph{$\mathcal{J}$-rotative} if, for every hyperplane $J \in \mathcal{J}$, the rotative stabiliser $\mathrm{stab}_{\circlearrowleft}(J)$ acts freely and transitively on the set of fibers of $J$. 
\end{definition}

\begin{definition}
Let $X$ be a quasi-median graph, $\mathcal{J}$ a collection of hyperplanes and $x_0 \in X$ a base vertex. A subcollection $\mathcal{J}_0 \subset \mathcal{J}$ is $x_0$-peripheral if, for every $J \in \mathcal{J}_0$, there does not exist a hyperplane of $\mathcal{J}$ separating $J$ and $x_0$.
\end{definition}

\noindent
We are now ready to state our decomposition theorem, which is a slight variation of \cite[Theorem 10.54]{Qm}.

\begin{thm}\label{thm:splittingthmQm}
Let $G$ be a group acting $\mathcal{J}$-rotatively on a quasi-median graph $X$. Fix a basepoint $x_0 \in X$, and assume that, if $J_1,J_2 \in \mathcal{J}$ are two transverse hyperplanes, then any element of $\mathrm{stab}_{\circlearrowleft} (J_1)$ commutes with any element of $\mathrm{stab}_\circlearrowleft (J_2)$. If $Y \subset X$ denotes the intersection of the sectors containing $x_0$ which are delimited by a hyperplane of $\mathcal{J}$, then
$$G= \mathrm{Rot}(\mathcal{J}) \rtimes \mathrm{stab}(Y), \ \text{where} \ \mathrm{Rot}(\mathcal{J})= \langle  \mathrm{stab}_{\circlearrowleft}(J), \ J \in \mathcal{J} \rangle.$$
Moreover, if $\mathcal{J}_0 \subset \mathcal{J}$ denotes the unique maximal $x_0$-peripheral subcollection of $\mathcal{J}$, then $\mathrm{Rot}(\mathcal{J})$ decomposes as a graph product $\Delta \mathcal{G}$, where $\Delta$ is the graph whose vertices are the hyperplanes of $\mathcal{J}_0$ and whose edges link two hyperplanes which are transverse, and where $\mathcal{G}= \{ \mathrm{stab}_{\circlearrowleft}(J) \mid J \in \mathcal{J}_0 \}$.
\end{thm}

\noindent
Recall that, given a \emph{simplicial graph} $\Gamma$ and a collection of groups $\mathcal{G}=\{ G_u \mid u \in V(\Gamma) \}$ indexed by the vertices of $\Gamma$, the \emph{graph product} $\Gamma \mathcal{G}$ is the quotient $$\left( \underset{u \in V(\Gamma)}{\ast} G_u \right) / \langle \langle [g,h]=1, g \in G_u, h \in G_v \ \text{if} \ (u,v) \in E(\Gamma) \rangle \rangle.$$
Every group of $\mathcal{G}$, called a \emph{vertex-group}, naturally embeds into $\Gamma \mathcal{G}$. For convenience, we will identify each vertex-group with its image into the graph product.

\noindent
A \emph{word} in $\Gamma \mathcal{G}$ is a product $g_1 \cdots g_n$ for some $n \geq 0$ and, for every $1 \leq i \leq n$, $g_i \in G$ for some $G \in \mathcal{G}$; the $g_i$'s are the \emph{syllables} of the word, and $n$ is the \emph{length} of the word. Clearly, the following operations on a word do not modify the element of $\Gamma \mathcal{G}$ it represents:
\begin{description}
	\item[Cancellation:] delete the syllable $g_i=1$;
	\item[Amalgamation:] if $g_i,g_{i+1} \in G$ for some $G \in \mathcal{G}$, replace the two syllables $g_i$ and $g_{i+1}$ by the single syllable $g_ig_{i+1} \in G$;
	\item[Shuffling:] if $g_i$ and $g_{i+1}$ belong to two adjacent vertex-groups, switch them.
\end{description}
A word is \emph{reduced} if its length cannot be shortened by applying these elementary moves. Every element of $\Gamma \mathcal{G}$ can be represented by a reduced word, and this word is unique up to the shuffling operation. This allows us to define the \emph{length} of an element $g \in \Gamma \mathcal{G}$ as the length of any reduced word representing $g$; and its \emph{support}, denoted by $\supp(g)$, as the set of vertices of $\Gamma$ which corresponds exactly to the vertex-groups containing the syllables of $g$. For more information, we refer to \cite{GreenGP} (see also \cite{HsuWise, GPvanKampen}). The following definition will also be useful:

\begin{definition}\label{def:headtail}
Let $\Gamma$ be a simplicial graph, $\mathcal{G}$ a collection of groups indexed by $V(\Gamma)$, and $g \in \Gamma \mathcal{G}$ an element. The \emph{head} of $g$, denoted by $\mathrm{head}(g)$, is the set of syllables of some reduced word representing $g$ which appear as the first syllable of some reduced word representing $g$. 
\end{definition}

\noindent
In order to show that a group decompose as a graph products, a convenient method is to apply the following ping-pong result coming from \cite[Proposition 8.44]{Qm}, whose proof is reproduced below for the reader's convenience:

\begin{prop}\label{prop:pingpong}
Let $G$ be a group acting on a set $X$, $\Gamma$ a simplicial graph and $\mathcal{H}= \{H_v \mid v \in V(\Gamma) \}$ a collection of subgroups. Suppose that $\bigcup\limits_{v \in V(\Gamma)} H_v$ generates $G$ and that $g$ and $h$ commute for every $g \in H_u$ and $h \in H_v$ if $u$ and $v$ are two adjacent vertices of $\Gamma$. Next, suppose that there exist a collection $\{ X_v \mid v \in V(\Gamma) \}$ of subsets of $X$ and a point $x_0 \in X \backslash \bigcup\limits_{v \in V(\Gamma)} X_v$ satisfying:
\begin{itemize}
	\item if $u,v \in V(\Gamma)$ are adjacent, then $g \cdot X_u \subset X_u$ for every $g \in H_v \backslash \{ 1 \}$;
	\item if $u,v \in V(\Gamma)$ are not adjacent and distinct, then $g \cdot X_u \subset X_v$ for every $g \in H_v \backslash \{ 1 \}$;
	\item for every $u \in V(\Gamma)$ and $g \in H_u \backslash \{ 1 \}$, $g \cdot x_0 \in X_u$.
\end{itemize}
Then $G$ is isomorphic to the graph product $\Gamma \mathcal{H}$. 
\end{prop}

\begin{proof}
By our assumptions on the subgroups of $\mathcal{H}$, we deduce that there exists a natural surjective morphism $\Gamma \mathcal{H} \twoheadrightarrow G$. In order to show that this morphism is also injective, we want to prove the following claim: for any non-empty reduced word $w$ of $\Gamma \mathcal{H}$, thought of as an element of $G$, $w \cdot x_0 \in X_u$ where $u$ is a vertex of $\Gamma$ which belongs to the support of the head of $w$. Notice that, since $x_0 \notin \bigcup\limits_{v \in V(\Gamma)} X_v$ by assumption, this implies that $w \cdot x_0 \neq x_0$. Proving this claim is sufficient to conclude the proof of our proposition. Also, an immediate consequence of the claim is the following fact, which we record for future use:

\begin{fact}\label{fact:pingpong}
For every non-trivial $g \in G$, we have $g \cdot x_0 \in \bigcup\limits_{u \in V(\Gamma)} X_u$.
\end{fact}

\noindent
So let us turn to the proof of our claim. We argue by induction on the length of $w$. If $w$ has length one, then $w \in H_u \backslash \{ 1 \}$ for some $u \in V(\Gamma)$. Our third assumption implies $w \cdot x_0 \in X_u$. Next, suppose that $w$ has length at least two. Write $w=gw'$ where $g$ is the first syllable of $w$, and $w'$ the rest of the word. Say $g \in H_u \backslash \{ 1 \}$. We know from our induction hypothesis that $w' \cdot x_0 \in X_v$ where $v$ is a vertex of $\Gamma$ which belongs to the support of the head of $w'$. Notice that $u \neq v$ since otherwise the word $gw'$ would not be reduced. Two cases may happen: either $u$ and $v$ are not adjacent, so that our second assumption implies that $w \cdot x_0 \in g \cdot X_v \subset X_u$; or $u$ and $v$ are adjacent, so that our first assumption implies that $w \cdot x_0 \in g \cdot X_v \subset X_v$. It is worth noticing that, in the former case, $u$ clearly belongs to the support of the head of $w$ since $g$ belongs to the head of $w$; in the latter case, if we write $w'=hw''$ where $h$ is a syllable of the head of $w'$ which belongs to $H_v$, then
$$w = gw'= g hw'' = hgw'',$$
so $h$ also belongs to the head of $w$, and a fortiori $v$ belongs to the support of the head of $w$. This concludes the proof.
\end{proof}

\noindent
Before turning to the proof of Theorem \ref{thm:splittingthmQm}, we need one last preliminary lemma:

\begin{lemma}\label{lem:rotativestab}
Let $G$ be a group acting on a quasi-median graph $X$ and $J_1,J_2$ two transverse hyperplanes. For any element $g \in \mathrm{stab}_\circlearrowleft (J_1)$ and any sector $S$ delimited by $J_2$, we have $g \cdot S = S$. 
\end{lemma}

\begin{proof}
Because $J_1$ and $J_2$ are transverse, there exists a prism $C_1 \times C_2$ such that $C_1 \subset J_1$ and $C_2 \subset J_2$. According to Theorem \ref{thm:MainQM}, there exists a vertex $x \in C_2$ such that $S=p^{-1}(x)$ where $p : X \to C_2$ denotes the projection onto $C_2$. Without loss of generality, we may suppose that $C_1 \cap C_2=\{x\}$. Fix an element $g \in \mathrm{stab}_\circlearrowleft (J_1)$. Notice that, because $g$ stabilises $C_1$, $g$ sends $C_2$ to another clique dual to $J_2$, hence $g \cdot J_2$. Therefore, we know that $g$ permutes the sectors delimited by $J_2$. On the other hand, because $g \cdot x$ belongs to $gC_1 = C_1$, the projection of $g \cdot x$ onto $C_2$ must be $x$, hence $g \cdot x \in S$. We conclude that $g \cdot S=S$ as desired. 
\end{proof}

\noindent
We are now ready to prove Theorem \ref{thm:splittingthmQm}. Our argument follows closely the proof of \cite[Theorem 10.54]{Qm}.

\begin{proof}[Proof of Theorem \ref{thm:splittingthmQm}.]
Notice that, since $\mathcal{J}$ is $G$-invariant, the subgroup $\mathrm{Rot}(\mathcal{J})$ is normal. 

\medskip \noindent
Let $g \in G$. Fix some $r \in \mathrm{Rot}(\mathcal{J})$ and suppose that $rg \cdot x_0 \notin Y$. Then there exists some $J \in \mathcal{J}$ which separates $rg \cdot x_0$ from $Y$. Since the action $\mathrm{stab}_{\circlearrowleft}(J) \curvearrowright \mathcal{S}(J)$ is transitive, there exists some $s \in \mathrm{stab}_{\circlearrowleft}(J)$ which sends the sector delimited by $J$ which contains $rg \cdot x_0$ to the sector delimited by $J$ which contains $x_0$. Let $a \in N(J)$ denote the projection of $rg \cdot x_0$ onto $N(J)$ and $C$ the clique dual to $J$ containing $a$. Notice that
$$\begin{array}{lcl} d(x_0,rg \cdot x_0) & = & d(rg \cdot x_0,a)+d(a,x_0) = d(srg \cdot x_0,s \cdot a) + d(x_0,s \cdot a)+1 \\ \\ & \geq & d(x_0,srg \cdot x_0)+1 \end{array}$$
Thus, if we choose some $r \in \mathrm{Rot}(\mathcal{J})$ such that
$$d(rg \cdot x_0,g \cdot x_0)= \min \{ d(sg \cdot x_0,g \cdot x_0) \mid s \in \mathrm{Rot}(\mathcal{J}) \},$$ 
we deduce from the previous observation that $rg \cdot x_0 \in Y$. On the other hand, $G$ permutes the connected components of $X$ cutting along the hyperplanes of $\mathcal{J}$, and $Y$ is precisely the connected component which contains $x_0$, so $rg \in \mathrm{stab}(Y)$. Therefore,
$$g \in r^{-1} \cdot \mathrm{stab}(Y) \subset \mathrm{Rot}(\mathcal{J}) \cdot \mathrm{stab}(Y).$$
Thus, we have proved that $G= \mathrm{Rot}(\mathcal{J}) \cdot \mathrm{stab}(Y)$. 

\medskip \noindent
Next, we want to apply Proposition \ref{prop:pingpong} in order to prove that $\mathrm{Rot}(\mathcal{J})$ is isomorphic to the graph product $\Delta \mathcal{G}$. 

\medskip \noindent
The first point to verify is that $\mathrm{Rot}(\mathcal{J})= \langle \mathrm{stab}_{\circlearrowleft}(J), \ J \in \mathcal{J}_0 \rangle$. Let $J_1 \in \mathcal{J}$. We want to prove that there exists some $r \in \langle \mathrm{stab}_{\circlearrowleft}(J), \ J \in \mathcal{J}_0 \rangle$ such that $rJ_1 \in \mathcal{J}_0$, which is sufficient to deduce the previous equality. 

\medskip \noindent
Fix some $r \in \langle \mathrm{stab}_{\circlearrowleft}(J), \ J \in \mathcal{J}_0 \rangle$, and suppose that $rJ_1 \notin \mathcal{J}_0$. Let $y_0$ denote the projection of $x_0$ onto $N(rJ_1)$. If no hyperplane of $\mathcal{J}$ separates $x_0$ and $y_0$, then $\mathcal{J}_0 \cup \{ rJ \}$ defines a new $x_0$-peripheral subcollection of $\mathcal{J}$, contradicting the maximality of $\mathcal{J}_0$. Therefore, there exists some hyperplane $J_2 \in \mathcal{J}$ separating $x_0$ and $y_0$. As a consequence of Corollary \ref{cor:projseparate}, $J_2$ also separates $rJ_1$ and $x_0$. Notice that, if $J_2 \notin \mathcal{J}_0$, then similarly there exists a third hyperplane separating $J_2$ and $x_0$, and so on. Since there exist only finitely many hyperplanes separating $rJ_1$ and $x_0$, we can suppose without loss of generality that $J_2 \in \mathcal{J}_0$. Let $s \in \mathrm{stab}_{\circlearrowleft}(J_2)$ be an element which sends the sector delimited by $J_2$ containing $rJ_1$ to the sector delimited by $J_2$ containing $x_0$. Notice that
$$\begin{array}{lcl} d(x_0,N(rJ_1)) & \geq & d(N(rJ_1),N(J_2)) + d(x_0,N(J_2)) +1 \\ \\ & \geq & d(N(srJ_1),N(J_2)) + d(x_0,N(J_2))+1 \\ \\ & \geq & d(x_0,N(srJ_1)) +1 \end{array}$$
Therefore, if we choose $r$ so that
$$d(x_0,N(rJ_1))= \min \left\{ d(x_0,N(sJ_1)) \mid s \in \langle  \mathrm{stab}_{\circlearrowleft}(J), \ J \in \mathcal{J}_0 \rangle \right\},$$
then $rJ_1 \in \mathcal{J}_0$. This concludes the proof of our first point.

\medskip \noindent
Next, notice that by assumption any element of $\mathrm{stab}_{\circlearrowleft}(J_1)$ commutes with any element of $\mathrm{stab}_{\circlearrowleft}(J_2)$ whenever $J_1,J_2 \in \mathcal{J}_0$ are adjacent in $\Delta$. Now, if $J \in \mathcal{J}_0$, let $X_J$ be the union of all the sectors which do not contain $x_0$. Notice that, for every $J \in \mathcal{J}_0$ and every $g \in \mathrm{stab}_{\circlearrowleft}(J) \backslash \{ 1 \}$, necessarily $g \cdot x_0 \in X_J$ since the action $\mathrm{stab}_{\circlearrowleft} (J) \curvearrowright \mathcal{S}(J)$ is free. Moreover, it follows from Lemma \ref{lem:rotativestab} that, if $J_1,J_2 \in \mathcal{J}_0$ are adjacent, then $g \cdot X_{J_1} \subset X_{J_1}$ for every $g \in \mathrm{stab}_{\circlearrowleft}(J_2) \backslash \{ 1 \}$. Finally, we need to verify that, for every $J_1, J_2 \in \mathcal{J}_0$ which are not transverse and for every $g \in \mathrm{stab}_{\circlearrowleft}(J_1) \backslash \{ 1 \}$, we have $g \cdot X_{J_2} \subset X_{J_1}$. Because $\mathcal{J}_0$ is a $x_0$-peripheral subcollection of $\mathcal{J}$, necessarily $X_{J_2}$ is contained into the sector $S$ delimited by $J_1$ which contains $x_0$. On the other hand, $g \cdot S \subset X_{J_1}$ because the action $\mathrm{stab}_{\circlearrowleft}(J_1) \curvearrowright \mathcal{S}(J_1)$ is free, hence
$$g \cdot X_{J_1} \subset g \cdot S \subset X_{J_2}.$$
Thus, all the hypotheses of Proposition \ref{prop:pingpong} are satisfied, and it follows that $\mathrm{Rot}(\mathcal{J})$ is isomorphic to $\Delta \mathcal{G}$.

\medskip \noindent
Since the hypotheses of Proposition \ref{prop:pingpong} hold, we also know from Fact \ref{fact:pingpong} that, if $g \in \mathrm{Rot}(\mathcal{J})$ is non-trivial, then $g \cdot x_0 \in X_{J}$ for some $J \in \mathcal{J}$. On the other hand, $g$ permutes the connected components of $X$ cutting along the hyperplanes of $\mathcal{J}$, $Y$ is precisely the connected component which contains $x_0$, and $X_J$ is a union of connected components, so $g \cdot Y \subset Y$. We deduce from  
$$g \cdot Y \cap Y \subset X_J \cap Y = \emptyset$$
that $g \cdot Y \cap Y \neq \emptyset$, so that $g$ does not stabilise the connected component $Y$, ie., $g \notin \mathrm{stab}(Y)$. Thus, we have proved that $\mathrm{Rot}(\mathcal{J}) \cap \mathrm{stab}(Y)= \{ 1 \}$. This concludes the proof of the decomposition $G = \mathrm{Rot}(\mathcal{J}) \rtimes \mathrm{stab}(Y)$.
\end{proof}

\subsection{Picture products as semidirect products}

\noindent
In this section, our goal is to exploit the actions of braided picture products on our quasi-median graphs in order to decompose them as semidirect products. More precisely, our decomposition theorem is:

\begin{thm}\label{thm:split}
Let $\mathcal{P}= \langle \Sigma \mid \mathcal{R} \rangle$ be a semigroup presentation, $\mathcal{G}$ a collection of groups indexed by $\Sigma$, and $w \in \Sigma^+$ a baseword. Suppose that the following technical condition is satisfied:
\begin{itemize}
	\item[\emph{(+)}]  Let $m \in \Sigma^+$ be a non-empty word and $\ell \in \Sigma$ a letter such that $G_{\ell}$ is non-trivial and such that there exists at least one braided $(w,m \ell)$-diagram. There does not exist a non-trivial braided $(m,m)$-diagram $P$ such that, for every word $z \in \Sigma^+$ and every braided $(m,z)$-diagram $U$, $U^{-1}PU$ is a permutation $(z,z)$-diagram.
\end{itemize}
Then the braided picture product $D_b(\mathcal{P}, \mathcal{G},w)$ splits as a semidirect product $\Gamma \rtimes D_b(\mathcal{P},w)$ where $\Gamma$ is a graph product of groups which belong to $\mathcal{G}$ and where $D_b(\mathcal{P},w)$ coincides with the subgroup of $D_b(\mathcal{P},\mathcal{G},w)$ whose braided diagrams have their wires labelled by letters of $\Sigma(\mathcal{G})$ with trivial second coordinates. 
\end{thm}

\noindent
Our goal is to apply Theorem \ref{thm:splittingthmQm} to the braided picture group $D_b(\mathcal{P},\mathcal{G},w)$ acting on $X$ with respect to the set $\mathcal{J}$ of \emph{linear hyperplanes} of $X$, ie., hyperplanes dual to pins. Unfortunately, this action is not automatically $\mathcal{J}$-rotative (see Remark \ref{remark:notrotative} below). Nevertheless, it turns out to be $\mathcal{J}$-rotative under the additional technical condition (+) mentioned in our theorem. This condition will be easily satisfied when we will apply Theorem \ref{thm:split} in the proof of Theorem \ref{thm:sequence}. 

\begin{prop}\label{prop:rotative}
Let $\mathcal{P}= \langle \Sigma \mid \mathcal{R} \rangle$ be a semigroup presentation, $\mathcal{G}$ a collection of groups indexed by $\Sigma$, and $w \in \Sigma^+$ a baseword. Suppose that the following technical condition is satisfied:
\begin{itemize}
	\item[\emph{(+)}]  Let $m \in \Sigma^+$ be a non-empty word and $\ell \in \Sigma$ a letter such that $G_{\ell}$ is non-trivial and such that there exists at least one braided $(w,m \ell)$-diagram. There does not exist a non-trivial braided $(m,m)$-diagram $P$ such that, for every word $z \in \Sigma^+$ and every braided $(m,z)$-diagram $U$, $U^{-1}PU$ is a permutation $(z,z)$-diagram.
\end{itemize}
Then the action $D_b(\mathcal{P}, \mathcal{G},w) \curvearrowright X_b(\mathcal{P}, \mathcal{G},w)$ is $\mathcal{J}$-rotative.
\end{prop}

\noindent
The key point to prove our proposition is to understand linear hyperplanes, and to understand these hyperplanes, the first step is to understand the squares of $X$. This is the goal of our next lemma. 

\begin{lemma}\label{lem:square}
Let $[A],[B],[D],[C]$ be the vertices of some induced square of $X$, in a cyclic order. There exist a permutation diagram $P$, two words $u,v \in \Sigma^+$ satisfying $\mathrm{bot}^-(P)=uv$, a unitary $(u,\ast)$-diagram $U$, and a unitary $(v, \ast)$-diagram $V$ such that, up to permuting $B$ and $C$, one has $[B]= [A \cdot P \cdot (U+ \epsilon(v))]$, $[C]=[A \cdot P \cdot (\epsilon(u)+V)]$, and $[D]= [A \cdot P \cdot (U+V)]$. 
\end{lemma}

\begin{proof}
Because our square is induced, the distance between $[A]$ and $[D]$ is equal to two, so that it follows from Lemma \ref{lem:Xgeod} that $[D]=[A \cdot S]$ for some braided diagram $S$ satisfying $\# S=2$. Six cases may happen:
\begin{itemize}
	\item[(i)] $S$ contains two wires labelled by letters of $\Sigma(\mathcal{G})$ with non-trivial second coordinates;
	\item[(ii)] $S$ contains a transistor $T$ and a wire labelled by a letter of $\Sigma(\mathcal{G})$ with a non-trivial second coordinate which is connected to the top side of $T$;
	\item[(iii)]  $S$ contains a transistor $T$ and a wire labelled by a letter of $\Sigma(\mathcal{G})$ with a non-trivial second coordinate which is connected to the bottom side of $T$;
	\item[(iv)]  $S$ contains a transistor $T$ and a wire labelled by a letter of $\Sigma(\mathcal{G})$ with a non-trivial second coordinate which is disjoint from the top and bottom sides of $T$;
	\item[(v)] $S$ contains two transistors $T_1,T_2$ satisfying $T_2 \prec T_1$;
	\item[(vi)] $S$ contains two transistors which are not $\prec$-related.
\end{itemize}
In the cases $(i)$, $(iv)$ and $(vi)$, one can find two permutation diagrams $P,Q$, two words $u,v \in \Sigma^+$ satisfying $\mathrm{bot}^-(P)=uv$, a unitary $(u,\ast)$-diagram $U$, and a unitary $(v, \ast)$-diagram $V$ such that $S= P \cdot (U+V) \cdot Q$. Therefore,
$$[D]=[A \cdot P \cdot (U+V) \cdot Q] = [A \cdot P \cdot (U+V)].$$
Notice that 
$$[A], \ [A \cdot P \cdot (U+ \epsilon(v))], \ [D] \ \text{and} \ [A], \ [A \cdot P \cdot ( \epsilon(u)+V)], \ [D]$$
are two geodesics from $[A]$ to $[D]$, so, because there exist at most two geodesics between two vertices at distance two apart (as a consequence of Claim \ref{claim:atmosttwogeod}), it follows that, up to permuting $B$ and $C$, one has $[B]= [A \cdot P \cdot (U+ \epsilon(v))]$ and $[C]=[A \cdot P \cdot (\epsilon(u)+V)]$.

\medskip \noindent
Now, we claim that the cases $(ii)$, $(iii)$ and $(v)$ cannot occur, which concludes the proof of our lemma. More precisely, we want to show that, in theses cases, there exists a unique geodesic between $[A]$ and $[D]$. 

\medskip \noindent
Case $(i)$. Fix a geodesic from $[A]$ to $[D]=[A \cdot S]$. According to Lemma \ref{lem:Xgeod}, $A^{-1}D=S$ decomposes as an absolutely reduced concatenation $S_1 \circ S_2$ such that $\# S_1= \# S_2 = 1$ and such that $[A \cdot S_1]$ is the unique interior vertex of our geodesic. If $S_1$ contained a transistor, it would follow that the wires connected to the top side of the unique transistor of $S= S_1 \circ S_2$ are labelled by letters of $\Sigma(\mathcal{G})$ with trivial second coordinates, since the wires of $S_1$ are labelled by such letters, but this is false in our case. Therefore, $S_1$ does not contain any transistor. As a consequence, we can write $S_1= R \circ S_1'$ for some permutation diagram $R$ and some linear diagram $S_1'$. Because $A$ is fixed, the class $[A \cdot S_1]$ is uniquely determined by the class $[S_1]=[R \circ S_1']$, which is in turn uniquely determined by $\mathrm{top}^-(R)=\mathrm{top}^-(S)$ and the place along the top of the frame of $S$ where the unique wire of $S_1$ labelled by a letter of $\Sigma(\mathcal{G})$ with a non-trivial second coordinate is connected. But all these data can be read from $S$, which is fixed. Therefore, there is a unique possibility for the choice of $[A \cdot S_1]$, which implies that there exists a unique geodesic in $X$ from $[A]$ to $[D]$.

\medskip \noindent
Case $(iii)$. We recover case $(i)$ when $[A]=[D \cdot S^{-1}]$ and $[D]$ are switched and when $S$ is replaced with $S^{-1}$. Therefore, also in this case, there exists a unique geodesic in $X$ from $[A]$ to $[D]$.

\medskip \noindent
Case $(v)$. Fix a geodesic from $[A]$ to $[D]=[A \cdot S]$. According to Lemma \ref{lem:Xgeod}, $A^{-1}D=S$ decomposes as an absolutely reduced concatenation $S_1 \circ S_2$ such that $\# S_1= \# S_2 = 1$ and such that $[A \cdot S_1]$ is the unique interior vertex of our geodesic. Notice that, since $S$ does not contain wires labelled by letters of $\Sigma(\mathcal{G})$ with non-trivial coordinates, necessarily $S_1$ contains a transistor of $S$. If $S_1$ contained $T_2$, it would follow that all the wires of $S$ connected to the top side of $T_2$ are connected to top side of the frame of $S$, but this is not true in our case since $T_2 \prec T_1$. Therefore, $S_1$ contains $T_1$. Because $A$ is fixed, the class $[A \cdot S_1]$ is uniquely determined by the class $[S_1]$, which is in turn uniquely determined by $\mathrm{top}^-(S_1)= \mathrm{top}^-(S)$ and the places along the top side of the frame of $S$ at which the wires of $S$ which are connected to the top side of $T_1$ are connected. But all these data can be read from $S$, which is fixed. Therefore, there is a unique possibility for the choice of $[A \cdot S_1]$, which implies that there exists a unique geodesic in $X$ from $[A]$ to $[D]$.
\end{proof}

\noindent
Now we are ready to describe linear hyperplanes of $X$. 

\begin{lemma}\label{lem:linearhyp}
Let $J$ be a linear hyperplane. Fix a vertex $v \in N(J)$. There exist a diagram $\Delta$, a word $m \in \Sigma^+$, and a letter $\ell \in \Sigma$ satisfying $[\Delta]=v$ and $\mathrm{bot}^-(\Delta)= m \ell$ such that an edge of $X$ is dual to $J$ if and only if its endpoints have the form
$$[\Delta \cdot (U+ \epsilon(\ell,g))] \ \text{and} \ [\Delta \cdot ( U + \epsilon(\ell,h))]$$
for some distinct $g,h \in G_{\ell}$ and some $(m, \ast)$-diagram $U$. 
\end{lemma}

\begin{proof}
We choose $\Delta$ so that the pin which is dual to $J$ and which contains $v$ has the form
$$\{ [ \Delta \cdot (\epsilon(m) + \epsilon(\ell,g))] \mid g \in G_{\ell} \}$$
for some diagram $\Delta$, some word $m \in \Sigma^+$ and some letter $\ell \in \Sigma$ satisfying $\mathrm{bot}^-(\Delta)=m \ell$. We want to prove that, if $e$ is an edge of $X$ dual to $J$, then its endpoints have the form
$$[\Delta \cdot (U+ \epsilon(\ell,g))] \ \text{and} \ [\Delta \cdot ( U + \epsilon(\ell,h))]$$
for some distinct $g,h \in G_{\ell}$ and some $(m, \ast)$-diagram $U$. According to \cite[Claim 10.29]{Qm}, there exists a sequence $f_0, \ldots, f_{n-1}, f_n=e$ of edges of $X$ such that $f_0$ belongs to our pin and, for every $1 \leq i \leq n-1$, the edges $f_i$ and $f_{i+1}$ are parallel in some square of $X$. We argue by induction over $n$. If $n=1$, there is nothing to prove. So suppose that $n \geq 2$. By applying our induction hypothesis, we know that the endpoints of $f_{n-1}$  have the form
$$[A]=[\Delta \cdot (U+ \epsilon(\ell,g))] \ \text{and} \ [B] = [\Delta \cdot ( U + \epsilon(\ell,h))]$$
for some distinct $g,h \in G_{\ell}$ and some $(m, \ast)$-diagram $U$. Next, since $[A]$, $[B]$ and the endpoints of $e$ define a square in $X$, we deduce from Lemma \ref{lem:square} that there exist words $x,y \in \Sigma^+$, a permutation $(m \ell, xy)$-diagram $R$, a unitary $(x, \ast)$-diagram $E$ and a unitary $(y,\ast)$-diagram $F$ such that $[B]=[A \cdot R \cdot (E+\epsilon(y))]$ and such that the endpoints of $e$ have the form $[A \cdot R \cdot (E+F)]$ and $[A \cdot R \cdot ( \epsilon(x)+ Y )]$. (There is another possible case, where $[B]=[A \cdot R ( \epsilon(x)+F)]$ and where the endpoints of $e$ have the form $[A \cdot R \cdot ( E+F)]$ and $[A \cdot R \cdot (E+ \epsilon(y))]$; we only consider the first situation, the argument for the second one being the same.) From the equalities
$$\begin{array}{lcl} [\Delta \cdot (U+ \epsilon(\ell,g)) \cdot R \cdot (E+ \epsilon(y))] & = & [B] = [\Delta \cdot (U+ \epsilon(\ell,h))] \\ \\ & = & [\Delta \cdot (U+ \epsilon(\ell,g)) \cdot (\epsilon(u)+ \epsilon(\ell,g^{-1}h)) ] \end{array}$$
it follows that $[R \cdot (E + \epsilon(y))] = [\epsilon(u)+ \epsilon(\ell,g^{-1}h)]$. As a consequence, there exist words $p,q \in \Sigma^+$ such that $E= \epsilon(p)+ \epsilon(\ell,g^{-1}h)+ \epsilon(q)$ and such that the letter $\ell$ in the word $\mathrm{bot}^-(R)=xy=p \ell q y$ labels the rightmost wire which is connected to the top side of the frame of $R$. We claim that there exist two permutation diagrams $R'$ and $R''$ such that
$$R \cdot (\epsilon(p \ell q)+F) = \left( R' \cdot ( \epsilon(pq)+F) + \epsilon(\ell)) \right) \cdot R''.$$
This assertion is justified by the fact that the wire labelled by $\ell$ in the concatenation $R \circ (\epsilon(p \ell q)+F)$ is not connected to the top side of the frame of $F$ (see Figure \ref{figure1}). Next, notice that $R \cdot (E+F)= R \cdot ( \epsilon(p)+ \epsilon(\ell,g^{-1}h) + \epsilon(q)+F)$ is obtained from $R \cdot (\epsilon (p \ell q)+F)$ just by right-multiplying by $g^{-1}h$ the second coordinate of the letter of $\Sigma(\mathcal{G})$ labelling the wire corresponding to $\ell$. Therefore, the previous equality automatically gives
$$R \cdot (E+F) = \left( R' \cdot ( \epsilon(pq)+F) + \epsilon(\ell , g^{-1}h)) \right) \cdot R''.$$
Finally, we deduce that the endpoints of $e$ are
$$[A \cdot R \cdot (\epsilon(x)+F)] = \left[\Delta \cdot \left( R' \cdot (\epsilon(pq)+F) + \epsilon(\ell) \right) \right]$$
and
$$[A \cdot R \cdot (E+F)]= \left[ \Delta \cdot \left( R' \cdot (\epsilon(pq)+F) + \epsilon( \ell, g^{-1}h) \right) \right].$$
The desired conclusion follows by noticing that $R' \cdot \epsilon(pq+F)$ is a $(m,\ast)$-diagram. 
\begin{figure}
\begin{center}
\includegraphics[trim={0 11cm 30cm 0},clip,scale=0.4]{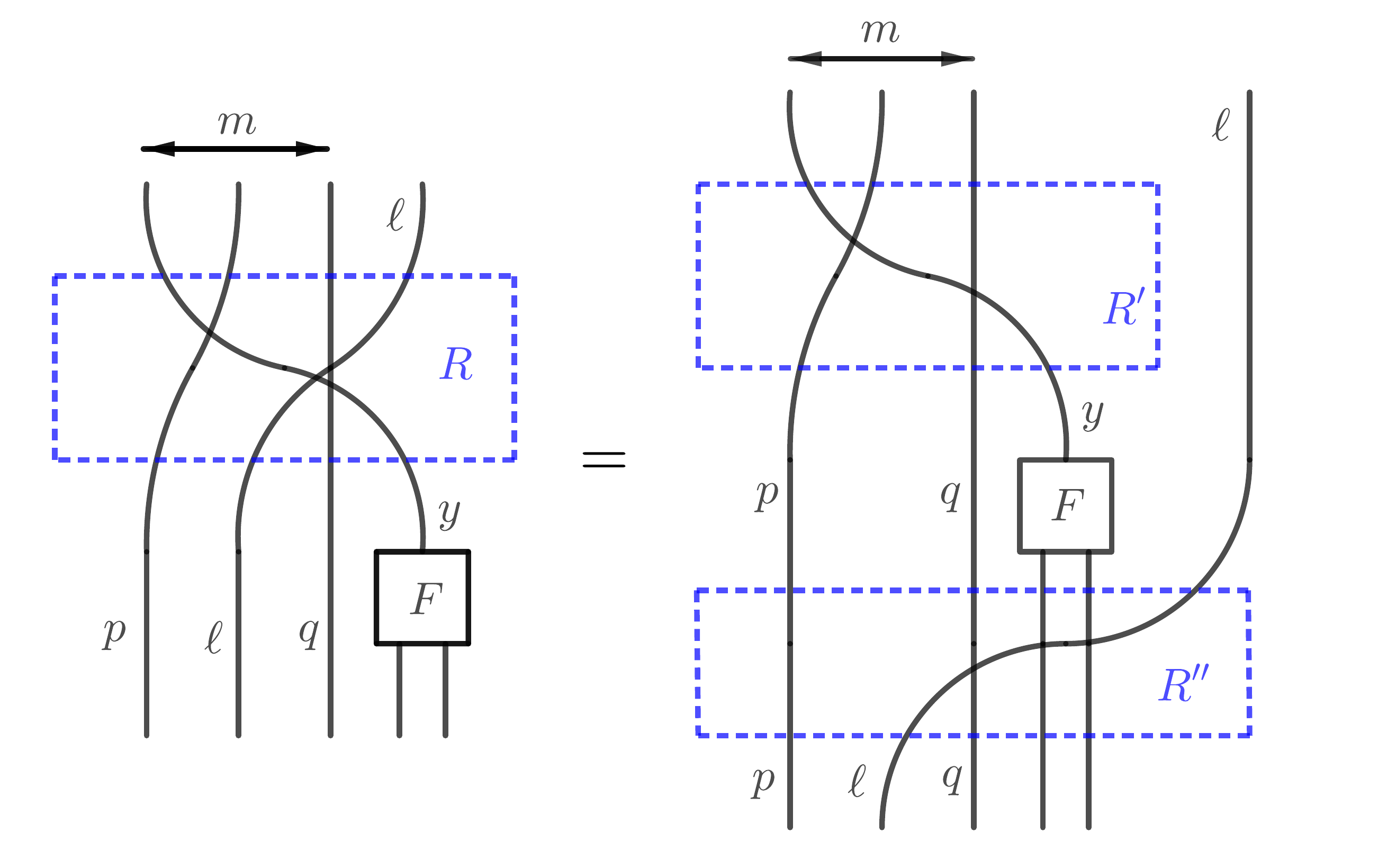}
\caption{}
\label{figure1}
\end{center}
\end{figure}

\medskip \noindent
Conversely, we want to prove that an edge $e$ of $X$ whose endpoints have the form 
$$[\Delta \cdot (U+ \epsilon(\ell,g))] \ \text{and} \ [\Delta \cdot ( U + \epsilon(\ell,h))],$$
for some distinct $g,h \in G_{\ell}$ and some $(m, \ast)$-diagram $U$, must be dual to our linear hyperplane $J$. Decompose $U$ as a concatenation $U_1 \circ \cdots \circ U_n$ such that $\# U_i=1$ for every $1 \leq i \leq n$. For every $1 \leq i \leq n$, let $e_i$ denote the edge linking the vertices
$$[\Delta \cdot ( U_1 \circ \cdots \circ U_i+ \epsilon( \ell, g))] \ \text{and} \ [\Delta \cdot ( U_1 \circ \cdots \circ U_i+ \epsilon( \ell, h))];$$
also, let $e_0$ denote the edge linking the vertices
$$[\Delta \cdot (\epsilon(m) + \epsilon( \ell, g))] \ \text{and} \ [\Delta \cdot ( \epsilon(m) + \epsilon( \ell, h))].$$
Notice that $e_0$ is contained into our pin, that $e_n=e$, and that, for every $0 \leq i \leq n-1$, the edges $e_i$ and $e_{i+1}$ are parallel in some square of $X$. Consequently, $e$ is indeed dual to $J$. 
\end{proof}

\begin{proof}[Proof of Proposition \ref{prop:rotative}.]
Let $J$ be a linear hyperplane of $X$. Fix a vertex $v \in N(J)$. Let $\Delta, m, \ell$ be respectively the diagram, the word and the letter provided by Lemma \ref{lem:linearhyp}. A consequence of this lemma is that the cliques dual to $J$ are precisely the
$$C(U)= \{ [ \Delta \cdot (U+ \epsilon(\ell, g))] \mid g \in G_{\ell} \}$$
where $U$ runs over the braided $(m,\ast)$-diagrams. Fix some braided $(m,\ast)$-diagram, and set $\mathrm{bot}^-(U)=z$. Let $D$ be an element of the stabiliser of $C(U)$. This means that, for every $g \in G_{\ell}$, there exist a permutation $(z \ell,z \ell)$-diagram $P$ and an element $h \in G_{\ell}$ such that
$$D \cdot \Delta \cdot (U+ \epsilon( \ell, g)) = \Delta ( U \cdot \epsilon (\ell,h) ) \cdot P.$$
In particular, this condition implies that there exist some permutation $(z \ell, z \ell)$-diagram $Q$ and two elements $d_1,d_2 \in G_{\ell}$ such that
$$D = \Delta \cdot (U+ \epsilon( \ell,d_1)) \cdot Q \cdot ( U^{-1}+ \epsilon( \ell, d_2)) \cdot \Delta^{-1}.$$
Thus, our condition becomes: for every $g \in G_{\ell}$, there exist a permutation $(z \ell,z \ell)$-diagram $P$ and an element $h \in G_{\ell}$ such that
$$(U+ \epsilon( \ell, d_1)) \cdot Q \cdot ( \epsilon(z)+ \epsilon( \ell, d_2g)) = (U+ \epsilon( \ell, h)) \cdot P.$$
Another equivalent formulation is: for every $g \in G_{\ell}$, there exists some $h \in G_{\ell}$ such that $(\epsilon(z)+ \epsilon(\ell,h)) \cdot Q \cdot (\epsilon(z)+ \epsilon(\ell, g))$ is a permutation $(z \ell, z \ell)$-diagram. We claim that this condition implies that $Q$ splits as a sum $Q' + \epsilon(\ell)$, where $Q'$ is permutation $(z,z)$-diagram. Indeed, the diagram $(\epsilon(z)+ \epsilon(\ell,h)) \cdot Q \cdot (\epsilon(z)+ \epsilon(\ell, g))$ is obtained from the permutation diagram $Q$ first by replacing the label of the rightmost wire connected to the top side of the frame with $(\ell,h)$, and next by multiplying by $h$ the second coordinate of the letter of $\Sigma(\mathcal{G})$ labelling the rightmost wire connected to the bottom side of the frame. The only possibility to get a permutation diagram is that $h=g^{-1}$ and that the rightmost wire connected to the top side of the frame is the same as the rightmost wire connected to the bottom side of the frame. The latter assertion precisely means that $Q$ decomposes as a sum $Q' + \epsilon(\ell)$ for some permutation $(z,z)$-diagram. 

\medskip \noindent
Thus, we have proved that the stabiliser of $C(U)$ is contained into 
$$\{ \Delta \cdot ( UQU^{-1}+ \epsilon(\ell,g)) \cdot \Delta^{-1} \mid g \in G_{\ell}, \ \text{$Q$ permutation $(z,z)$-diagram} \}.$$
The reverse inclusion is clear, so the previous set turns out to be equal to the stabiliser of $C(U)$. Consequently:

\begin{fact}\label{fact:rotativestab}
An element of the rotative-stabiliser
$$\mathrm{stab}_{\circlearrowleft} (J) = \bigcap\limits_{\text{$U$ $(m,\ast)$-diagram}} \mathrm{stab}(C(U))$$
decomposes as $\Delta \cdot (P+ \epsilon(\ell,g)) \cdot \Delta^{-1}$ for some element $g \in G_{\ell}$ and some $(m,m)$-diagram $P$ satisfying the condition: for every word $z \in \Sigma^+$ and every $(m,z)$-diagram $U$, there exists a permutation $(z,z)$-diagram $Q$ such that $P=UQU^{-1}$. 
\end{fact}

\noindent
This condition can be restated as: for every word $z \in \Sigma^+$ and every $(m,z)$-diagram $U$, $U^{-1}PU$ is a permutation $(z,z)$-diagram. It follows from the condition (+) of our proposition that such a braided diagram must be trivial. Consequently, the rotative-stabiliser $\mathrm{stab}_{\circlearrowleft}(J)$ must be contained into
$$\{ \Delta \cdot ( \epsilon(m)+ \epsilon( \ell,g) ) \cdot \Delta^{-1} \mid g \in G_{\ell} \}.$$
The reverse inclusion is clear, so the previous set turns out to be equal to $\mathrm{stab}_{\circlearrowleft}(J)$. We record this statement for future use:

\begin{fact}\label{fact:RotativeStabWithPlus}
Let $J$ be a hyperplane. Fix a vertex $v \in N(J)$. If $\Delta, m, \ell$ are respectively the diagram, the word and the letter provided by Lemma \ref{lem:linearhyp}, then $$\mathrm{stab}_\circlearrowleft (J)= \{ \Delta \cdot (\epsilon(m)+ \epsilon(\ell,g) ) \cdot \Delta^{-1} \mid g \in G_\ell \}.$$
\end{fact}

\noindent
In order to conclude that this subgroup acts freely and and transitively on the set of fibers of $J$, it is sufficient to show that it acts freely and transitively on the vertices of the clique $C(\epsilon(m))$. This is clear since
$$\Delta \cdot (\epsilon(m)+ \epsilon( \ell,g)) \cdot \Delta^{-1} \cdot [ \Delta \cdot (\epsilon(m)+ \epsilon(\ell,h))] = [\Delta \cdot ( \epsilon(m)+ \epsilon(\ell,gh))]$$
for every $g,h \in G_{\ell}$. 
\end{proof}

\begin{remark}\label{remark:notrotative}
Our condition (+) in the statement of Proposition \ref{prop:rotative} turns out to be necessary, ie., if it is not satisfied then the action $D_b(\mathcal{P}, \mathcal{G},w) \curvearrowright X_b(\mathcal{P}, \mathcal{G},w)$ is not $\mathcal{J}$-transitive: the action $\mathrm{stab}_{\circlearrowleft}(J) \curvearrowright \mathcal{S}(J)$ is transitive for every linear hyperplane $J$, but is not free for some $J$. An example where the condition (+) is not satisfied is the following. Let $\mathcal{P}$ be the semigroup presentation $\langle a,p \mid a=ap \rangle$, $w=p^2a$ our baseword, and $\mathcal{G}=\{ G_a= \mathbb{Z}, G_p= \{ 1 \} \}$ our collection of groups. Let $P$ denote the permutation $(p^2,p^2)$-diagram which switches only the two braids labelled by $p$. By noticing that $P$ is the unique non-trivial braided $(p^2,\ast)$-diagram, it is clear that, for every word $z \in \Sigma^+$ and every braided $(m,z)$-diagram $U$, $U^{-1}PU$ is a permutation $(z,z)$-diagram. Therefore, the condition (+) does not hold.
\end{remark}

\begin{proof}[Proof of Theorem \ref{thm:split}.]
In order to show that Theorem \ref{thm:splittingthmQm}, it remains to show that:

\begin{claim}\label{claim:HypsCommute}
Let $J_1$ and $J_2$ be two transverse linear hyperplanes of $X_b(\mathcal{P}, \mathcal{G},w)$. Every element of $\mathrm{stab}_\circlearrowleft(J_1)$ commutes with every element of $\mathrm{stab}_\circlearrowleft(J_2)$. 
\end{claim}

\noindent
Because $J_1$ and $J_2$ are transverse, there exists a prism $C_1 \times C_2$ such that $C_1 \subset J_1$ and $C_2 \subset J_2$. Let $v$ denote the unique vertex of $C_1 \cap C_2$. As a consequence of Lemmas \ref{lem:square} and \ref{lem:linearhyp} and of Fact \ref{fact:RotativeStabWithPlus}, there exist a diagram $\Delta$, a word $m \in \Sigma^+$ and two letters $\ell_1,\ell_2 \in \Sigma$ satisfying $[\Delta]=v$ and $\mathrm{bot}^-(\Delta)=m \ell_1 \ell_2$ such that 
$$\mathrm{stab}_\circlearrowleft (J_1) = \{ \Delta \cdot (\epsilon(m) + \epsilon(\ell_1,g)+ \epsilon(\ell_2) ) \cdot \Delta^{-1} \mid g \in G_{\ell_1} \}$$
and
$$\mathrm{stab}_\circlearrowleft (J_2) = \{ \Delta \cdot (\epsilon(m \ell_1) + \epsilon(\ell_2,g) ) \cdot \Delta^{-1} \mid g \in G_{\ell_2} \}.$$
Our claim now clearly follows.

\medskip \noindent
Thus, by applying Theorem \ref{thm:splittingthmQm} thanks Proposition \ref{prop:rotative} and Claim \ref{claim:HypsCommute}, it follows that the braided picture product $D_b(\mathcal{P},\mathcal{G},w)$ splits as a semidirect product $\Gamma \rtimes \mathrm{stab}(Y)$ where $\Gamma$ is a graph product of groups which belong to $\mathcal{G}$ and where $Y$ is the maximal subgraph of $X$ containing $\epsilon(w)$ and disjoint from any linear hyperplane. Since an edge is dual to a linear hyperplane if and only if it is labelled by a linear diagram, it follows that $Y$ is the subgraph containing all the braided $(w,\ast)$-diagram whose wires are labelled by letters of $\Sigma(\mathcal{G})$ with trivial second coordinates. Consequently, $\mathrm{stab}(Y)$ is the subgroup of all the braided $(w,w$)-diagram whose wires are labelled by letters of $\Sigma(\mathcal{G})$ with trivial second coordinates, ie., $\mathrm{stab}(Y)$ is the canonical copy of $D_b(\mathcal{P},w)$ in $D_b(\mathcal{P},\mathcal{G},w)$. 
\end{proof}

\subsection{Planar and annular picture products}

\noindent
Let $\mathcal{P}= \langle \Sigma \mid \mathcal{R} \rangle$ be a semigroup presentation and $\mathcal{G}$ a collection of groups indexed by $\Sigma$. As we defined braided diagrams over $(\mathcal{P},\mathcal{G})$ as braided diagrams over $\mathcal{P}(\mathcal{G})$, we may define \emph{planar} (resp. \emph{annular}) \emph{diagrams over $(\mathcal{P}, \mathcal{G})$} as planar (resp. annular) diagrams over $\mathcal{P}(\mathcal{G})$. All the terminology about braided diagrams overs $(\mathcal{P},\mathcal{G})$ introduced in Section \ref{section:diagramsPG} applies to planar and annular diagrams over $(\mathcal{P}, \mathcal{G})$, except for the sum since the sum of two annular diagrams is not necessarily annular. 

\begin{definition}
Let $\mathcal{P}= \langle \Sigma \mid \mathcal{R} \rangle$ be a semigroup presentation, $\mathcal{G}$ a collection of groups indexed by $\Sigma$, and $w \in \Sigma^+$ a baseword. The \emph{planar picture product} $D(\mathcal{P},\mathcal{G},w)$ is the group of all planar $(w,w)$-diagrams over $(\mathcal{P},\mathcal{G})$, modulo dipoles, endowed with the concatenation. Similarly, the \emph{annular product} $D_a(\mathcal{P}, \mathcal{G},w)$ is the group of all annular $(w,w)$-diagrams over $(\mathcal{P}, \mathcal{G})$, modulo dipoles, endowed with the concatenation.
\end{definition}

\noindent
In the same way that $D(\mathcal{P},w) \subset D_a(\mathcal{P},w) \subset D_b(\mathcal{P},w)$, one has
$$D(\mathcal{P}, \mathcal{G},w) \subset D_a(\mathcal{P}, \mathcal{G},w) \subset D_b(\mathcal{P},\mathcal{G},w).$$
It is worth noticing that a prefix of a planar (resp. annular) diagram remains planar (resp. annular), so that it follows from Lemma \ref{lem:Xgeod} that the subgraph $X(\mathcal{P}, \mathcal{G}) \subset X_b(\mathcal{P},\mathcal{G})$ (resp. $X_a(\mathcal{P}, \mathcal{G}) \subset X_b(\mathcal{P}, \mathcal{G})$) generated by planar (resp. annular) diagrams is convex. Moreover, the connected component $X(\mathcal{P}, \mathcal{G},w)$ (resp. $X_a(\mathcal{P}, \mathcal{G},w)$) containing the trivial diagram $\epsilon(w)$ is clearly $D(\mathcal{P}, \mathcal{G},w)$-invariant (resp. $D_a(\mathcal{P}, \mathcal{G},w)$). Consequently, we get an action of any planar (resp. annular) picture product on a quasi-median graph.

\medskip \noindent
Let $\mathcal{J}$ denote the collection of hyperplanes of $X(\mathcal{P}, \mathcal{G},w)$ (resp. $X_a(\mathcal{P},\mathcal{G},w)$) coming from linear hyperplanes of $X_b(\mathcal{P}, \mathcal{G},w)$. Interestingly, by following the proof of Proposition \ref{prop:rotative}, it can be proved that these actions are rotative without any additional assumption.

\begin{prop}
The planar picture product $D(\mathcal{P},\mathcal{G},w)$ (resp. the annular picture product $D_a(\mathcal{P},\mathcal{G},w)$) acts $\mathcal{J}$-rotatively on $X(\mathcal{P}, \mathcal{G},w)$ (resp. $X_a(\mathcal{P},\mathcal{G},w)$).
\end{prop}

\begin{proof}
Let $J \in \mathcal{J}$. For convenience, let $S$ (resp. $S_a$, $S_b$) denote the rotative stabiliser of $J$ in $D(\mathcal{P}, \mathcal{G},w)$ (resp. in $D_a(\mathcal{P}, \mathcal{G},w)$, $D_b(\mathcal{P}, \mathcal{G},w)$). Notice that $S$ (resp. $S_a$) is the subgroup of all the planar (annular) diagrams of $S_b$. Notice that, during the proof of Proposition \ref{prop:rotative}, the condition (+) is only used to show that the braided diagram $P$ in Fact \ref{fact:rotativestab} has to be trivial. Therefore, the condition (+) can be replaced with the following fact to deduce that the actions we consider are rotative.

\begin{fact}
Let $P$ be a permutation diagram and $U$ a linear diagram. If $P+U$ is planar or annular, necessarily $P$ is trivial. 
\end{fact}

\noindent
The statement is clear if $P+U$ is planar, since $P$ must be planar as well and any planar permutation diagram is necessarily trivial. Next, suppose that $P+U$ is annular, and represent this diagram on an annulus so that any two distinct wires are disjoint. Cutting this annulus along a wire of $U$ shows that $P$ must be a planar permutation diagram, and a fortiori has to be trivial. 
\end{proof}

\noindent
Finally, by reproducing the proof of Theorem \ref{thm:split}, one gets:

\begin{thm}\label{thm:splittingpa}
The planar picture product $D(\mathcal{P},\mathcal{G},w)$ (resp. the annular picture product $D_a(\mathcal{P},\mathcal{G},w)$) splits as a semidirect product $\Gamma \rtimes D(\mathcal{P}, w)$ (resp. $\Gamma \rtimes D_a(\mathcal{P},w)$) where $\Gamma$ is a graph product of groups which belong to $\mathcal{G}$ and where $D(\mathcal{P},w)$ (resp. $D_a(\mathcal{P},w)$) coincides with the subgroup whose diagrams have their wires labelled by letters of $\Sigma(\mathcal{G})$ with trivial second coordinates. 
\end{thm}

\begin{remark}
It is worth noticing that, according to \cite[Section 10.1]{Qm}, planar picture products coincide with diagram products introduced in \cite{MR1725439}. Moreover, the quasi-median graph $X(\mathcal{P},\mathcal{G},w)$ is the same as the graph which we introduced in \cite{Qm} (with a slightly, but equivalent, language). Also, Theorem \ref{thm:splittingpa} for planar picture products is (a weak version of) \cite[Theorem 10.58]{Qm}. 
\end{remark}

\section{Universal picture products}\label{section:universalpictureproducts}

\subsection{Main embedding}

\noindent
In this section, our goal is to introduce a class of braided picture products and to show that any braided diagram group embeds into one group of this collection. The braided picture products we are interested in are the following:

\begin{definition}
For every cardinal $\alpha$, let $\mathscr{V}_{\alpha}$ denote the braided picture product $D_b(\mathcal{Q}, \mathcal{G},x)$ where $\mathcal{Q}$ is the semigroup presentation $\langle x \mid x=x^2 \rangle$ and where $\mathcal{G}$ contains a single free group of rank $\alpha$. 
\end{definition}

\noindent
The next theorem shows that any braided diagram group embeds into $\mathscr{V}_\alpha$ for some $\alpha$. Moreover, when braided diagram groups and the $\mathscr{V}_\alpha$'s are endowed with natural metrics, our embedding turns out to be biLipschitz. These metrics come from two length functions, the first one defined for braided diagram groups and the second one for the $\mathscr{V}_\alpha$'s specifically. For convenience, we refer to both of them as \emph{diagram lengths}. In a braided diagram group, the \emph{diagram length} of a reduced diagram corresponds to the number of transistors it contains. And in $\mathscr{V}_\alpha$, once a free basis of the group of $\mathcal{G}$ is fixed, we define the \emph{diagram length} of a reduced braided diagram of $\mathscr{V}_{\alpha}$ as the sum of its number of transistors with the sum of all the lengths of the second coordinates of the letters of $\Sigma(\mathcal{G})$ labelling its wires.

\begin{thm}\label{thm:universal}
Let $\mathcal{P} = \langle \Sigma \mid \mathcal{R} \rangle$ be a semigroup presentation and $w \in \Sigma^+$ a baseword. If $\kappa$ denotes the cardinality of $\mathcal{R}$, then the braided diagram group $D_b(\mathcal{P},w)$ embeds into $\mathscr{V}_{\alpha}$ for every $\alpha \geq \kappa$. Moreover, if $\mathcal{R}$ is finite, then the embedding is biLipschitz with respect to diagram lengths. 
\end{thm}

\begin{proof}
Recall that a transistor $T$ in some braided diagram over $\mathcal{P}$ is labelled by a relation $u=v$ such that either $u=v$ or $v=u$ belongs to $\mathcal{R}$ (more precisely, $u$ and $v$ are respectively the top and bottom labels of the transistor). In the former case, we will say that $T$ is labelled by $u=v$, and in the latter case, that $T$ is labelled by $(v=u)^{-1}$. Therefore, transistors are labelled by $\mathcal{R} \sqcup \mathcal{R}^{-1}$. Moreover, since $\alpha \geq \kappa$, we can consider $\mathcal{R}$ as a subset of a free basis of the free group of $\mathcal{G}= \{ \mathbb{F}_{\alpha} \}$, so that the transistors of any braided diagram over $\mathcal{P}$ will be labelled by elements of $\mathbb{F}_{\alpha}$.

\medskip \noindent
Let $\{ \Gamma_n \mid n \geq 0 \}$ be the collection of braided diagrams over $\mathcal{Q}$ constructed by induction in the following way:
\begin{itemize}
	\item set $\Gamma_0= \epsilon(x)$;
	\item construct $\Gamma_{i+1}$ from $\Gamma_i$ by adding a positive transistor $T$ and by gluing the top endpoint of the wire which connected to the top side of $T$ to the bottom endpoint of the leftmost wire of $\Gamma_i$ which is connected to the bottom side of the frame.
\end{itemize}
\begin{figure}
\begin{center}
\includegraphics[trim={0 17cm 50cm 0},clip,scale=0.4]{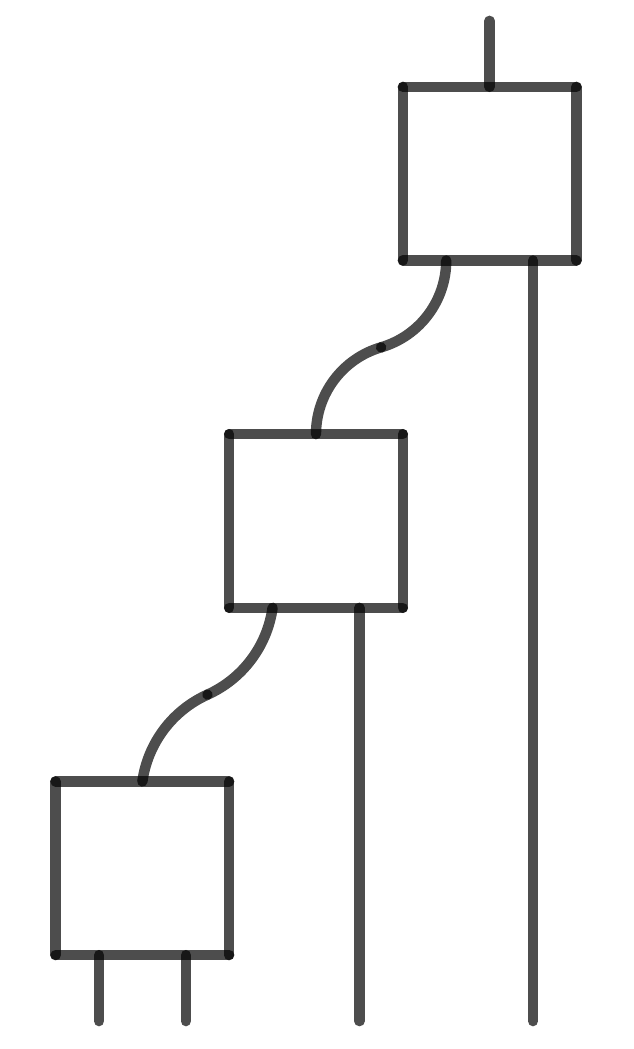}
\caption{The braided diagram $\Gamma_3$.}
\label{figure13}
\end{center}
\end{figure}
For instance, $\Gamma_3$ is illustrated by Figure \ref{figure13}. Now, let $D_b(\mathcal{P})$ denote the groupoid of braided diagrams over $\mathcal{P}$, modulo dipoles, endowed with concatenation; and $D_b(\mathcal{Q}, \mathcal{G})$ the groupoid of braided diagrams over $(\mathcal{Q}, \mathcal{G})$, modulo dipoles, endowed with concatenation. We want to define an injective morphism of groupoids
$$\Psi : D_b(\mathcal{P}) \to D_b(\mathcal{Q},\mathcal{G}).$$
Because $D_b(\mathcal{P})$ is generated by the permutation and transistor diagrams over $\mathcal{P}$, it is sufficient to define their images under $\Psi$. If $P$ is a permutation diagram, let $\Psi(P)$ be the braided diagram over $\mathcal{Q}$ obtained from $P$ by relabelling all the wires of $P$ by $x$. If $T$ is a transistor diagram, decompose it as a sum $\epsilon(a)+ T'+ \epsilon(b)$, where $T'$ is a transistor diagram all of whose wires are connected to the transistor of $T'$. Notice that this decomposition is unique. Let $t$ and $m$ denote respectively the length of the word $\mathrm{top}(T')$ and $\mathrm{bot}(T')$, and let $R^{\pm 1} \in \mathcal{R} \sqcup \mathcal{R}^{-1}$ be the ordered relation corresponding to $T'$ (which can be thought of as an element of $\mathbb{F}_\alpha$ as explained above). Now define 
$$\Psi(T)= \epsilon \left( x^{\mathrm{length}(a)} \right) + \Gamma_t^{-1} \cdot \epsilon \left( x,R^{\pm 1} \right) \cdot \Gamma_m + \epsilon \left( x^{\mathrm{length}(b)} \right)$$
where $\epsilon(x,R^{\pm 1})$ is a single wire labelled by the letter $(x,R^{\pm 1}) \in \Sigma(\mathcal{G})$. In other words, if $\Delta$ is a braided diagram over $\mathcal{P}$, the diagram $\Psi(\Delta)$ is obtained from $\Delta$ by relabelling all its wires by $x$ and by replacing each of its transistors with a block 
$$\epsilon(x^n) + \Gamma_p^{-1} \cdot \epsilon \left( x,R^{\pm 1} \right) \cdot \Gamma_q + \epsilon(x^m)$$
for some integers $n,m,p,q \geq 0$ and some relation $R \in \mathcal{R}$. Such a block will be called a \emph{$(p,q)$-block}, or a \emph{$(\ast,q)$-block} (resp. a \emph{$(p, \ast)$-block}) if we do not want to mention $p$ (resp. $q$) explicitly. See Figure \ref{figure14} for an example with respect to the semigroup presentation $\langle a,b,p \mid ap^2=a, pb=bp \rangle$. As a consequence, the diagram $\Psi(\Delta)$ does not depend on the decomposition of $\Delta$ as a reduced concatenation of permutation and transistor diagrams we chose to compute $\Psi(\Delta)$. We also have to verify that, if $\Delta_1$ and $\Delta_2$ are equal modulo dipoles, then so are $\Psi(\Delta_1)$ and $\Psi(\Delta_2)$. For that purpose, it is sufficient to show that replacing the transistors of a dipole by blocks as above produces a diagram which is trivial modulo dipoles. But a dipole can be written as 
$$(\epsilon(a) + T + \epsilon(b)) \cdot (\epsilon(a)+T^{-1}+ \epsilon(b)),$$
so that its image under $\Psi$ is
$$\left((\epsilon ( x^{n} ) + \Gamma_p^{-1} \cdot \epsilon(x,R) \cdot \Gamma_q + \epsilon(x^m) \right) \cdot \left((\epsilon ( x^{n} ) + \Gamma_p^{-1} \cdot \epsilon(x,R^{-1}) \cdot \Gamma_q + \epsilon(x^m) \right),$$
which is also equal to
$$\epsilon(x^n) + \Gamma_p^{-1} \cdot \epsilon(x, RR^{-1}) \cdot \Gamma_q + \epsilon(x^m),$$
and which is of course trivial modulo dipoles. 
\begin{figure}
\begin{center}
\includegraphics[trim={0 10cm 30cm 0},clip,scale=0.4]{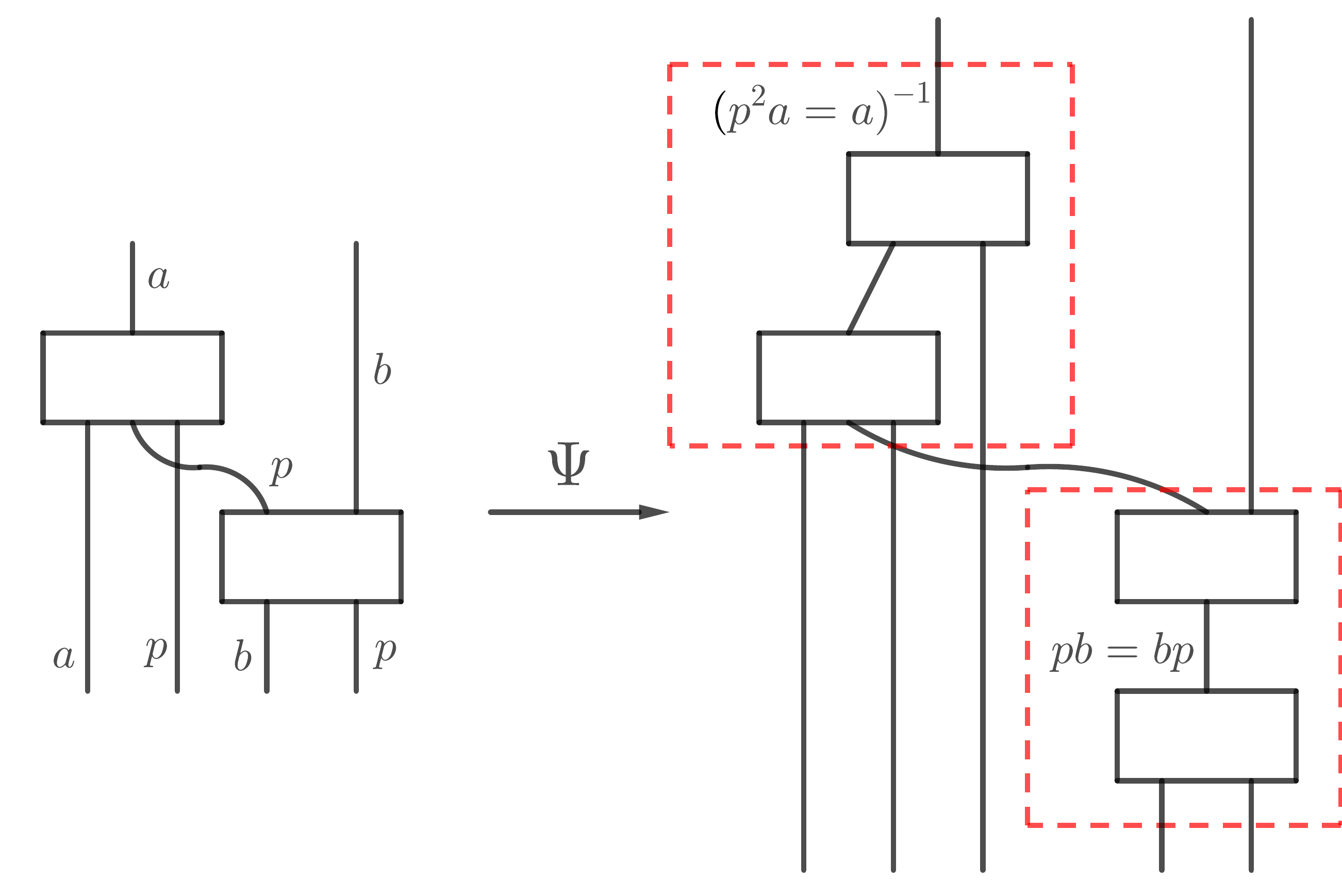}
\caption{The image under $\Psi$ of a braided diagram. Blocks are in red.}
\label{figure14}
\end{center}
\end{figure}

\medskip \noindent
Now, our goal is to show that $\Psi$ is injective.

\medskip \noindent
Let $\Delta$ be a braided diagram over $\mathcal{P}$ which contains at least one transistor. So, as described above, $\Psi(\Delta)$ is obtained from $\Delta$ by replacing its transistors with blocks. If $B_1$ and $B_2$ are two blocks of $\Psi(\Delta)$, we write $B_1 \prec B_2$ to mean that there exists a wire connected to the top side of $B_1$ which is also connected to the bottom side of $B_2$. Moreover, we write $B_1 \prec_o B_2$ if all the wires connected to the top side of $B_1$ are connected to the bottom side of $B_2$ in the same left-to-right order; the relation $\prec_o$ is defined similarly on the transistors of $\Delta$. Notice that there exists a bijection from the set of transistors of $\Delta$ to the set of blocks of $\Psi(\Delta)$ which respects the relations $\prec$ and~$\prec_o$. 

\medskip \noindent
Let $\Psi_0, \ldots, \Psi_m$ be a sequence of diagrams such that 
\begin{itemize}
	\item $\Psi_0=\Psi(\Delta)$;
	\item $\Psi_{i+1}$ is obtained from $\Psi_i$ by reducing a dipole of transistors for every $1 \leq i \leq m-1$;
	\item $\Psi_m$ does not contain dipoles of transistors.
\end{itemize}
We claim that each $\Psi_i$ can be obtained from $\Delta$ by replacing each of its transistors by a block 
$$\epsilon(x^n) + \Gamma_p^{-1} \cdot \epsilon(x,R^{\pm 1}) \cdot \Gamma_q \cdot \epsilon(x^m)$$
for some integers $n,m,p,q \geq 0$ and some relation $R \in \mathcal{R}$, such that there exists a bijection from the set of transistors of $\Delta$ to the set of blocks of $\Psi_i$ which respects the relations $\prec$ and $\prec_o$. We argue by induction on $i$. We already know that our assertion holds for $i=0$. Fix some $0 \leq i \leq n-1$, suppose that our assertion holds for $\Psi_i$, and consider $\Psi_{i+1}$. Because the blocks of $\Psi_i$ are reduced, $\Psi_{i+1}$ is obtained from $\Psi_i$ by reducing a dipole defined by two transistors which belong to two different blocks $B_1$ and $B_2$. Of course, $B_1$ and $B_2$ are $\prec$-related, say $B_1 \prec B_2$. Notice that $\prec$ defines a total order on the transistors of each block, so our dipole must be defined by the $\prec$-minimal transistor of $B_2$ and the $\prec$-maximal transistor of $B_1$. The situation is illustrated by Figure \ref{figure2}. Clearly, the blocks of $\Psi_i$ define blocks of $\Psi_{i+1}$ so that, by applying our induction hypothesis to $\Psi_i$, we get a bijection from the set of transistors of $\Delta$ to the set of blocks of $\Psi_{i+1}$ which respects the relations $\prec$ and $\prec_o$. This concludes the proof of our intermediate claim. 
\begin{figure}
\begin{center}
\includegraphics[trim={0 7cm 25cm 0},clip,scale=0.4]{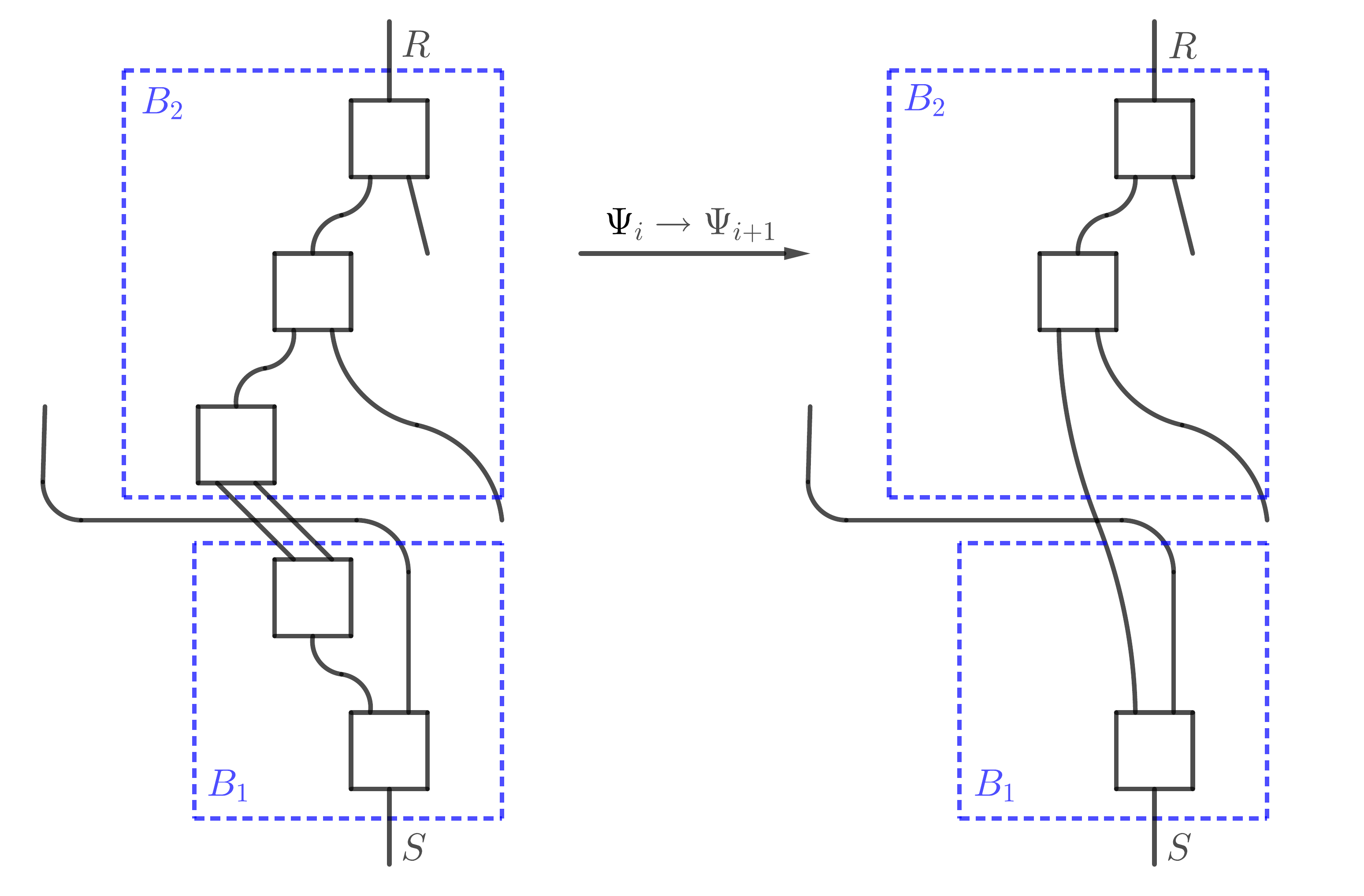}
\caption{}
\label{figure2}
\end{center}
\end{figure}

\medskip \noindent
Now consider $\Psi_m$, and suppose that it is not reduced. Because it does not contain dipoles of transistors, it has to contain a $(\ast,0)$-blocks $B_1$ and a $(0,\ast)$-block $B_2$, labelled by $R$ and $R^{-1}$ for some $R \in \mathcal{R}$, such that $B_2 \prec B_1$. Notice that the block $B_1$ has a unique wire connected to its bottom side, which is also the unique wire connected to the top side of the block $B_2$. A fortiori, $B_2 \prec_o B_1$ and $B_2 \prec B$ implies $B=B_1$ for every block $B$. Therefore, if $T_1,T_2$ denote the transistor of $\Delta$ associated tot $B_1,B_2$ respectively, then $T_2 \prec_0 T_1$, $T_2 \prec T$ implies $T=T_1$ for every transistor $T$, and $T_1,T_2$ are labelled by $R,R^{-1}$. This precisely means that $T_1$ and $T_2$ define a dipole in $\Delta$.

\medskip \noindent
Thus, we have proved that, if $\Delta$ is reduced, then $\Psi_m$ must be reduced as well. A fortiori, the diagram length of $\Psi(\Delta)$ is bounded from below by the number of blocks of $\Psi_m$, since each contains a wire labelled by a letter of $\Sigma(\mathcal{G})$ with a non-trivial second coordinate. Since the number of blocks of $\Psi_m$ is equal to the number of transistors of $\Delta$, it follows that 
$$\mathrm{length}(\Psi(\Delta)) \geq \mathrm{length}(\Delta).$$
Consequently, $\Psi(\Delta)$ is non-trivial if $\Delta$ contains at least one transistor. Otherwise, if $\Delta$ is a permutation diagram, then it is clear from the construction of $\Psi$ that $\Psi(\Delta)$ is trivial if and only if $\Delta$ is trivial. This proves that $\Psi$ is injective. Therefore, $\Psi$ naturally induces an injective morphism $D_b(\mathcal{P},w) \to D_b(\mathcal{Q}, \mathcal{G},x) \simeq \mathscr{V}_{\alpha}$, concluding the proof of the first statement of our theorem. 

\medskip \noindent
Now, notice that the diagram length of $\Psi(\Delta)$ is bounded from above by the product of the number of transistors of $\Psi$ with the maximal diagram length of a block. But the diagram length of a block is at most $K+1$ where $K= \max \{ |u| , |v| \mid u=v \in \mathcal{R} \}$, which is finite if $\mathcal{R}$ is finite itself. Therefore,
$$\mathrm{length}(\Delta) \leq \mathrm{length} (\Psi(\Delta)) \leq (K+1) \cdot \mathrm{length}( \Delta),$$
which concludes the proof of the second assertion of our theorem.
\end{proof}

\subsection{Universal planar and annular products}

\noindent
It is worth noticing that, thanks to Theorem \ref{thm:universal}, we can introduce a collection of planar (resp. annular) picture products and show that any planar (resp. annular) diagram group embeds into one of these products. We begin by defining this family of groups. 

\begin{definition}
For every cardinal $\alpha$, let $\mathscr{T}_{\alpha}$ and $\mathscr{F}_{\alpha}$ denote respectively the annular picture product $D_a(\mathcal{Q},\mathcal{G},x)$ and the planar picture product $D(\mathcal{Q}, \mathcal{G},x)$, where $\mathcal{Q}$ is the semigroup presentation $\langle x \mid x=x^2 \rangle$ and where $\mathcal{G}$ contains a single free group of rank $\alpha$. 
\end{definition}

\noindent
Now, the analogue of Theorem \ref{thm:universal} is:

\begin{thm}\label{thm:universalpa}
Let $\mathcal{P}= \langle \Sigma \mid \mathcal{R} \rangle$ be a semigroup presentation and $w \in \Sigma^+$ a baseword. If $\kappa$ denotes the cardinality of $\mathcal{R}$, then the planar diagram group $D(\mathcal{P},w)$ (resp. the annular diagram group $D_a(\mathcal{P},w)$) embeds into $\mathscr{F}_{\alpha}$ (resp. $\mathscr{T}_{\alpha}$) for every $\alpha \geq \kappa$. Moreover, if $\mathcal{R}$ is finite, the embedding is biLipschitz with respect to the diagram lengths. 
\end{thm}

\begin{proof}
Let $\Psi : D_b(\mathcal{P},w) \hookrightarrow \mathscr{V}_{\alpha}$ denote the embedding constructed in the proof of Theorem \ref{thm:universal}. Clearly, $\Psi$ sends planar diagrams to planar diagrams (resp. annular diagrams to annular diagrams). Therefore, $\Psi$ embeds $D(\mathcal{P},w)$ (resp. $D_a(\mathcal{P},w)$ into the subgroups of planar (resp. annular) diagrams of $\mathscr{V}_{\alpha}$, which is precisely $\mathscr{F}_{\alpha}$ (resp. $\mathscr{T}_{\alpha}$). This proves the first assertion of our theorem. The second one follows from Theorem \ref{thm:universal} and from the observation that the embedding $D(\mathcal{P},w) \subset D_b(\mathcal{P},w)$ and $\mathscr{F}_{\alpha} \subset \mathscr{V}_{\alpha}$ (resp. $D_a (\mathcal{P},w) \subset D_b(\mathcal{P},w)$ and $\mathscr{T}_{\alpha} \subset \mathscr{V}_{\alpha}$) are isometric with respect to diagram lengths. 
\end{proof}

\begin{remark}
In the context of planar diagram groups, examples of universal groups were also constructed in \cite{MR1396957} and \cite{MR2193190, MR2193191}
\end{remark}

\section{Applications}\label{section:applications}

\noindent
By combining the constructions detailed in Sections \ref{section:pictureproducts} and \ref{section:universalpictureproducts}, we are now able to state and prove the main result of our paper:

\begin{thm}\label{thm:sequence}
Any braided (resp. planar, annular) diagram group $D$ decomposes as a short exact sequence
$$1 \to R \to D \to S \to 1$$
for some subgroup $R$ of a right-angled Artin group and some subgroup $S$ of Thompson's group $V$ (resp. $F$, $T$). 
\end{thm}

\begin{proof}
Write $D=D_b(\mathcal{P},w)$ for some semigroup presentation $\mathcal{P}= \langle \Sigma \mid \mathcal{R} \rangle$ and some baseword $w \in \Sigma^+$. Let $\kappa$ denote the cardinality of $\mathcal{R}$. According to Theorem \ref{thm:universal}, $D$ embeds into $\mathscr{V}_{\kappa}$. Now, we want to apply Theorem \ref{thm:split} to $\mathscr{V}_{\kappa}$. So we need to verify that the condition (+) holds.

\medskip \noindent
From now on, we work with the semigroup presentation $\mathcal{Q}= \langle x \mid x=x^2 \rangle$, the baseword $x$, and the collection of groups $\mathcal{G}=\{ \mathbb{F}_{\kappa} \}$. Fix some $m \geq 1$. If $P$ is a braided $(x^m,x^m)$-diagram such that, for every $n \geq 1$ and every braided $(x^m,x^n)$-diagram $U$, $U^{-1}PU$ is a permutation $(x^n,x^n)$-diagram, we claim that $P$ must be trivial. First of all, notice that $P=P^{-1}PP$ must be a permutation diagram. If $P$ is non-trivial, there must exist some $1 \leq k \leq m$ such that $k$th wire $b_k$ connected to the bottom side of the frame of $P$ (in the left-to-right ordering) is different from the $k$th wire $t_k$ connected to the top side of the frame. Let $U$ be the transistor $(x^m,x^{m+1})$-diagram whose unique transistor corresponds to the transformation $x \to x^2$ and such that the wire connected to its top side is the $k$th wire connected to the top side of the frame. Notice that, in the concatenation $U^{-1} \circ P \circ U$, the wire connected to the top side of the transistor of $U$ is $b_k$, which is different to the wire $t_k$, coinciding with the wire connected to the bottom side of the transistor of $U^{-1}$. As a consequence, the two transistors of $U^{-1}PU$ do not define a dipole, so that $U^{-1}PU$ cannot be a permutation diagram. This proves that the condition (+) holds. 

\medskip \noindent
Consequently, Theorem \ref{thm:split} applies, so that $\mathscr{V}_{\kappa}$ splits as a semidirect product $\Gamma \rtimes D_b(\mathcal{Q},x)$ where $\Gamma$ is a graph product of free groups, so a right-angled Artin group, and $D_b(\mathcal{Q},x) \simeq V$. The restriction to $D \subset \mathscr{V}_{\kappa}$ of the projection $\mathscr{V}_{\kappa} \to V$ provides the desired short exact sequence. 

\medskip \noindent
In the context of planar and annular diagram groups, it is sufficient to combine Theorems \ref{thm:universalpa} and \ref{thm:splittingpa}.
\end{proof}

\begin{remark}
In the context of planar diagram groups, the previous theorem was proved in \cite{MR2193191}. This implies, for instance, that planar diagram groups are locally indicable since right-angled Artin groups and Thompson's group $F$ are locally indicable themselves. 
\end{remark}

\noindent
A first consequence of Theorem \ref{thm:sequence} provides restrictions on diagrams groups which do not contain non-abelian free subgroups (extending \cite[Theorem 7.5]{MR2193191} from the context of planar diagram groups).

\begin{prop}
A subgroup of a braided (resp. planar, annular) diagram group without non-abelian free subgroups is (free abelian)-by-(a subgroup of $V$ (resp. $F$, $T$)). 
\end{prop}

\begin{proof}
Let $D$ be a braided (resp. planar, annular) diagram group and let $H$ be a subgroup of $D$ which does not contain any non-abelian free subgroup. By applying Theorem \ref{thm:sequence}, it follows that $H$ splits as an exact sequence $1 \to R \to H \to S \to 1$ where $R$ is a subgroup of a right-angled Artin group and where $S$ is a subgroup of Thompson's group $V$ (resp. $F$, $T$). But it follows from \cite{MR634562} that a subgroup of a right-angled Artin groups either is free abelian or contains a non-abelian free subgroup, so that $R$ is necessarily free abelian. Our proposition follows. 
\end{proof}

\noindent
Another interesting consequence of Theorem \ref{thm:sequence} is that simple diagram groups embed into some Thompson's group (extending \cite[Theorem 7.2]{MR2193191} from the context of planar diagram groups). Such a result motivates the study of simple subgroups of Thompson's groups.

\begin{prop}\label{prop:simpleembed}
Every simple subgroup of a braided (resp. planar, annular) diagram group is isomorphic to a subgroup of Thompson's group $V$ (resp. $F$, $T$).
\end{prop}

\begin{proof}
Let $D$ be a braided (resp. planar, annular) diagram group and let $H$ be a simple subgroup of $D$. By applying Theorem \ref{thm:sequence}, it follows that $H$ splits as an exact sequence $1 \to R \to H \to S \to 1$ where $R$ is a subgroup of a right-angled Artin group and where $S$ is a subgroup of Thompson's group $V$ (resp. $F$, $T$). Because $H$ is simple, necessarily $H$ embeds into a right-angled Artin group or into Thompson's group $V$ (resp. $F$, $T$). But a non-trivial simple group cannot embed into a right-angled Artin group, for instance because non-trivial subgroups of right-angled Artin groups maps onto $\mathbb{Z}$, so our proposition follows. 
\end{proof}

\noindent
For instance, the previous proposition can be used to prove that some Thompson-like groups are quite different from braided diagram groups. This is the case for the higher dimensional Thompson's groups $nV$ introduced in \cite{HigherDimThompson}.

\begin{cor}
For every $n \geq 2$, $nV$ does not embed into any braided diagram group.
\end{cor}

\begin{proof}
Proposition \ref{prop:simpleembed} implies that, if $nV$ embeds into a braided diagram group, then it must embed into Thompson's group $V$ since $nV$ is simple according to \cite{HigherDimThompson}. But $nV$ does not embed into $V$, for instance because $nV$ contains the free product $\mathbb{Z}^2 \ast \mathbb{Z}$ (initially proved in \cite{CorwinThesis}, and reproved in \cite{RAAGnV} for $n \geq 5$ and in \cite{KatonV} for $n \geq 2$), which is not a subgroup of $V$ according to \cite[Theorem 1.5]{FreeProductsInV}. 
\end{proof}

\noindent
Now, we will deduce from Theorem \ref{thm:sequence} some criteria to show that a diagram group embeds into some Thompson's group and we will apply them to the examples we mentioned in Section \ref{section:ex}. 

\begin{prop}\label{prop:embed1}
A braided (resp. planar, annular) diagram group all of whose non-trivial normal subgroups contain a non-trivial simple group is isomorphic to a subgroup of $V$ (resp. $F$, $V$). 
\end{prop}

\begin{proof}
By noticing that a non-trivial simple group cannot embed into a right-angled Artin group, for instance because non-trivial subgroups of right-angled Artin groups map onto $\mathbb{Z}$, it follows that in our situation the kernel of short exact sequence provided by Theorem \ref{thm:sequence} must be trivial. Our proposition follows. 
\end{proof}

\begin{cor}
For every $n,r \geq 1$, $F_{n,r}$ (resp. $T_{n,r}$) embeds into $F$ (resp. $T_n$).
\end{cor}

\begin{proof}
According to \cite[Theorem 4.13]{BrownFiniteness}, every non-trivial normal subgroup of $F_{n, r}$ contains the commutator subgroup, which is simple. The conclusion follows from Proposition \ref{prop:embed1} since we saw in Example \ref{ex:ThompsonsVariations} that $F_{n,r}$ is a planar diagram group. Similarly, according to \cite[Theorem 4.15]{BrownFiniteness}, every non-trivial normal subgroup of $T_{n,r}$ contains the second commutator subgroup, which is simple. The conclusion follows from Proposition \ref{prop:embed1} since we saw in Example \ref{ex:ThompsonsVariations} that $T_{n,r}$ is an annular diagram group. 
\end{proof}

\begin{prop}\label{prop:embed2}
A braided diagram all of whose non-trivial normal subgroups have torsion is isomorphic to a subgroup of $V$. 
\end{prop}

\begin{proof}
Because right-angled Artin groups are torsion-free, our proposition follows directly from Theorem \ref{thm:sequence}.
\end{proof}

\noindent
As an application, we are able to recover a result proved in \cite{HigmanBook}.

\begin{cor}
For every $n \geq 2$ and $r \geq 1$, Higman's group $V_{n,r}$ embeds into $V$. 
\end{cor}

\begin{proof}
We saw in Example \ref{ex:ThompsonsVariations} that Higman's group $V_{n,r}$ is a braided diagram group. Also, Higman proved in \cite{HigmanBook} that $V_{n,r}$ is simple if $n$ is even, and contains a simple subgroup $S_{n,r}$ of index two if $n$ is odd. Therefore, a direct application of Proposition~\ref{prop:simpleembed} shows that $V_{n,r}$ embeds into $V$ if $n$ is even. Suppose now that $n$ is odd. If $N$ is a normal subgroup of $V_{n,r}$, then either $N$ is disjoint from $S_{n,r}$, so that it is necessarily finite, or it contains $S_{n,r}$, so that it has to contain a finite-order element. It follows from Proposition~\ref{prop:embed2} that $V_{n,r}$ embeds into $V$. 
\end{proof}

\noindent
In our next application, we extend \cite[Theorem 4]{BMN} where it is proved that $QV_{2,1,0}$ embeds into $V$.

\begin{cor}
For every $n \geq 2$, $r \geq 1$ and $p \geq 0$, $QV_{n,r,p}$ embeds into $V$.
\end{cor}

\begin{proof}
Let $N$ be a non-trivial normal subgroup of $QV_{n,r,p}$. Fix a non-trivial element $g \in N$. Let $\sigma \in QV_{n,r,p}$ be an element whose support is finite. Then, for every vertex $x \in T_{n,r,p}$ outside $\mathrm{supp}(\sigma) \cup g^{-1} \mathrm{supp}(\sigma)$, one has $g^{-1} \sigma^{-1} g \sigma \cdot x = x$. Consequently, the element $g^{-1} \sigma^{-1} g \sigma$ has finite support, and so it defines an element of finite order of $N$. We claim that $\sigma$ can be chosen so that this element is moreover non-trivial. Because $g$ is non-trivial, there must exist a vertex $o \in T_{n,r,p}$ such that $g \cdot o \neq o$. Fix a vertex $q \in T_{n,r,p}$ distinct from $o$. Now let $\sigma \in QV_{n,r,p}$ be an element whose support is finite and such that $\sigma(o)=g^{-1}q$ and $\sigma(q) = o$. Then $g^{-1} \sigma^{-1} g \sigma \cdot o = g^{-1}o \neq o$, showing that $g^{-1} \sigma^{-1} g \sigma$ is non-trivial, desired.


\medskip \noindent
Thus, we have proved that any non-trivial normal subgroup of $QV_{n,r,p}$ contains a non-trivial finite-order element. Our corollary follows from Proposition \ref{prop:embed2}. 
\end{proof}

\noindent
Our following corollary can be shown by reproducing the previous proof word for word. It extends \cite[Proposition 2.6]{RoverThesis} proved for Houghton's groups $H_n$.

\begin{cor}
For every $n \geq 1$ and $p \geq 0$, the generalised Houghton group $H_{n,p}$ embeds into $V$.
\end{cor}

\noindent
Let us conclude our paper with a remark about the distortions of our embeddings. 

\begin{remark}\label{remark:distortion}
Let $D$ be a diagram group defined from a finite semigroup presentation. If one of Propositions \ref{prop:simpleembed}, \ref{prop:embed1} or \ref{prop:embed2} applies, then we get an embedding into some Thompson's group which is quasi-isometric with respect to diagram lengths. Following \cite{MR2271228}, we say that a diagram group satisfies \emph{Property B} if the diagram length is quasi-isometric to a word length (with respect to a finite generating set). According to \cite{BurilloPropB} and \cite{ThasPropB}, Thompson's groups $F$ and $T$ satisfy Property B (with respect to their usual descriptions as diagram groups; see Examples \ref{ex:F} and \ref{ex:T}). Therefore, if $D$ is a planar or annular diagram group satisfying Property B, then our embeddings are quasi-isometric with respect to word lengths. For Thompson's group $V$ (with respect to its usual description as a diagram group, see Example \ref{ex:V}) the situation is more complicated: according to \cite{VandPropB}, it does not satisfy Property B, but the distortion between diagram and word lengths is asymptotically bounded by $n \cdot \log(n)$. Therefore, if $D$ is a braided diagram group satisfying Property B, then our embeddings into $V$ have compression one. 
\end{remark}

\noindent
This observation motivates the following question:

\begin{question}
Do the diagram groups mentioned in Section \ref{section:ex} satisfy Property B?
\end{question}

\appendix

\section{About hyperplanes in quasi-median graphs}

\noindent
This appendix is dedicated to the proof of Theorem \ref{thm:MainQM} used in Section \ref{section:DecompositionTheorem}. Our goal here is to expose a short argument based on published articles, and we refer to the thesis  \cite[Propositions 2.15 and 2.30, Corollary 2.21]{Qm} for a (longer) self-contained argument. 

\begin{thm}\label{thm:GenQM}
Let $X$ be a quasi-median graph. Fix a clique $C \subset X$ and let $J$ denote the hyperplane which contains it. The following assertions hold:
\begin{itemize}
	\item[(i)] If $p : X \to C$ denotes the projection onto $C$, then $\{p^{-1}(x) \mid x \in C\}$ coincides with the collection of the sectors delimited by $J$. Moreover, such a sector is gated.
	\item[(ii)] The neighborhood $N(J)$ of $J$ is a gated subgraph.
	\item[(iii)] The fibers of $J$ are gated subgraphs.
	\item[(iv)] Every geodesic of $X$ crosses $J$ at most once. 
\end{itemize}
\end{thm}

\noindent
We refer to Section \ref{section:QMgeom} for the definition of quasi-median graphs, and to Section \ref{section:DecompositionTheorem} for the definitions of gated subgraphs, hyperplanes, etc. Also, recall from \cite[Theorem 1]{quasimedian} that, in quasi-median graphs, cliques are gated. The first preliminary result needed to prove Theorem \ref{thm:GenQM} is the following. 

\begin{lemma}\label{lem:GatedChepoi}\emph{\cite{ChepoiRussian}}
Let $X$ be a quasi-median graph and $Y \subset X$ a connected induced subgraph. Then $Y$ is gated if and only if it \emph{locally convex} (i.e., if three consecutive vertices of an induced 4-cycle belongs to $Y$ then the entire 4-cycle must be included into $Y$) and if $Y$ \emph{contains its triangles} (i.e., any triangle having of its edges in $Y$ must be included into~$Y$). 
\end{lemma}

\noindent
A straightforward consequence of the previous lemma (which can also be proved directly from the definition) is:

\begin{cor}\label{cor:GatedInGated}
Let $X$ be a quasi-median graph and $Y \subset Z \subset X$ two subgraphs. Assume that $Z$ is gated, so that $Z$ is a quasi-median graph in its own right. If $Y$ is gated in $Z$, then it is gated in $X$. 
\end{cor}

\noindent
The following alternative characterisation of hyperplanes, essentially contained in \cite{QMandTreeLike}, will be fundamental in our argument:

\begin{prop}\label{prop:RefHypProj}
Let $X$ be a quasi-median graph. Fix a clique $C \subset X$ and let $J$ denote its associated hyperplane. An edge $[a,b]$ belongs to $J$ if and only if the projections of $a$ and $b$ onto $C$ are distinct.
\end{prop}

\begin{proof}
Assume that $[a,b] \subset J$. As a consequence of \cite[Corollary 4.5(i)]{QMandTreeLike}, there exists an edge $[p,q] \subset C$ such that $d(a,p)<d(a,q)$ and $d(b,q)<d(b,p)$. Let $a'$ denote the projection of $a$ onto $C$. If $a' \neq p$, then 
$$d(a,q)-1 = d(a,q)-d(a',q)=d(a,a') \leq d(a,p)-1 \leq d(a,q)-2$$
which is impossible. Therefore, $p$ is the projection of $a$ onto $C$. Similarly, $q$ is the projection of $b$ onto $C$. A fortiori, $a$ and $b$ have different projections onto $C$. Conversely, suppose that the projections $a'$ and $b'$ respectively of $a$ and $b$ are different. We have $d(a,a')<d(a,b')$ and $d(b,b')<d(b,a')$. It follows from \cite[Corollary 4.5(i)]{QMandTreeLike} that the edges $[a,b]$ and $[a',b']$ belong to the same hyperplane, namely $J$. 
\end{proof}

\noindent
Before turning to the proof of Theorem \ref{thm:GenQM}, let us mention and prove two last elementary lemmas:

\begin{lemma}\label{lem:ACliques}
Let $X$ be a quasi-median graph, $J$ a hyperplane and $C_1,C_2 \subset J$ two cliques. Either $C_1=C_2$ or $C_1 \cap C_2= \emptyset$.
\end{lemma}

\begin{proof}
Assume that $C_1 \cap C_2 \neq \emptyset$. Fix a vertex $x \in C_1 \cap C_2$. If $y \in C_1$ is a vertex distinct from $x$, then it follows from Proposition \ref{prop:RefHypProj} that $x$ and $y$ have distinct projections onto $C_2$. Of course, $x$ is the projection of $x$ onto $C_2$. If $y'$ denote the projection of $y$ onto $C_2$, then $d(y,y')+1 = d(y,y')+d(y',x)=d(y,x)=1$, hence $d(y,y')=0$. It follows that $y=y' \in C_2$. Thus, we have proved that $C_1 \subset C_2$. One shows similarly that $C_2 \subset C_1$, proving that $C_1=C_2$. 
\end{proof}

\begin{lemma}\label{lem:AppendixSquare}
Let $X$ be a quasi-median graph, $C_1,C_2 \subset X$ two distinct cliques included into the same hyperplane and $a_1 \in C_1$, $a_2 \in C_2$ two adjacent vertices. For every vertex $b_1 \in C_1 \backslash \{a_1\}$, there exists $b_2 \in C_2$ such that $a_1,a_2,b_1,b_2$ span an induced square.
\end{lemma}

\begin{proof}
Let $b_2$ denote the projection of $a_2$ onto $C_2$. Notice that, as a consequence of Lemma \ref{lem:ACliques}, $b_1$ has to be the projection of $a_1$ onto $C_2$. Because $a_1$ and $a_2$ must have distinct projections onto $C_2$ according to to Proposition \ref{prop:RefHypProj}, it follows that $b_1 \neq b_2$. Moreover, 
$$d(a_2,b_2)+1 = d(a_2,b_2)+d(b_2,b_1) = d(a_2,b_1)  \leq d(a_2,a_1)+d(a_1,b_1) = 2$$
hence $d(a_2,b_2)=1$, i.e., $a_2$ and $b_2$ are adjacent. Therefore, $a_1,a_2,b_1,b_2$ span a square. But if this square is not induced, namely if $a_2$ is adjacent to $b_1$ or if $a_1$ is adjacent to $b_2$, then it follows from the fact that the cliques $C_1$ and $C_2$ are gated (according to \cite[Theorem 1]{quasimedian}) that they contain their triangles and finally that $C_1=C_2$, contrary to our assumptions.
\end{proof}

\begin{proof}[Proof of Theorem \ref{thm:GenQM}.]
First, notice that, for every vertex $x \in C$, the subgraph $p^{-1}(x)$ is connected. Indeed, if $y \in p^{-1}(x)$, then $x$ coincides with the unique vertex of $C$ minimising the distance to $y$, so that $x$ has to be the projection of any vertex of any geodesic from $y$ to $x$. In other words, any geodesic from $y$ to $x$ is included into $p^{-1}(x)$, which implies the connectivity of $p^{-1}(x)$.

\medskip \noindent
Next, fix two distinct vertices $x,y \in C$. We claim that $p^{-1}(x)$ and $p^{-1}(y)$ are separated by $J$. So let $a \in p^{-1}(x)$ and $b \in p^{-1}(y)$ be two vertices. Considering a path from $a$ to $b$, there must exist two consecutive vertices $r$ and $s$ such that $p(r)=x$ and $p(s) \neq x$. It follows from Proposition \ref{prop:RefHypProj} that the edge $[r,s]$ belongs to $J$. Thus, we have proved that any path from $a$ to $b$ must intersect $J$, showing that $J$ separates $a$ and $b$, as desired. This concludes the proof of the first assertion of (i).

\medskip \noindent 
Now, let $S$ denote a sector delimited by $J$. According to Lemma \ref{lem:GatedChepoi}, it is sufficient to show that $S$ is locally convex and that it contains its triangles in order to deduce that it is a gated subgraph. So let $T$ be a triangle (resp. a 4-cycle) having one of its edges (resp. three consecutive vertices) in $S$. If $T$ is not included into $S$, then $J$ has to contain an edge of $T \cap Y$, which is not possible as $J$ cannot separate two vertices of $S$ by definition of sector. Thus, (i) is proved.

\medskip \noindent
Now, we turn to the proof of (ii). First, we claim that $N(J)$ is connected. So let $x_1,x_2 \in N(J)$ be two vertices. By definition of $N(J)$, there exist two edges $e_1,e_2 \subset J$ containing $x_1,x_2$ respectively. So there exists a sequence of edges from $e_1$ to $e_2$ so that any two consecutive edges either belong to a common triangle or are opposite edges of a square. The union of all these edges defines a connected subgraph of $N(J)$ containing $x_1$ and $x_2$. Next, we want to show that $N(J)$ contains its triangles. So let $T$ be a triangle having one of its edges in $N(J)$, say $[a,b]$. If $J$ contains an edge of $T$, then $T \subset J \subset N(J)$, so suppose that no edge of $T$ belongs to $J$. As a consequence of Lemma~\ref{lem:AppendixSquare}, there exist two edges $e,f \subset J$ which have respectively $a,b$ as an endpoint and which are opposite edges of a square $C$. See Figure \ref{appendix}. It follows from the triangle condition that $T \cup C$ lives inside a prism, which implies that $T \subset N(J)$. Finally, we claim that $N(J)$ is locally convex. So let $S$ be a 4-cycle having three consecutive vertices, say $a_1,a_2,a_3$, in $N(J)$. If $J$ contains an edge of $S$, then we clearly have $S \subset N(J)$, so assume that $J$ does not contain any edge of $S$. By applying Lemma \ref{lem:AppendixSquare} twice, we find that there exist three edges $e_1,e_2,e_3 \subset J$ having respectively $a_1,a_2,a_3$ as an endpoint such that $e_1,e_2$ and $e_2,e_3$ are opposite sides of two squares $C_1$ and $C_2$ respectively. See Figure \ref{appendix}. Let $b_1,b_2,b_3$ denote the endpoints of $e_1,e_2,e_3$ distinct from $a_1,a_2,a_3$ respectively. Also, let $a \in S$ denote the fourth vertex of $S$. Notice that $a$ and $b_1$ are not adjacent, since otherwise $J$ would contain the edge $[a,a_1]$ of $S$, hence $d(a,b_1)=2$. Similarly, $d(a,b_3)=2$. Next, because $J$ does not cross $S$, the vertices $a$ and $a_2$ have to belong to the same sector delimited by $J$, so that it follows from (i) that $a_2$ is the projection of $a$ onto the clique of $J$ containing $e_2$. We deduce that
$$d(a,b_2)=d(a,a_2)+d(a_2,b_2)=2+1=3.$$
As a consequence, the quadrangle condition applies, and it follows that $S \cup C_1 \cup C_2$ lives inside a subgraph which decomposes as the product of a square and an edge of $J$. We conclude that $S \subset N(J)$. Thus, we have shown that $N(J)$ is connected, that it is locally convex, and that it contains its triangles. Lemma \ref{lem:GatedChepoi} then implies (ii).
\begin{figure}
\begin{center}
\includegraphics[trim={0 16cm 35cm 0},clip,scale=0.45]{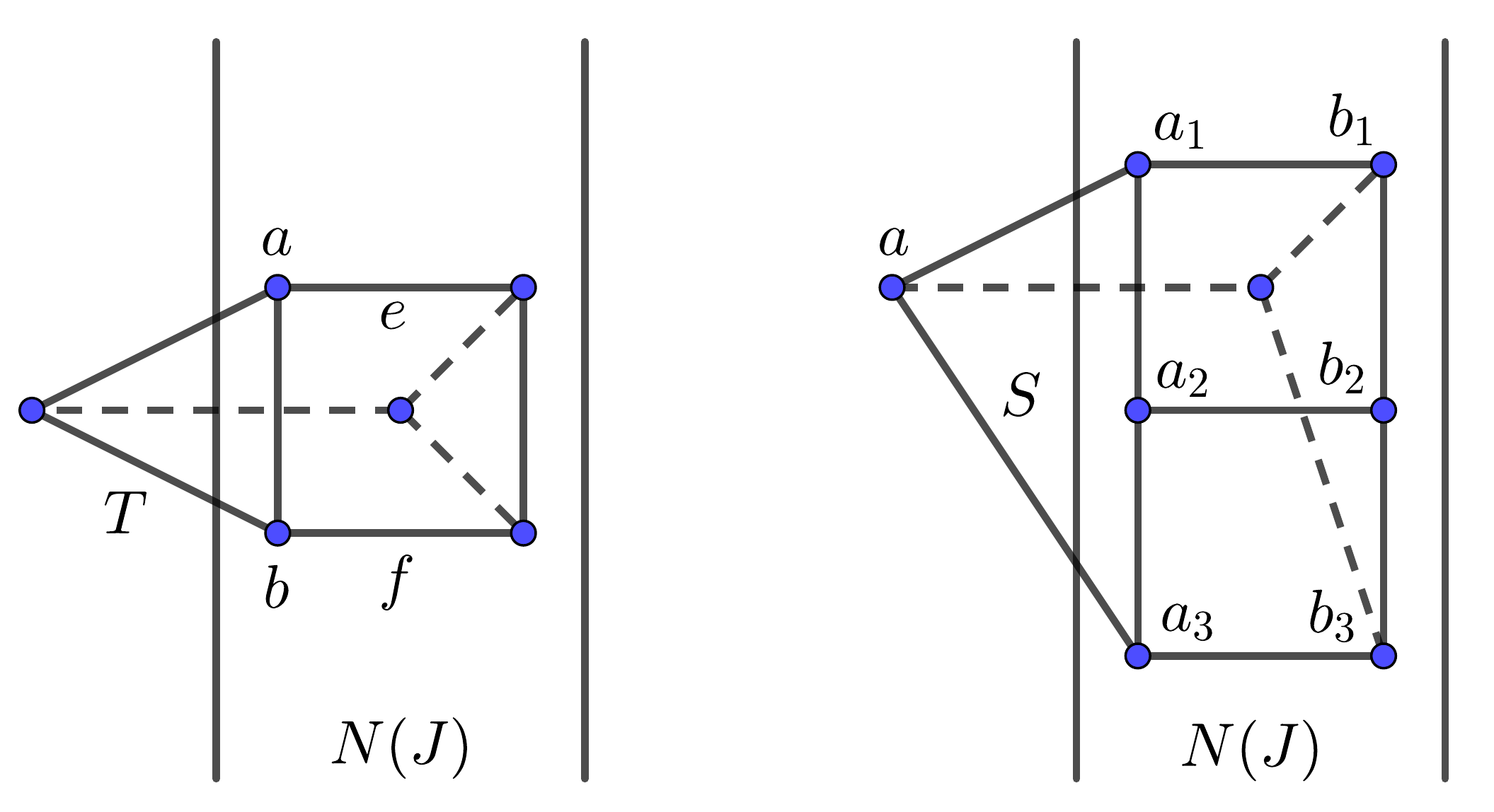}
\caption{Configurations from the proof of Theorem \ref{thm:GenQM}.}
\label{appendix}
\end{center}
\end{figure}

\medskip \noindent
Notice that (iii) is a consequence of (i) and (ii). Indeed, as $N(J)$ is a gated subgraph according to (i), it defines a quasi-median graph in its own right. Then the fibers of $J$ coincide with the sector delimited by $J$ inside $N(J)$. Then (iii) follows from (i), (ii) and Corollary \ref{cor:GatedInGated}.

\medskip \noindent
Now we turn to the proof of (iv). Let $\gamma$ be a path crossing $J$ twice. So there exist two distinct edges $e_1,e_2 \subset \gamma \cap J$. Let $\sigma$ denote the subsegment of $\gamma$ between $e_1$ and $e_2$. Without loss of generality, we suppose that $\sigma \cap J= \emptyset$. If $\sigma$ has length zero, then it follows from Lemma \ref{lem:ACliques} that $e_1$ and $e_2$ belong to a common clique of $J$. So either $e_1$ and $e_2$ define a backtrack, so that $\gamma$ can be shortened by removing $e_1 \cup e_2$, or $e_1$ and $e_2$ span a triangle, so that $\gamma$ can be shortened by replacing $e_1 \cup e_2$ with the third side of the triangle. Now suppose that $\sigma$ has positive length, and let $e$ denote the first edge of $\sigma$. If $e$ does not belong to $N(J)$, then $\gamma$ cannot be a geodesic since $N(J)$ is convex by (ii) and so it can be shortened. If $e$ belongs to $N(J)$, then there must exist a clique of $J$ containing the endpoint of $e$ which does not belong to $e_1$, and we deduce from Lemma \ref{lem:AppendixSquare} that $e_1$ and $e$ are two consecutive edges of a square $C$. Let $\gamma'$ denote the path obtained from $\gamma$ by replacing $e_1 \cup e$ with the path (of length two) passing through the edges of $C$ opposite to $e_1$ and $e$. Then $\gamma'$ contains two edges of $J$ such that the subsegment between them has length $< \mathrm{length}(\sigma)$. By induction, it follows that $\gamma'$ can be shortened. Therefore, we have shown that any path crossing twice $J$ can be shortened, concluding the proof of (iv). 
\end{proof}

\addcontentsline{toc}{section}{References}

\bibliographystyle{alpha}
{\footnotesize\bibliography{ThompsonFromQM}}

\end{document}